\documentclass{amsart}
\usepackage[utf8]{inputenc}
\usepackage[T1]{fontenc}
\usepackage{amsmath, amssymb, amsthm}
\usepackage{bm}
\usepackage{mathrsfs}
\usepackage{epsfig,graphics}
\usepackage[margin=1in]{geometry}
\usepackage{setspace}
\usepackage{booktabs}
\usepackage{tikz}
\usetikzlibrary{shapes, arrows, positioning}
\usepackage{ragged2e}
\usepackage{blindtext}
\usepackage{pdfpages}
\usepackage{epstopdf}
\usepackage{amsmath}
\usepackage{amssymb}
\usepackage{lipsum}
\usepackage{graphicx,adjustbox}
\usepackage{hyperref}
\usepackage{amsthm}
\usepackage{mathalfa,mathrsfs} 
\usepackage{float}
\usepackage{caption}
\usepackage{subcaption}
\usepackage{endnotes}
\usepackage{bibentry}
\usepackage{cite}
\usepackage{tikz-cd} 

\newtheorem{theorem}{Theorem}[section]
\theoremstyle{plain}

\newtheorem{corollary}[theorem]{Corollary}
\newtheorem{definition}[theorem]{Definition}

\newtheorem{lemma}[theorem]{Lemma}

\newtheorem{remark}[theorem]{Remark}

\numberwithin{equation}{section}
\usepackage{fancyhdr}

\usepackage{hyperref}
\hypersetup{
    colorlinks   = true,
    urlcolor     = blue,
    citecolor    = blue,
    bookmarksopen=true
}

\title[Liouville type theorems for fully nonlinear elliptic equations ... ]{Liouville type theorems for fully nonlinear elliptic equations with superlinear growth in gradient}
\author{Amit Kumar Acharya}
\address[Amit Kumar Acharya]{SRM University Amaravati, Andhra Pradesh-522502, India}
\email{amitacharya.h@gmail.com, amitkumar\_acharya@srmap.edu.in}
\author{Ram Baran Verma}
\address[Ram Baran Verma]{SRM University Amaravati, Andhra Pradesh-522502, India}
\email{rambv88@gmail.com, rambaran.v@srmap.edu.in}
\subjclass[2020]{Primary 35B50, 35B51, 35B53, 35J47.}
\keywords{Liouville-type theorems, Pucci operator, positive supersolutions, exterior domains}
\begin{document}
\begin{abstract} 
This article investigates positive supersolutions of the fully nonlinear elliptic system
\[
\begin{cases}
-\mathcal{M}_{\lambda,\Lambda}^{+}(D^{2}u)+|\nabla u|^{q}
\geq\lambda_{1}f_{1}(v)~~\text{in}~~\mathbb{R}^{n}\setminus B_{R_0},\\
-\mathcal{M}_{\lambda,\Lambda}^{+}(D^{2}v)+|\nabla v|^{q}
\geq\lambda_{2}f_{2}(u)~~\text{in}~~\mathbb{R}^{n}\setminus B_{R_0},
\end{cases}
\]
where $q>1,\lambda_{1},\lambda_{2}>0,$ and the nonlinearities exhibit power-type behaviour either near the origin or at infinity. Introducing the effective dimension $\widetilde n_{+}=\frac{\lambda}{\Lambda}(n-1)+1$
associated to extremal Pucci operator, we identify the critical exponent $q_{c}=\frac{\widetilde n_{+}}{\widetilde n_{+}-1},$ which governs the qualitative behaviour of positive supersolutions. Using this framework, we establish sharp Liouville-type nonexistence theorems in exterior domains and determine optimal nonexistence regions through the interaction between the gradient exponent $q,$ the effective dimension $\widetilde n_{+}$ and the nonlinear couplings. In the prototype case $f_{1}(t)=t^{p_{1}},$ $f_{2}(t)=t^{p_{2}}$ the obtained conditions are shown to be optimal. The analysis is carried out under the natural regularity assumption
$u,v\in W^{2,p}_{\mathrm{loc}}(\mathbb{R}^{n}\setminus B_{R_0}),$
for $p>n,$ which is the regularity available for fully nonlinear uniformly elliptic equations, rather than within a classical $C^{2}$ framework. To overcome the absence of pointwise second-order differentiability, we develop an approach based on a nonlinear transformation adapted to the gradient structure, level-crossing principles for absolutely continuous radial profiles, transversality properties of suitable auxiliary quantities, and gradient-domination mechanisms. These tools yield eventual monotonicity and quantitative asymptotic control of radial supersolutions, allowing the Liouville theory to be established at the Sobolev regularity level while preserving the sharpness of the critical thresholds. The results reveal the role of the effective dimension in determining critical nonexistence phenomena for fully nonlinear systems with superlinear gradient terms and provide a framework for the study of Liouville-type problems beyond the classical $C^{2}$ setting.\\
Our results provide the fully nonlinear Pucci analogue of the Liouville theory for semilinear elliptic systems involving nonlinear gradient terms and reveal the fundamental role of the effective dimension in determining the critical nonexistence thresholds.
\end{abstract}
\maketitle
\section{Introduction}
This article is concerned with the nonexistence of positive supersolutions to the following system: 
\begin{equation}\label{main}
\begin{cases}
-\mathcal M_{\lambda,\Lambda}^{+}(D^2u)+|\nabla u|^q
\geq\lambda_{1}f_1(v)~~\text{in}~~\mathbb R^n\setminus B_{R_0},\\
-\mathcal M_{\lambda,\Lambda}^{+}(D^2v)+|\nabla v|^q
\geq\lambda_{2}f_2(u)~~\text{in}~~\mathbb R^n\setminus B_{R_0},
\end{cases}
\end{equation}
where $n\ge2,$ $R_0>0,$ $q>1,$ $\lambda_1,\lambda_2>0$ and $f_1,f_2:(0,\infty)\to(0,\infty)$ are nondecreasing functions. Our main result establishes Liouville-type nonexistence theorems for positive supersolutions of \eqref{main} in exterior domains. To place our results in the broader context of the existing literature, we begin with a brief overview of previous work related to this problem.\\
The Liouville type results in itself is an interesting as it deals with the classification of global solutions to semilinear elliptic equations in the whole space. Besides their intrinsic interest, Liouville-type theorems play a fundamental role in the qualitative theory of elliptic equations. They constitute one of the principal tools for deriving \emph{a priori} estimates for positive solutions of semilinear elliptic equations through blow-up arguments; see, for instance,
\cite{gidas1981global,gidas1981priori,chen1991classification,defigueiredo1994liouville,de2014liouville}.
Subsequently, this theory was extended to quasilinear elliptic equations, including the $p$-Laplacian and more general quasilinear operators; see \cite{mitidieri1996nonexistence,serrin1996local,mitidieri1998liouville,serrin2002cauchy,mitidieri2002fujita}. It also finds interesting application in the study of regularity theory, which was introduced by L. Simon\cite{simon1997schauder}. While semilinear and many quasilinear equations possess a divergence or variational structure that allows the use of integral identities and energy methods, fully nonlinear elliptic equations are naturally posed in nondivergence form and generally lack such tools. For equations involving the Pucci extremal operators, Cutr\`{i} and Leoni \cite{Alessandra} established Liouville-type theorems by combining a nonlinear version of Hadamard's three-circle theorem with comparison arguments. Later, Armstrong and Sirakov \cite{Armstrong01112011,armstrong2011sharp} introduced a powerful maximum-principle approach based on scaling exponents and fundamental solutions of uniformly elliptic operators. For nice introduction of Hadamard's three-circle theorem we refer,\cite{protter2012maximum}. While for major survey on this topic we refer to  \cite{tyagi2016survey,felmer2005some} for further developments in this direction.\\
A natural extension of scalar equations is provided by elliptic systems. The classical Liouville theorem of Gidas and Spruck \cite{gidas1981priori,gidas1981global} has its counterpart in the study of the Lane--Emden system
\[
\begin{cases}
-\Delta u=v^{p_{1}}~~\text{in}~~\mathbb R^{n},\\
-\Delta v=u^{p_{2}}~~\text{in}~~\mathbb R^{n},
\end{cases}
\]
which has attracted considerable attention over the past three decades. Fundamental contributions were obtained by de Figueiredo and P. Felmer \cite{defigueiredo1994liouville}, Mitidieri \cite{mitidieri1996nonexistence}, Serrin and Zou\cite{serrin1996lane,serrin2002cauchy},
and Busca and Man\'{a}sevich \cite{busca2002lane}. These works laid the foundations of the modern Liouville theory for elliptic systems and are closely related to the celebrated Lane--Emden conjecture; see also \cite{cheng2019lane,de2014liouville}. The corresponding theory for fully nonlinear elliptic systems has been developed much more recently. Existence and Liouville-type results for cooperative fully nonlinear systems were obtained by Quaas and Sirakov \cite{quaas2009existence}, while further developments include symmetry and Liouville-type theorems established in
\cite{Maia_2023}. At this juncture we would like to mention that while proving a prior estimate in a non-convex smooth domain we also need Liouville type theorem in the half space. The usual procedure to get it by proving the monotonicity of solution with the help of moving plan method, see\cite{quaas2006existence,felmer2004positive}.\\
Liouville-type theorems in exterior domains constitute an important branch of modern elliptic theory and play a fundamental role in the study of isolated singularities, asymptotic behaviour at infinity. In the semilinear setting, nonexistence results for equations of the form
\[-\Delta u=\lambda f(u)~~\text{in}~~\mathbb{R}^{n}\setminus B_{R_0},\]
originated in the pioneering works of Gidas and Spruck \cite{gidas1981global,gidas1981priori}. Since then, considerable effort has been devoted to identifying optimal conditions on the nonlinearity ensuring the nonexistence of positive supersolutions in exterior domains; see, for example,
\cite{BandleEssen,Liskevich2004,alarcon2016optimal}.\\
A natural and important extension of this theory consists in incorporating nonlinear gradient terms and considering equations of the form
\[-\Delta u+|\nabla u|^{q}=\lambda f(u)~~\text{in}~~\mathbb R^n\setminus B_{R_0}.\]
The presence of the gradient term substantially changes the qualitative behaviour of solutions by destroying the natural scaling invariance of the equation and creating a delicate interplay between diffusion reaction and gradient effects. Sharp Liouville-type theorems for positive supersolutions of such equations in exterior domains were established by Alarc\'{o}n, Burgos-P\'{e}rez, Garc\'{i}a-Meli\'{a}n and Quaas
\cite{Alarcon,AlarconGarciaMelianQuaas2013a,AlarconGarciaMelianQuaas2013b}, where the critical threshold is determined by the interaction between $q, n$ and the asymptotic behaviour of the nonlinearity.\\
The corresponding theory for elliptic systems has also received considerable attention. In particular, the Lane--Emden system
\[
\begin{cases}
-\Delta u=v^{p_{1}}~~\text{in}~~\mathbb R^n\setminus B_{R_0},\\
-\Delta v=u^{p_{2}}~~\text{in}~~\mathbb R^n\setminus B_{R_0},
\end{cases}
\]
has been extensively studied because of its close connection with the Lane--Emden conjecture and reaction--diffusion models. Fundamental contributions were obtained by de Figueiredo and Felmer \cite{defigueiredo1994liouville}, Mitidieri \cite{mitidieri1996nonexistence}, Serrin and Zou \cite{SerrinZou1996}, and Busca and Man\'asevich \cite{busca2002liouville}; see also the survey \cite{de2014liouville} for a comprehensive account of Liouville-type theorems, monotonicity methods, and \emph{a priori} estimates for elliptic systems.\\
More recently, Burgos-P\'{e}rez and Garc\'{i}a-Meli\'{a}n \cite{burgos2018lioville} extended this theory to elliptic systems with nonlinear gradient terms,
\[
\begin{cases}
-\Delta u+|\nabla u|^{q}
=\lambda_{1}f_{1}(v)~~\text{in}~~\mathbb R^n\setminus B_{R_0},\\
-\Delta v+|\nabla v|^{q}
=\lambda_{2}f_{2}(u)~~\text{in}~~\mathbb R^n\setminus B_{R_0},
\end{cases}
\]
and established optimal Liouville-type nonexistence theorems for positive supersolutions in exterior domains. Their approach combines radial reduction, the method of sub- and supersolutions, comparison principles, monotonicity arguments, and delicate asymptotic estimates.
Motivated by these developments, in the present paper we investigate the corresponding problem for fully nonlinear elliptic systems driven by the Pucci extremal operators. To the best of our knowledge, Liouville-type nonexistence theorems for positive supersolutions of fully nonlinear elliptic systems involving nonlinear gradient terms in exterior domains have not been previously established. The main objective of this paper is to fill this gap.\\
One of the distinctive features of the fully nonlinear theory is the appearance of the \emph{effective dimension}
\[\widetilde n_+=\frac{\lambda}{\Lambda}(n-1)+1,\]
which replaces the Euclidean dimension in many qualitative properties of solutions involving the Pucci extremal operators. Consequently, the critical exponent associated with the nonlinear gradient term is no longer $\frac{n}{n-1},$ but instead
\[q_c=\frac{\widetilde n_+}{\widetilde n_+-1}.\]
As will be shown in this paper, the effective dimension $\widetilde n_+$ and the critical exponent $q_c$ govern the asymptotic behaviour of positive supersolutions of \eqref{main} and determine the threshold between existence and nonexistence.\\
The extension of the Laplacian theory to the present fully nonlinear setting presents several substantial difficulties. Unlike the Laplacian, the Pucci operator is a fully nonlinear operator in nondivergence form and therefore lacks both the variational structure and the integral identities that play a fundamental role in the semilinear theory. Moreover, the nonlinear gradient term destroys the homogeneity of the equation and gives rise to new asymptotic regimes that do not appear in homogeneous fully nonlinear equations. Consequently, many of the techniques available for the Laplacian cannot be adapted directly to the present framework.\\
One of the contribution of this article is to overcome the problem posed due to lack of regularity of solution in our setting. In contrast to \cite{burgos2018lioville}, where the analysis is carried out in the classical $C^2$ framework, our results are established under the natural assumption $u,v\in W^{2,p}_{\mathrm{loc}}(\mathbb R^n\setminus B_{R_0})$ for $p>n,$ which is precisely the regularity furnished by the $W^{2,p}$-theory for uniformly elliptic fully nonlinear equations, see
\cite{nornberg2019c1,swiech2020pointwise}. Working at this level of regularity requires new arguments that avoid pointwise second-order differentiability. Our approach combines nonlinear transformations with refined monotonicity arguments and asymptotic analysis for radial supersolutions, leading to sharp decay estimates and eventual monotonicity properties.

The principal contribution of this paper is the establishment of sharp Liouville-type nonexistence theorems for positive supersolutions of \eqref{main}. In particular, we identify the effective dimension $\widetilde n_+$ and the critical exponent $q_c$ as the fundamental quantities governing the asymptotic behaviour of solutions and derive optimal nonexistence criteria in terms of $q$, $\widetilde n_+$ and the asymptotic behaviour of the nonlinearities. In the prototype case
\[
f_1(t)=t^{p_1},~~f_2(t)=t^{p_2},
\]
our results recover the sharp threshold separating existence and nonexistence and extend the Liouville theory of Burgos-P\'erez and Garc\'ia-Meli\'an from the Laplacian to fully nonlinear elliptic systems driven by the Pucci extremal operators.

The paper is organized as follows. Section~2 recalls the necessary preliminaries on the Pucci extremal operators and radial supersolutions. Section~3 establishes the monotonicity and asymptotic properties of radial supersolutions. Sections~4--6 are devoted to the proof of the Liouville-type theorems in the three regimes $q<q_c$, $q=q_c$, and $q>q_c$, respectively. The final section deals with supersolutions exhibiting blow-up behaviour at infinity.
\subsection{Prelimanaries}
In this section, we present the essential definitions and results that will be used throughout the article. For specified constants where $0 < \lambda<\Lambda$, Pucci's extremal operators are defined as follows:
\begin{equation}\label{pucci}
\mathcal{M}^{\pm}_{\lambda,\Lambda}(X)=\Lambda\sum_{\pm e_{i}>0}e_{i}+\lambda\sum_{\pm e_{i}<0}e_{i},
\end{equation}
where $X$ is a $n\times n$ symmetric matrix, and $e_i$ are its eigenvalues. Thus in order to compute $\mathcal{M}^{\pm}_{\lambda,\Lambda}(D^{2}u),$ we need to compute the eigenvalues of Hessian $D^{2}u.$ In general, it is not possible to compute the eigenvalues of the Hessian. But in some particular cases we can easily compute the eigenvalues of the Hessian. One such instance is when $u$ is a radial function then we have.
\begin{lemma}\label{radial}
Let $\tilde{u}:~[0,\infty)\longrightarrow \mathbb{R}$ be $C^{2}$ function such that $u(x)=\tilde{u}$. Then for any $x\in\mathbb{R}^{n}\setminus\{0\},$ the eigenvalues of the Hessian $D^{2}u(x)$ are $\frac{\tilde{u}^{\prime}(|x|)}{|x|}$ with multiplicity $n-1$ and $\tilde{u}^{\prime\prime}(|x|)$ with multiplicity $1.$
\end{lemma}
Let us introduce the function 
\begin{equation}\label{theta12}
\theta(s)= 
\begin{cases}
\Lambda,& \text{if}~~s\geq 0\\
\lambda,& \text{if}~~s<0.
\end{cases}
\end{equation}
Then, since $u$ is radially symmetric, the eigenvalues of $D^2 \tilde{u}$ are $\tilde{u}^{\prime \prime}(r)$ and $\tilde{u}^{\prime}(r)/r$, so by definition of $\mathcal{M}^{+}_{\lambda,\Lambda}$ we have
\[
\mathcal{M}^{+}_{\lambda,\Lambda}\left(D^2 \tilde{u}\right)(r)=\theta\left(\tilde{u}^{\prime \prime}(r)\right) \tilde{u}^{\prime \prime}(r)+\theta\left(\tilde{u}^{\prime}(r)\right)(n-1) \frac{\tilde{u}^{\prime}(r)}{r}.
\]
\begin{remark} \hspace{-1cm}
\begin{enumerate}
\item{} By a \textbf{positive} solution $(u,v)$ of \eqref{main}, we mean $u,v>0$ and satisfies the equation.
\item{} We denote by $\widetilde{n}_{+}:=\frac{\lambda}{\Lambda}(n-1)+1,$  which is the effective dimension associated with the Pucci's extremal operator $\mathcal M^{+}_{\lambda,\Lambda}$ defined above.
\item{} We denote $q_c:=\frac{\widetilde{n}_{+}}{\widetilde{n}_{+}-1}.$
\item{} There is no harm in assuming that $0<\lambda<1.$ We will be assuming it throughout the article.
\end{enumerate}
\end{remark} 
\begin{definition}
A pair $(u,v)\in C(\mathbb{R}^{n}\setminus B_{R_0})^{2}$ is called a \emph{viscosity supersolution} of \eqref{main} if for every $x_0\in\mathbb{R}^{n}\setminus B_{R_0}$ the following hold:
\begin{enumerate}
\item Whenever $\varphi\in C^2(\mathbb{R}^{n}\setminus B_{R_0}))$ is such that $u-\varphi$ attains a local minimum (resp. maximum) at $x_0,$ one has
\[-\mathcal M_{\lambda,\Lambda}^{+}(D^2\varphi(x_0))+|\nabla\varphi(x_0)|^q\geq(\text{resp.}\leq) \lambda_1 f_1(v(x_0)).\]
\item Whenever $\psi\in \in C^2(\mathbb{R}^{n}\setminus B_{R_0}))$ is such that $v-\psi$
attains a local minimum (resp. maximum) at $x_0,$ one has
\[-\mathcal M_{\lambda,\Lambda}^{+}(D^2\psi(x_0))+|\nabla\psi(x_0)|^q
\geq(\text{resp.}\leq)\lambda_2 f_2(u(x_0)).\]
\end{enumerate}
The function $(u,v)\in C(\mathbb{R}^{n}\setminus B_{R_0})^{2}$ is called viscosity solution if it is sub and supersolution.
\end{definition}
\section{Main result}
As we consider the system \eqref{main}. We impose the following conditions on $f_{1},f_{2}$ near zero.
\begin{itemize}
\item[{\textbf{(Z):}}]
\begin{equation}\label{mueq}
\mu_{1}:=\liminf_{t\to 0}\frac{f_{1}(t)}{t^{p_{1}}}>0,\quad \mu_{2}:=\liminf_{t\to 0}\frac{f_{2}(t)}{t^{p_{2}}}>0,
\end{equation}
for some $p_{1},p_{2}>0.$ At this point we would like to emphasize that $\lambda_1,\lambda_2$ in this system of equation has no relation with $\lambda$ appearing in the definition of Pucci's exteremal operator.
\item[\textbf{(F):}] We say that $f_{1},f_{2}$ satisfy the condition $(F)$ if, $f_{1}$ and $f_{2}$ be functions defined on $(0,+\infty)$ which are continuous, nondecreasing, and strictly positive. \\
\end{itemize}
Following the ideas of \cite{burgos2018lioville}, we state our results in three cases:
\begin{enumerate}
\item{}$1<q<q_{c}$
\item{}$q=q_{c}$
\item{}$q>q_{c}.$
\end{enumerate}
For the $q$ lying in the first two cases, that is, $1<q\leq q_{c}$ and also $p_{1}p_{2}>q_{c}^{2}$ we use the notation 
\begin{equation}\label{alpha}
\alpha_{1}=\frac{q(p_{1}+q)}{p_{1}p_{2}-q^{2}}~~\text{and}~~\alpha_{2}=\frac{q(p_{2}+q)}{p_{1}p_{2}-q^{2}}
\end{equation}
Because of the symmetric nature of the Problem \eqref{main}, we can assume that $p_{1}\geq p_{2}.$ We observe that the relative positions of $\alpha_{1}$ and $\alpha_{2}$ with respect this new quantity $\beta=\frac{2-q}{q-1}$ determine the region of non-existence of positive supersolutions. Depending on the position of $q,$ we divide the main theorem into four parts. Now we are ready to state the first part of the main theorem. 
\begin{theorem}[First case]\label{thm1}
Assume that $\widetilde{n}_{+}\geq2$ and  $1<q<q_{c}.$ Let ($F$) be in force and \eqref{mueq} holds for some exponents satisfying $p_{1} \geq p_{2} > 0$. Then the following statements are valid.
\begin{enumerate}
\item{} If $p_{1}p_{2}\le q^{2}$, there exist no positive supersolutions of \eqref{main} which do not exhibit blow-up behavior at infinity.
\item{} If $p_{1}p_{2}>q^{2}$ and $\alpha_{1}> \beta$, then equation \eqref{main} admits no positive supersolutions which do not exhibit blow-up behavior at infinity. 
\item{} Assume that $p_{1}p_{2}>q^{2}$ and $\alpha_{1}=\alpha_{2}=\beta$. 
If every pair of positive numbers $(t_{1},t_{2})$ satisfies at least one of the inequalities
\begin{equation}\label{1.5}
\begin{cases}
\alpha_{1}\left(\Lambda(n-1)-\lambda(\alpha_{1}+1)\right)t_1+\alpha_{1}^q t_1^q<\lambda_1\mu_{1}t_2^{p_1},\\
\alpha_{2}\left(\Lambda(n-1)-\lambda(\alpha_{2}+1)\right)t_2+\alpha_{2}^q t_2^q<\lambda_2\mu_{2}t_1^{p_2},
\end{cases}
\end{equation}
then \eqref{main} has no positive supersolutions which do not exhibit blow-up behavior at infinity.
\item Suppose that $p_{1}p_{2} >q^{2}$, $\alpha_{2}<\alpha_{1}=\beta$, and that for every $(t_{1},t_{2})>0$ at least one of the following inequalities holds:
\begin{equation}\label{1.6}
\begin{cases}
\alpha_{1}\left(\Lambda(n-1)-\lambda(\alpha_{1}+1)\right)t_1+\alpha_{1}^q t_1^q<\lambda_1\mu_{1}t_2^{p_1},\\
\alpha_{2}^q t_2^q<\lambda_2\mu_{2}t_1^{p_2}.
\end{cases}
\end{equation}
Then there are no positive supersolutions of \eqref{main} which do not exhibit blow-up behavior at infinity.
\end{enumerate}
\end{theorem}
\begin{theorem}[Second case]\label{thm2}
Let $\widetilde{n}_{+}\geq3$ and $q=q_{c}$. Let ($F$) be in force and \eqref{mueq} holds for some exponents satisfying $p_{1} \geq p_{2}$. Then the following statements are valid.
\begin{itemize}
\item[(a)] If $p_{1}p_{2}\leq q_c^2$, equation \eqref{main} admits no positive supersolutions which do not exhibit blow-up behavior at infinity.
\item[(b)] If $p_{1}p_{2}>q_c^2$ and $\alpha_{1}>\widetilde{n}_{+}-2$, then \eqref{main} has no positive supersolutions which do not exhibit blow-up behavior at infinity.
\end{itemize}
\end{theorem}
We now turn to the case in which $q>q_{c}$. In this regime the analysis becomes more involved, and the nonexistence results depend on the interplay between two distinct pairs of critical exponents, namely
\begin{equation}\label{1.7}
\bar{\alpha}_{1}:=\frac{2(p_1+1)}{p_1p_2-1},\quad 
\bar{\alpha}_{2}:=\frac{2(p_2+1)}{p_1p_2-1},
\end{equation}
and
\begin{equation}\label{1.8}
\hat{\alpha}_{1}:=\frac{q(p_1+2)}{p_1p_2-q},\quad 
\hat{\alpha}_{2}:=\frac{q+2p_2}{p_1 p_2-q},
\end{equation}
respectively. We note that the latter pair of exponents is relevant only in the range $p_1p_2>q$.
\begin{theorem}[Third case]\label{thm3}
Assume that $\widetilde{n}_{+}\geq3$ and that $q>q_{c}$.  
Let ($F$) be in force and \eqref{mueq} holds for some exponents satisfying $p_{1}\geq p_{2}>0.$ Then the following alternatives hold.
\begin{itemize}
\item[(a)] If $p_1p_2 \leq 1$, there exist no positive supersolutions of \eqref{main} which do not exhibit blow-up behavior at infinity.
\item[(b)] If $p_1p_2>1$, $\bar{\alpha}_{1}\geq\widetilde{n}_{+}-2$ and $\bar{\alpha}_{2}\geq\beta$, then \eqref{main} admits no positive supersolutions which do not exhibit blow-up behavior at infinity.
\item[(c)] If $p_1p_2 > 1,$ $\bar{\alpha}_{1}\geq\widetilde{n}_{+}-2,$ $\bar{\alpha}_{2}<\beta$ and $\hat{\alpha}_{1} \geq\widetilde{n}_{+}-2$, then there are no positive supersolutions of \eqref{main} which do not exhibit blow-up behavior at infinity.
\end{itemize}
\end{theorem} 
 Our next result is concerned with nonexistence of positive supersolution under the conditions on $f_{1}$ and $f_{2}$ near infinity, more precisely
\begin{equation}\label{infinity}
\lim_{|x|\to\infty}u(x)=\lim_{|x|\to\infty}v(x)=+\infty,
\end{equation}
\begin{remark}\hspace{-1cm}
\begin{enumerate}
\item We can see that if $p_1p_2>1$ and $\alpha_{2}\le\beta$ then, $p_1p_2>q,$ so the exponents in \eqref{1.7} are positive and as well as well defined .
\item If $n=2$, then it is not covered by Theorem~\ref{thm2} and Theorem~\ref{thm3}. Indeed, this will be proved in the forthcoming Lemma \ref{lem12} when $n=2$ and $q \ge 2$, every positive supersolution of system \eqref{main} satisfies condition \eqref{infinity}. Consequently, this two–dimensional case will be treated separately in the next theorem.
\end{enumerate}
\end{remark}
 We conclude by considering positive supersolution which verify \eqref{infinity}, In this case, only the behaviour of $f_1$ and $f_2$  at infinity. Hence We assume that
\begin{equation}\label{infinity2} 
\liminf_{t\to+\infty}\frac{f_1(t)}{t^{p_1}}>0,\quad
\liminf_{t\to+\infty}\frac{f_2(t)}{t^{p_2}}>0,   
\end{equation}
for some $p_1,p_2>0$.
\begin{theorem}\label{thm4}
Let $\widetilde{n}_{+} \ge 2.$ Let ($F$) be in force and \eqref{infinity2} holds for some exponents satisfying $p_{1},p_{2}>0.$ Furthermore, if $p_{1}p_{2}>q^2,$ then  \eqref{main} admits no positive supersolutions that exhibit blow-up behavior at infinity.
\end{theorem}
In order to prove these theorem we need some intermediate results in which some of them is important in itself.
\section{Intermediate results}
We have decoupled system of fully nonlinear elliptic equations \eqref{main}. If we consider this system of equation in annulus domain then the solution is $C^{1,\alpha}$ (see \cite{nornberg2019c1}). Furthermore, the operator considered here is rotationally invariant. So we can apply the result of Theorem 1.1 in \cite{dos2020symmetry} to get all the solution of $ \mathcal{M}_{\lambda,\Lambda}^+(D^2u)+|\nabla u|^q=\lambda_{1} f_{1}(v)$ is radial in the annular domain.
\subsection{Classification of Supersolutions and Radial Reduction}
Let us define
\begin{equation}\label{min} 
m_u(R):=\min_{|x|=R}u(x).
\end{equation}
\begin{remark}\label{monotone} 
Observe that if $(u,v)$ are nonnegative solutions of \eqref{main}, then it also satisfies 
\[-\mathcal{M}_{\lambda,\Lambda}^+(D^2u)+|\nabla u|^q\ge 0\]
so as a consequence of comparison principle we find that the function $m_u(R)$ is monotone decreasing for all $R>R_{1}$ for some $R_1.$ Same is true for $m_{v}(R).$
\end{remark}
Our next result is related to the existence of supersolutions of \eqref{main} which blow up at infinity and which do not blow up at infinity.
\begin{lemma}\label{lem2}
Let $q>1$ and $f_1,f_2$ satisfying ($F$). Suppose that $(u,v)\in C^{1}( \mathbb{R}^{n} \setminus B_{R_0})^{2}$ is a positive supersolution of \eqref{main}. 
Denote by $m_u(R)$ and $m_v(R)$ the radial minimum functions defined in \eqref{min}. Then one of the following alternatives holds.
\begin{enumerate}
\item[(a)] Given functions $m_u$ and $m_v$ are nonincreasing, bounded and converge to $0$ as $R\to\infty$.
Moreover, if $f_1$ and $f_2$ satisfy \eqref{mueq}, then there exists positive constant $C>0$ such that
\[
\begin{cases}
m_v^{p_{1}}(2R)\le C\left(\frac{m_u(R)}{R^2}+\frac{m_u^q(R)}{R^q}\right),\\
m_u^{p_2}(2R)\le C\left(\frac{m_v(R)}{R^2}+\frac{m_v^{q}(R)}{R^q}\right),
\end{cases}
\]
for all $R>R_0$. Or,
\item[(b)] The functions $m_u(R)$ and $m_v(R)$ are nondecreasing and diverge to infinity as $R\to\infty$.
\end{enumerate}
\end{lemma}
\begin{proof} We will follow the idea of \cite{Alessandra,burgos2018lioville}. Let  $\varphi\in C_0^\infty(\mathbb{R})$ be a cut-off function such that $0\le \varphi \le 1,$
\[
\begin{cases}
\varphi =0& \text{in}~~(-\infty,1)\cup(4,+\infty),\\
\varphi =1& \text{in}~~[2,3].
\end{cases}
\]
For $R>R_0,$ let us consider the following test function $\psi(x)=m_u(2R)\varphi\left(\frac{|x|}{R}\right),~~|x|>R_0.$
Note that there exists a point $x_R \in \partial B_{2R}$ such that $u(x_{R})-\psi(x_{R})=0$. Moreover, since $u>\psi$ in $\mathbb{R}^{n}\setminus B_{4R}$ and $B_R\setminus B_{R_0}$, the function $\psi$ touches $u$ at some point $y_R\in B_{4R}\setminus B_R,$ as $u$ is a supersolution we have
\begin{equation}\label{super11}
f_{1}(v(y_{R}))\leq-\mathcal{M}{_{\lambda,\Lambda}^+}(D^2\psi(y_R))+|D\psi(y_{R})|^{q}.
\end{equation}
\[D^{2}\psi(y_R)=\frac{m_u(2R)}{R^{2}}
\varphi^{\prime\prime}\left(\frac{|y_R|}{R}\right)\frac{y_R\otimes y_R}{|y_R|^{2}}
+\frac{m_u(2R)}{R}\frac{\varphi^{\prime}\left(\frac{|y_R|}{R}\right)}{|y_R|}
\left(I-\frac{y_R\otimes y_R}{|y_R|^{2}}\right).\]
Since $y_R\in B_{4R}\setminus B_R,$  we have $R\le |y_R|\le 4R,$
and therefore $\frac{1}{|y_R|}\leq\frac{C}{R}.$ Using that $\varphi^{\prime},\varphi^{\prime\prime}$ are bounded and
\[\left|\frac{y_R\otimes y_R}{|y_R|^{2}}\right|\leq1,\quad\left|I-\frac{y_R\otimes y_R}{|y_R|^{2}}\right|\leq C,\]
we find $|D^{2}\psi(y_R)|\leq C\frac{m_u(2R)}{R^{2}}.$ We also know that
$\left|\mathcal M_{\lambda,\Lambda}^{+}(X)\right|\leq C(n,\lambda,\Lambda)|X|,$
which implies that
\[\left|\mathcal M_{\lambda,\Lambda}^{+}\bigl(D^{2}\psi(y_R)\bigr)\right|\leq C\frac{m_u(2R)}{R^{2}},\]
consequently, we have
\[-\mathcal M_{\lambda,\Lambda}^{+}\bigl(D^{2}\psi(y_R)\bigr)\leq C\frac{m_u(2R)}{R^{2}}.\]
Using the above values in \eqref{super11}, we have 
\[
f_{1}(v(y_{R}))\leq-\mathcal{M}{_{\lambda,\Lambda}^+}(D^2 \psi(y_R))+|D\psi(y_{R})|^{q}\leq C\left( \frac{m_u(2R)}{R^2}+\frac{m_u^{q}(2R)}{R^q} \right),
\]
where $C$ is a positive constant depending on $n,\lambda,\Lambda,\varphi$ that is $C=C(n,\lambda,\Lambda,\varphi)$ and $R$ is sufficiently large. Furthermore, as $(u,v)$ is a supersolution of \eqref{main} so 
\begin{equation}\label{super1}
f_{1}(v(y_R))\le C\left(\frac{m_u(2R)}{R^2}+\frac{m_u^q(2R)}{R^q}\right).
\end{equation}
Let us start by considering the case  $m_u(R)$ is bounded, consequently by \eqref{super1}, we have  \[\lim_{R\to+\infty} f_{1}(v(y_R))=0.\] Moreover, as $f_1$ is nondecreasing and positive  in $(0,+\infty)$, we get $\lim_{R\to+\infty}v(y_R)=0.$\\
From the Remark \eqref{monotone}, we see that $m_{u}$ and $m_{v}$ are monotonic decreasing. Also we have assumed that these are bounded so we find that $\lim_{R\to+\infty} m_v(R)=0$. Similarly for $v$, it also follows that $m_u$ nonincreasing for large $R$ and $\lim_{R\to+\infty} m_u(R)=0.$ Now observe that $y_{R}\in B_{4R}\setminus B_{R}$ so we have $v(y_{R})\geq m_{v}(4R).$ Thus by using \eqref{mueq}, we get $f_1(v(y_R))\ge Cv^{p_1}(y_R)\geq Cm_v^{p_1}(4R)$ for sufficiently large $R.$
Hence
\[m_v^{p_{1}}(4R)\leq C\left(\frac{m_u(2R)}{R^2}+\frac{m_u^q(2R)}{R^{q}} \right)~~\text{for large} ~R.\]
Similarly,
\[m_u^{p_{2}}(4R)\leq C\left(\frac{m_v(2R)}{R^2}+\frac{m_v^{q}(2R)}{R^{q}} \right)~~\text{for large}~R.\]
Replacing $2R$ by $R$  gives the estimates in $(a).$ \\
In the case $(b)$ we assume that $m_u$ and $m_v$ both the unbounded. Indeed, if one of them were bounded, the previous argument would imply that the other is also bounded and part $(a)$ holds. Also in view of Remark \eqref{monotone}, we  have that $m_u$ and $m_v$ are non decreasing for a sufficiently large $R$, such that
\[\lim_{R\to+\infty} m_u(R)=\lim_{R\to+\infty} m_v(R)=+\infty.
\]
This completes the proof.
\end{proof}
\begin{remark}
In the following two lemmas, we show that the existence of a positive supersolution of \eqref{main} guarantees the existence of a positive radially symmetric supersolution. Moreover, when the supersolution do not blow up at infinity, one can further deduce the existence of a positive radial solution.
\end{remark}
\begin{lemma}\label{lem3}
Let $n \geq 2$ and $q>1$ and $(F)$ applies. Assume that system \eqref{main} admits a positive viscosity supersolution $(u,v)$. Then there is a positive radially symmetric supersolution $(w,z)$ of \eqref{main}. 
Moreover, if $(u,v)$ does not exhibit blow-up behavior at infinity, then
\[w(R)\to 0,~~z(R)\to0,~~w'(R)\to 0,~~z'(R)\to 0 
\quad\text{as}~~R\to+\infty.\]
On the other hand, if $(u,v)$ exhibits blow-up behavior at infinity, then we can say the radial supersolution $(w,z)$ is also exhibits blow-up behavior at infinity.
\end{lemma}
\begin{proof}
Let us consider $R_{1} > R_0$ and define the following annulus  
\[\mathcal{A}(R_0,R_{1}):=\{x\in\mathbb{R}^{n}:R_0<|x|<R_{1}\}.\]
For $R > R_0$, consider the following decoupled system 
\begin{equation}
\begin{cases}
-\mathcal{M}^+(D^2w)+|\nabla w|^{q}=\lambda_1 f_{1}(m_v(|x|))&\text{in}~~\mathcal{A}(R_0,R_{1}),\\
-\mathcal{M}^+(D^2z)+|\nabla z|^{q}=\lambda_2f_{2}(m_u(|x|))&\text{in}~~\mathcal{A}(R_0,R_{1}),\\
w=z=0&\text{on}~~\partial\mathcal{A}(R_0,R_{1}),
\end{cases}
\end{equation}
where $m_u, m_v$ is defined by \eqref{min}. The considered operator satisfy the comparison principle see Theorem 3.3 in \cite{crandall1992user}. Moreover, it is easy to see that $\underline{z}=\underline{w}=0$ subsolution and for large value of $M$ the function $\overline{z}=\overline{w}=M(|x|-R_0)(R_{1}-|x|)$ is a supersolution. Consequently the existence of solution follows from Perron's method, see Theorem 4.1 in \cite{crandall1992user}.  Moreover, as the considered operator is concave in the Hessian, so by the result of Section 5 in \cite{swiech2020pointwise},  $z_{R_1}, w_{R_{1}} \in W_{\text{loc}}^{2,s}(\mathcal{A}(R_0,R_1))$, for some $s>n$. Furthermore, our considered operator satisfies the comparison principle so the uniqueness follows thereby both $w_{R_1}$ and $z_{R_1}$ must be radially symmetric.
Set
\[\hat w_{R_{1}}:=\min\{m_u(R_0),m_u(R_{1})\}+z_{R_{1}},~~
\hat w_{R_{1}}:=\min\{m_v(R_0),m_v(R_{1})\}+z_{R_{1}}\]
and notice that $\hat w_{R_{1}}$ and $\hat z_{R_{1}}$ are radially symmetric solutions of
\begin{equation}\label{3.7}
\begin{cases}
-\mathcal{M}^+(D^2w) + |\nabla w|^{q}=\lambda_{1} f_{1}(m_v(|x|))&\text{in}~~\mathcal{A}(R_0,R_{1}),\\
w=\min\{m_u(R_0),m_u(R_{1})\}&\text{on}~~\partial \mathcal{A}(R_0,R_{1}),
\end{cases}
\end{equation}
and
\begin{equation}\label{3.8}
\begin{cases}
-\mathcal{M}^+(D^2z)+|\nabla z|^{q}=\lambda_2f_2 (m_u(|x|))&\text{in}~~\mathcal{A}(R_0,R_{1}),\\
z=\min\{m_v(R_0),m_v(R_{1})\}& \text{on}~~\partial\mathcal{A}(R_0,R_{1}),
\end{cases}
\end{equation}
respectively. Since $f_{1}$ and $f_{2}$ are monotone, so $f_{1}(u)\geq f_{1}(m_u(|x|)),~~f_{2}(v) \geq f_{2}(m_v(|x|)).$  Also in view of \eqref{main}, we get
\[
\begin{cases}
-\mathcal{M}^+(D^2u)+|\nabla u|^{q}=\lambda_{1}f_{1}(v(x))\ge\lambda_1f_{1}(m_v(|x|)),\\
-\mathcal{M}^+(D^2v) + |\nabla v|^{q} =\lambda_2 f_{2}(u(x))\ge\lambda_2f_{2}(m_u(|x|)).
\end{cases}
\]
Thus $u$ and $v$ are supersolution of \eqref{3.7} and \eqref{3.8} respectively. Thus by comparison principle we have 
\[
0\leq\hat w_{R_{1}}\leq u,~~0\leq \hat z_{R_{1}}\leq v~~\text{in}~~\mathcal{A}(R_0,R_{1}).
\]
Hence the families $\{\hat z_{R_{1}} \}_{R_{1}>R_0}$ and  $\{\hat w_{R_{1}} \}_{R_{1}>R_0}$ are bounded. Moreover, as a consequence of local $W^{2,s}$ estimate and Sobolev embedding we find that  $\{\hat w_{R_{1}}\}$ and $\{\hat z_{R_{1}}\}$ are uniformly bounded in $C^{1,\alpha}_{\text{loc}}.$ Therefore, by Arzela-Ascoli Theorem and diagonal argument, there exists a sequence $\{R_{1,n}\}_{n\ge1}$ satisfying $R_{1,k}\to +\infty$ such that 
\[
\left\{
\begin{aligned}
&\lim_{k\to\infty}R_{1,k}=+\infty&\text{in}~~C^{1}_{\mathrm{loc}}(\mathbb{R}^{n}\setminus B_{R_0}),\\
&\lim_{k\to\infty}\hat w_{R_{1,k}}=w&\text{in}~~C^{1}_{\mathrm{loc}}(\mathbb{R}^{n}\setminus B_{R_0}),\\
&\lim_{k\to\infty}\hat z_{R_{1,k}}=z&\text{in}~~C^{1}_{\mathrm{loc}}(\mathbb{R}^{n}\setminus B_{R_0}).
\end{aligned}
\right.
\]
We denote 
\[ 
l_u:=\lim_{R\to\infty} m_u(R)~~\text{and}~~l_v:=\lim_{R\to\infty} m_v(R),
\]
 then the radially symmetric solution of $w$ is given by
\[
\begin{cases}
-\mathcal{M}^+(D^2w)+|\nabla w|^{q}=\lambda_1 f_{1}(m_v(|x|))
&\text{in}~~\mathbb{R}^{n}\setminus B_{R_0},\\
\text{with}~~\min\{m_u(R_0),l_u\}\le w\le m_u,
\end{cases}
\]
similarly, $z$ is radially symmetric solution of
\[
\begin{cases}
-\mathcal{M}^+(D^2z)+|\nabla z|^{q}=\lambda_{2}f_{2}(m_u(|x|))&\text{in}~~ \mathbb{R}^{n}\setminus B_{R_0},\\
\text{with}~~\min\{m_v(R_0),l_v\}\le z\le m_v.
\end{cases}
\]
By the regularity of the equation and radial symmetry give that $w$ and $z$ are classical solutions. Moreover, the monotonicity of $f_{1}$ and $f_{2}$ implies that $f_{1}(m_v(|x|)) \ge f_{1}(z),~f_{2}(m_u(|x|)) \ge f_{2}(w),$ so $(w,z)$ is a positive and radially symmetric supersolution of
\begin{equation}\label{2.7}
\begin{cases}
-\mathcal{M}^+(D^2w)+|\nabla w|^{q}=\lambda_{1}f_{1}(z)&\text{in}~~\mathbb{R}^{n}\setminus B_{R_0},\\
-\mathcal{M}^+(D^2z)+|\nabla z|^{q}=\lambda_2f_{2}(w)&\text{in}~~\mathbb{R}^{n}\setminus B_{R_0}.
\end{cases}
\end{equation}
This proves the first part of the lemma.\\ 
Now we will show how does the behaviour of $\bm{(u,v)}$ near infinity determine the behavior of the above radial solution.
\begin{enumerate}
\item{} If $(u,v)$ does not blow up at infinity: In this case Lemma \ref{lem2} (a) we have $l_u=l_v=0$. Since $w\le m_u$ and $z\le m_v$, we obtain $\lim_{r\to\infty} w(r)=\lim_{r\to\infty}z(r)=0.$ To show $w'(r)\to 0$~as~$r\to \infty$, choose an arbitrary sequence 
$\{b_k\}\subset\mathbb{R}^n\backslash B_{R_0}$ with $|b_k|\to\infty$ and consider
\[
\bar{w}_k(y):=w(b_k+y)~~\text{for}~~|y|<|b_k|-R_0.
\]
Then $\bar{w}_k$ satisfies $-\mathcal{M}^{+}(D^2\bar w_k)+|\nabla\bar w_k|^{q}=\lambda_{1}f_{1}(m_v(|b_k+y|)).$ As before, $\bar w_k\to\bar w$ in $C^{1}_{\mathrm{loc}}(\mathbb{R}^{n})$  with
$- \mathcal{M}^{+}(D^2\bar w)+|\nabla \bar w|^{q}=0.$ Since $\bar w_k(0)=w(b_k)\to 0=\bar w(0)$, then from the strong maximum principle we obtain, $\bar w= 0~~\text{in}~\mathbb{R}^n$, hence $\nabla w(b_k)=\nabla w_k(0)\to 0$. Because $\{b_k\}$ was arbitrary, we proved $w(r)\to 0$ and $w'(r)\to 0$ as $r\to\infty$.  The same holds for $z$ and $z'$.
\item{} If $\bm{(u,v)}$ exhibit blow-up behavior at infinity: Again from Lemma~\ref{lem2} $(b),$ $l_u=l_v=+\infty$. Since
\[
\begin{cases}m_u(R_0)=\min\{m_u(R_0),l_u\}\le w,\\m_v(R_0)=\min\{m_v(R_0),l_v\}\le z,
\end{cases}
\]
then functions $w$ and $z$ are bounded  below. Thus, $(w,z)$ definitely exhibit blow-up behavior at infinity in view of Lemma \ref{lem2}.
\end{enumerate}
\end{proof}
\begin{remark}\label{rem:w}
The construction in the proof of Lemma \ref{lem3} yields radial functions
$w,z\in W^{2,s}_{\mathrm{loc}}(\mathbb R^n\setminus B_{R_0})$
for some $s>n.$ More precisely, the approximating solutions $\hat w_{R_{1,k}},~\hat z_{R_{1,k}}$ are uniformly bounded in $W^{2,s}_{\mathrm{loc}}(\mathbb R^n\setminus B_{R_0}),$ and converge, up to a subsequence, in $C^{1}_{\mathrm{loc}}(\mathbb R^n\setminus B_{R_0})$ to the radial supersolution $(w,z).$ At this stage, however, one cannot conclude that $w$ and $z$ are classical $C^2$-solutions. But we have these solution in $W^{2,s}.$
\end{remark}
\begin{remark}
For positive radial supersolutions $(w, z)$ which do not blow up at infinity, Lemma~\eqref{lem3} ensures that 
$w, w', z, z' \to 0~~\text{as}~~r \to +\infty.$
Hence, it is enough to assume $q<2$ in the proofs, since if $(w, z)$ supersolution of \eqref{main}  then it is also a supersolution of
\[
\begin{cases}
-\mathcal{M}^{+}(D^{2}w)+|\nabla w|^{\tilde{q}}=\lambda_{1}f_1(z), \\
-\mathcal{M}^{+}(D^{2}z)+|\nabla z|^{\tilde{q}}=\lambda_{2}f_2(w),
\end{cases}
\]
for every $q>\tilde{q}$, provided $R_1$ is large enough.
\end{remark}
\begin{lemma}\label{lem4}
Let $n\geq2$ and $q>1$ and also $q<2$ if $n=2.$ Let $(F)$ and \eqref{mueq} satisfy. 
Assume that system~\eqref{main} admits a positive supersolution $(u,v)$ which do not blow up at infinity. Then, for any choice of $\varepsilon_{1} \in (0,\mu_{1})$ and $\varepsilon_{2} \in (0,\mu_{2})$, one can find a radius $R_{1} > R_0$ and a bounded, positive radially symmetric solution $(w,z)$ of the system
\begin{equation}\label{2.8}
\begin{cases}
-\mathcal{M}^{+}(D^{2}w)+|\nabla w|^{q}=\lambda_1\varepsilon_{1} z^{p_1}&\text{in}~~\mathbb{R}^{n}\setminus B_{R_{1}},\\
-\mathcal{M}^{+}(D^{2}z)+|\nabla z|^{q}=\lambda_2\varepsilon_{2}w^{p_2}&\text{in}~~\mathbb{R}^{n}\setminus B_{R_{1}}.
\end{cases}
\end{equation}
\end{lemma}
\begin{proof} From the Lemma~\ref{lem3}, there exists a bounded, positive and radially symmetric supersolution $(\bar{u},\bar{v})$ of system~\eqref{main}, moreover by Lemma \ref{lem2} this supersolution also satisfies $\lim_{R\to+\infty} \bar u(R)=\lim_{R\to+\infty} \bar v(R)=0.$ Let $\varepsilon_1\in(0,\mu_{1})$ and $\varepsilon_2\in(0,\mu_{2})$. In view of \eqref{mueq}, and the choice of $\varepsilon_1,\varepsilon_2$ there exists $\tilde{R_0}>R_0$ such that 
\[f_{1}(v(R))\geq \varepsilon_1 v(R)^{p_1},~~f_{2}(u(R))\geq \varepsilon_{2} u(R)^{p_2},~~\text{for}~~R>\tilde{R}_0.\]
It shows that $(\bar{u},\bar{v})$ is a supersolution of system~\eqref{2.8} in $\mathbb{R}^{n}\setminus B_{R_1}$. Let us set 
\[N:=\sup_{R\ge 2R_0} \left(\bar u(R)+\bar v(R)\right).\]
 By following the similar idea as in Lemma 5 \cite{Alarcon}, there exists an $\tilde{\tilde{R}}>\tilde{R}_{0}$ (depending on $N$) such that following holds: For every $R_1,R_2>\tilde{\tilde{R}}$ satisfying $R_{2}>2R_{1},$ there exists a radially symmetric, positive functions $\underline{u}_{R_2}$ and $\underline{v}_{R_2}$ satisfying
\[-\mathcal{M}^+(D^2w)+|\nabla w|^q\le 0~~\text{in}~~\mathcal{A}(R_1,R_2),
\]
with boundary conditions
\[
\begin{cases}
\underline{u}_{R_2}(R_1)=\bar u(R_1),& \underline{u}_{R_2}(R_2)=\bar u(R_2),\\
\underline{v}_{R_2}(R_1)=\bar v(R_1),& \underline{v}_{R_2}(R_2)=\bar v(R_2).
\end{cases}
\]
As $\bar u(R_1)>\bar u(R_2)$, the maximum principle implies
\[\underline{u}_{R_2}\leq\bar{u}(R_1)\quad\text{in}~~\mathcal{A}(R_1,R_2),\quad \text{so}~~\underline{u}'_{R_2}(R_1)\leq 0.\]
Similarly, $\underline{v}'_{R_2}(R_1)\le 0.$  Fix $R_1>\tilde{\tilde{R}}_{0}$ and choose $R_2$ such that $R_2>2R_1$. Then the comparison principle implies:
\[\underline{u}_{R_2}\leq\bar{u},\quad \underline{v}_{R_2}\le\bar{v}\quad\text{in}~~\mathcal{A}(R_1,R_2).\]
Therefore, the result follows by applying the sub- and supersolution method to construct a radially symmetric solution 
$(z_{R_2},w_{R_2})$ of
\begin{equation}\label{2.9}
\begin{cases}
-\mathcal{M}^+(D^2w)+|\nabla w|^q&=\lambda_1 \varepsilon_1 z^{p_1}\quad\text{in}~~\mathcal{A}(R_1,R_2)\\
-\mathcal{M}^+(D^2z)+|\nabla z|^q&=\lambda_2 \varepsilon_2 w^{p_2}\quad\text{in}~~\mathcal{A}(R_1,R_2)\\
\hspace{2.2cm}(w,z)&=(\bar{u},\bar{v})\quad\text{on}~~\partial \mathcal{A}(R_1,R_2).
\end{cases}
\end{equation}
 Here our sub and supersolution approach is similar to Lemma 14\cite{burgos2018lioville}. Since $0\leq w_{R_2}\leq\bar u$ and $0\leq z_{R_2}\leq\bar v$, it follows that the families $\{w_{R_2}\}$ and $\{z_{R_2}\}$ are uniformly bounded in $L^{\infty}$. Next, we derive locally bounds for the gradient estimates. Using Lemma \ref{appen01}, for any $R>2R_1$ there exists $C>0$ such that
\[|w'_{R_2}(r)|,\quad|z'_{R_2}(r)|\le C,
\quad \text{for}~~r\in[2R_1,R],\]
where $R_2>R.$ Since $\underline{u}_{R_2}(R_1)=w_{R_2}(R_1)=\bar u(R_1),$ we have
\[0\ge\underline{u}_{R_2}'(R_1)\ge w_{R_2}'(R_1)\ge \bar u'(R_1),\]
which provides uniform control of the derivatives of $w_{R_2}'(R_1)$ and $z_{R_2}'(R_1)$ at $R_1$. Hence, again in view of Lemma \ref{appen01}, we get
\[|w'_{R_2}(r)|\le C,\quad|z'_{R_2}(r)|\leq C\quad\text{for}~~r\in[R_1,2R_1].\]
Thus by passing the limit we get, for a subsequence $R_2\to+\infty$
\[
\left\{
\begin{aligned}
\lim_{R_2\to+\infty}w_{R_2}=w~~~\text{in}~~C^1_{\mathrm{loc}}[R_1,+\infty),\\ \lim_{R_2\to+\infty}z_{R_2}=z~~~\text{in}~~C^1_{\mathrm{loc}}[R_1,+\infty)
\end{aligned}
\right.\]
where $(w,z)$ is a nonnegative, radial solution of system~\eqref{2.9} in $\mathbb{R}^n\setminus \overline{B}_{R_1}$.
Since the convergence holds up to $\partial B_{R_1}$, we have
$w(R_1)=\bar u(R_1),~~z(R_1)=\bar v(R_1),$ which gives that $(w,z)$ is nontrivial. By strong maximum principle, $(w,z)$ is strictly positive. This completes the proof.
\end{proof}
\section{lower estimates}\label{lower}
In this section we establish lower bound for the positive, bounded, radially symmetric supersolution of the equation $- \mathcal{M}^+(D^2w)+|\nabla w|^{q}.$ As we will be dealing with radial solution of the considered operator, it is useful to recall the following radial version of Pucci's extremal operator
\[M^{+}_{\lambda,\Lambda}\left(D^2 \tilde{u}\right)(r)=\theta\left(\tilde{u}^{\prime \prime}(r)\right)\tilde{u}^{\prime \prime}(r)+\theta\left(\tilde{u}^{\prime}(r)\right)(n-1)\frac{\tilde{u}^{\prime}(r)}{r},
\]
where $\tilde{u}$ is a radial function. Thus, if $w\in C^2(R_0,+\infty)$ is a radial solution of $-\mathcal{M}^+(D^2(\cdot))+|D(\cdot)|^{q}>0,$ then 
\[-\theta\big(w^{\prime\prime}(r)\big)w^{\prime\prime}(r)
-\theta\big(w^{\prime}(r)\big)(n-1)\frac{w^{\prime}(r)}{r}+|w^{\prime}|^q>0~~\text{in}~~(R_0,+\infty).
\]
First observe that $w$ can not be constant. Then it is easy to see that if $w'\le 0,$ then necessarily $w'<0$. In fact, if $w'$ vanishes at some $r_0>R_0$, then above equation implies $w^{\prime\prime}(r_0)<0.$ This means that $w$ is non-increasing function which has maximum at some point $r_0.$ This is only possible if $w$ is constant, which is not the case. Thus $w^{\prime}<0.$
\begin{lemma}\label{lem5}
Suppose $q>1$ and $w\in W^{2,\infty}_{\mathrm{loc}}\big(R_0,+\infty\big)
\cap C^{1}\big(R_0,+\infty\big)$ be a positive function satisfying
\begin{equation}\label{5.1lem5}
\left\{
\begin{aligned}
&-\theta\left(w''(r)\right)w''(r)-\theta\left(w'(r)\right)(n-1)\frac{w'(r)}{r}+|w'(r)|^{q}\ge0~~\text{a.e. in }(R_0,+\infty),\\
&\text{with}~~\lim_{r\to+\infty}w(r)=0,~~w'(r)<0.
\end{aligned}
\right.
\end{equation}
Assume also that $\widetilde n_{+}>2$. There exists a positive constant $C$ such that, for all $r\ge R_0$ and every $1<q<q_c$,
\[
|w'(r)|\ge Cr^{-\frac{1}{q-1}},~~
w(r)\ge Cr^{-\frac{2-q}{q-1}}.
\]
\end{lemma}
Before proving  the lemma \ref{lem5}, we introduce a structural inversion lemma that will be used in the proof of lemma \ref{lem5} .
\begin{lemma}\label{inversion lemma}
Let $I\subset\mathbb R$ be an interval and let $\xi\in W^{2,s}_{\mathrm{loc}}(I)$ for some $s\geq1$.
Suppose there exists a function
$f\in L^{s}_{\mathrm{loc}}(I)$
such that
\[
\theta(\xi''(r))\,\xi''(r)\le f(r),~~
\text{for a.e.}~~r\in I,
\]
where $\theta$ is defined by \eqref{theta12}. Then $\xi''(r)\le\frac{f(r)}{\Lambda},~~\text{for a.e.}~~r\in I.$
\end{lemma}
\begin{proof}
Since \[\theta(\xi'')=
\begin{cases}
\Lambda,& \xi''\geq0,\\
\lambda,& \xi''<0,
\end{cases}\] we distinguish two cases.\\
\begin{itemize}
\item[\textbf{Case 1.}] Suppose that $\xi''(r)\ge0$. Then $\theta(\xi''(r))=\Lambda,$
and therefore $\Lambda\xi''(r)\le f(r).$ Since $\Lambda>0$, dividing by $\Lambda$ gives $\xi''(r)\le\frac{f(r)}{\Lambda}.$\\
\item[\textbf{Case 2.}] Suppose that $\xi''(r)<0$. Then $\theta(\xi''(r))=\lambda,$ and hence
$\lambda\xi''(r)\le f(r).$ Since $\lambda\le\Lambda$ and $\xi''(r)<0$, we have $\Lambda\xi''(r)\le\lambda\xi''(r).$ Consequently,
\[
\Lambda\xi''(r)\leq\lambda\xi''(r)\leq f(r),
\]
which immediately implies $\xi''(r)\le\frac{f(r)}{\Lambda}.$ Since one of the above two cases holds almost everywhere in $I$, we conclude that
\[
\xi''(r)\le\frac{f(r)}{\Lambda},~~\text{for almost every}~~r\in I.
\]
This completes the proof.
\end{itemize}
\end{proof}
\begin{proof}[\textbf{Proof of Lemma \ref{lem5}}]
Define $v(r):=-w'(r),~~ r>R_0.$ Since $w\in W^{2,\infty}_{\mathrm{loc}}(R_0,\infty)$ and $w'(r)<0$, we have $v(r)>0$ for every $r>R_0$. Also,
$v\in W^{1,\infty}_{\mathrm{loc}}(R_0,\infty).$ Moreover,
the differential inequality satisfied by $w$ is
\[
-\theta(w'')w''-\lambda(n-1)\frac{w'}{r}+|w'|^{q}\ge0.
\]

Since $v=-w',~~v'=-w'',$ we obtain $\theta(-v')v'+\lambda(n-1)\frac{v}{r}+v^{q}\ge0.$ Hence,
\[
\theta(-v')v'\ge-\lambda(n-1)\frac{v}{r}-v^{q}.
\]
Applying the Structural Inversion Lemma \ref{inversion lemma} with $\xi=v,$ we obtain
$v'\ge-\frac{\lambda(n-1)}{\Lambda}\frac{v}{r}
-\frac{1}{\Lambda}v^{q}~~\text{for a.e.}~~r>R_0.$ Equivalently
\begin{equation}\label{5.2lem5}
v'(r)+\frac{\widetilde n_{+}-1}{r}v(r)
\ge-\frac1{\Lambda}v(r)^q,~~\text{for a.e.}~~r>R_0.
\end{equation}
Since $q>1$, define $Y(r):=v(r)^{1-q}.$ Because $1-q<0$, the chain rule gives
\[
Y'(r)=(1-q)v(r)^{-q}v'(r)=-(q-1)v(r)^{-q}v'(r)
\]
for almost every $r>R_0$. From \eqref{5.2lem5},
$v'(r)\ge-\frac{\widetilde n_{+}-1}{r}v(r)-\frac1{\Lambda}v(r)^q.$ Multiplying both sides by the negative quantity
$-(q-1)v(r)^{-q}$ reverses the inequality, yielding
\[
Y'(r)\le(q-1)\frac{\widetilde n_{+}-1}{r}v(r)^{1-q}+\frac{q-1}{\Lambda}.
\]
Since $Y=v^{1-q},$ we obtain
\begin{equation}\label{5.3lem5}
Y'(r)\le\frac{\gamma}{r}Y(r)+\frac{q-1}{\Lambda},
\end{equation}
where $\gamma=(q-1)(\widetilde n_{+}-1).$ Since $q<q_c=\frac{\widetilde n_{+}}{\widetilde n_{+}-1},$ we have
$\gamma=(q-1)(\widetilde n_{+}-1)<1.$ Multiplying \eqref{5.3lem5} by the integrating factor
$r^{-\gamma}$
gives
\[
\left(r^{-\gamma}Y(r)\right)'\le\frac{q-1}{\Lambda}r^{-\gamma}.
\]
Integrating over $[R_0,r]$ yields
\[
r^{-\gamma}Y(r)-R_0^{-\gamma}Y(R_0)\le\frac{q-1}{\Lambda}
\int_{R_0}^{r}s^{-\gamma}ds.
\]
Since $\gamma<1$,
\[
\int_{R_0}^{r}s^{-\gamma}ds=\frac{r^{1-\gamma}-R_0^{\,1-\gamma}}{1-\gamma}.
\]
Therefore,
\[
r^{-\gamma}Y(r)\leq R_0^{-\gamma}Y(R_0)+
\frac{q-1}{\Lambda(1-\gamma)}\left(r^{1-\gamma}-R_0^{\,1-\gamma}\right).
\]
Multiplying both sides by $r^\gamma$, we obtain
\[
Y(r)\leq R_0^{-\gamma}Y(R_0)r^\gamma+\frac{q-1}{\Lambda(1-\gamma)}
\left(r-r^\gamma R_0^{\,1-\gamma}\right).
\]
Since $\gamma<1$, we have $r^\gamma\le r$ for every
$r\ge R_0$. Hence there exists a constant
$C_1>0$, depending only on
$R_0,\Lambda,q$ and $Y(R_0)$, such that $Y(r)\le C_1r,~~\text{for}~~r\ge R_0.$ Recalling that $Y(r)=v(r)^{1-q},$ we obtain $v(r)^{1-q}\le C_1r.$
Since $1-q<0$, raising both sides to the power
$-\frac{1}{q-1}$ gives $v(r)\geq Cr^{-\frac1{q-1}},$
where $C>0.$ Since $v=-w'$, this proves
\[
|w'(r)|=-w'(r)\geq Cr^{-\frac1{q-1}},~~r\ge R_0.
\]
It remains to estimate $w$. Since
$\lim_{r\to+\infty}w(r)=0$ and $w$ is absolutely continuous on every bounded interval, the Fundamental Theorem of Calculus gives
\[
w(r)=\int_r^\infty\bigl(-w'(t)\bigr)dt=\int_r^\infty v(t)dt\ge C\int_r^{\infty}t^{-\frac{1}{q-1}}dt =Cr^{-\frac{2-q}{q-1}},
\]
since $v(t)\ge Ct^{-\frac{1}{q-1}}.$ 
\end{proof}
\begin{lemma}\label{lem6}
Suppose that $q>1$, $C>0$ and $u$ is a positive function satisfying
\begin{equation}\label{3.4}
-\mathcal{M}^+(D^2u)+|\nabla u|^{q}\ge C|x|^{-\gamma-2}~~\text{in}~~\mathbb{R}^n\setminus B_{R_0},
\end{equation}
for some $\gamma$ in the range $\beta\le\gamma<n-2$. Then there exists
$C'>0$ such that
\[u(x)\ge C'|x|^{-\gamma}~~\text{for all }~~|x|>R_0.\]
\end{lemma}
\begin{proof} As $\beta\leq \gamma<n-2,$ we can choose $A_1$ sufficiently small so that $\Psi(x)=A_1|x|^{-\gamma}+A_2$ becomes a subsolution. In fact any $0<A_{1}\leq \overline{A},$ where 
\[
\gamma\overline{A}\Big[\lambda (n-1)-\Lambda(\gamma+1)\Big]
+\gamma^q\overline{A}^q R_0^{-q(\gamma+1)+\gamma+2}=C.
\]
will work. Now observe that in order to prove the conclusive inequality, it is sufficient to show that $\liminf_{|x|\to +\infty} m_u(|x|)\,|x|^{\gamma}>0.$ Now suppose by contradiction it does not hold, so we can find a sequence $(R_k)$ with $R_k \to +\infty$ such that $m_u(R_k)R_k^{\gamma}\to 0.$ Choose $k_0$ such that $m_u(R_{k_0})R_{k_0}^{\gamma}<\overline{A}.$ Define a function
\[\xi(x)=\frac{m_u(R_{k_0})-m_u(R_k)}{R_{k_0}^{-\gamma}-R_k^{-\gamma}}\left(|x|^{-\gamma} -R_k^{-\gamma}\right)+ m_u(R_k).\]
Notice that for sufficiently large $k,$
\begin{equation}\label{supp}
\begin{cases}
-\mathcal{M}^+(D^2\xi)+|D\xi|^{q}\leq C|x|^{-\gamma-2}&\text{in}~~\mathcal{A}(R_{n_0}, r_k),\\
\xi=m_u&\text{on}~~\partial\mathcal{A}(R_{k_0}, r_k).
\end{cases}.
\end{equation}
By the comparison principle, we obtain $\xi \le m_u~\text{in}~~\mathcal{A}(R_{k_0}, R_k).$ Letting, $k \to +\infty$ we get
\[m_u(|x|)|x|^{\gamma}\ge m_u(R_{n_0})R_{n_0}^{\gamma}~~\text{in}~~|x|>R_{n_0}.
\]
Finally choosing $x_j$ satisfying $|x_j|=R_j$ and $u(x_j)=m_u(|x_j|)$ and taking limit $j \to+\infty$ we get a contradiction. 
Therefore, $\liminf_{R \to +\infty} m_u(R)\, R^{\gamma} > 0.$  Which gives $u(x) \ge C' |x|^{-\gamma}.$
\end{proof}
Our next lemma is concerned with the case $\gamma\leq\beta.$
\begin{lemma}\label{lem7}
Assume that $1<q<2$, $C>0$ and $u$ is a positive function satisfying
\begin{equation}\label{3.5}
-\mathcal{M}^+(D^2u)+|\nabla u|^{q}\ge C|x|^{-(\gamma+1)q}~~\text{in}~~\mathbb{R}^n \setminus B_{R_0},
\end{equation}
with $\gamma\leq\beta$. Then there exists $C'>0$ such that 
\[
u(x)\ge C'|x|^{-\gamma}~~\text{for all} ~~|x| > R_0.
\]
\end{lemma}
\begin{proof}
The proof of this part also follows on the same line as Lemma \ref{lem6}. So will only give required changes. Since $\gamma \leq\beta$ the function $\Psi(x)=A_1 |x|^{-\gamma}+A_2$ is a subsolution of \eqref{3.5} if $0<A_1\le \overline{A}$, where $\overline{A}$ satisfies
\[\gamma\overline{A}\big[\lambda (n-1)-\Lambda(\gamma+1)\big]+\gamma^q\overline{A}^q=C.\]
Notice that such a choice of $\overline{A}$ is possible.
\end{proof}
\begin{remark}
When
\[-\mathcal{M}^+(D^2u)\ge C|x|^{-n}~~\text{in}~~\mathbb{R}^n\setminus B_{R_0},\]
in similar way can show, as in Lemma~\eqref{lem6}, that $u(x)\ge C|x|^{2-n}\log|x|$ for $|x|$ large enough and for some $C>0$. In this case, one considers the function $\xi(x)= A_1|x|^{2-n}\log|x|$ which is a subsolution of the equation for a suitable choice of $A_1$ (see~ \cite{Alessandra}). We will be using the above lower estimate in the proof of \ref{thm5}.
\end{remark}
\section{Upper estimates}
In the previous section we have obtained the lower for the solution of system \eqref{main}. In this section we establish the upper bound on the solution of \eqref{main}. It will be obtained in two cases separately. 
\subsection{}\textbf{Upper estimates in the case $\bm{q<q_c}:$}
This subsection deals with the upper bound on the solution of \eqref{main} under the case $q<q_{c}.$ Although the proof of the related results follows on the similar line as in \cite{burgos2018lioville}, but the solution in our case is not $C^{2}$ instead these belongs to $W^{2,s}_{\mathrm{loc}}.$ So we can not directly adopt the technique of \cite{burgos2018lioville} as it involves the evaluation of $w^{\prime\prime},z^{\prime\prime}$ at arbitrary points which is not possible for Sobolev functions. We will try to use the fact that these derivative exists almost everywhere as $w^{\prime},z^{\prime}$ are absolutely continuous.
\begin{lemma}\label{lem8}
Let $\widetilde{n}_{+}\geq2$ and $1<q<q_{c}$. Assume also that $w,z\in W^{2,s}_{loc}\big(R_0,+\infty\big)\cap
C^{1}\big(R_0,+\infty\big)$ be bounded positive functions satisfying 
\begin{equation}\label{4.1}
\begin{cases}
-\theta(w''(r))w''(r)-\theta(w'(r))(n-1)\frac{w'(r)}{r}+|w'|^q= c z^{p_1}&\text{a.e. in}~(R_0,+\infty) \\
-\theta(z''(r))z''(r)-\theta(z'(r)) (n-1)\frac{z'(r)}{r}+|z'|^q=d w^{p_{2}}&\text{a.e. in}~(R_0,+\infty),
\end{cases}
\end{equation}
where $p_1,p_2>0$ and $c,d>0$. Then,
\begin{itemize}
\item[(a)] If $p_1p_2<q^{2}$, then \eqref{4.1} admits no positive solution.
\item[(b)] If $p_{1}p_{2}=q^{2},$ there exist positive constants $C_1,C_2>0$ such that
\[
w\le C_1r^{-\frac{p_1}{q}}e^{-C_2r},~~z\le C_1r^{-\frac{p_{2}}{q}}e^{-C_2 r},~~\text{for}~~~r\ge R_0.
\]
\item[(c)] If $p_{1}p_{2}>q^{2},$ then there exists a positive constant $C$ such that
\[w\leq Cr^{-\alpha_{1}},~~z\leq Cr^{-\alpha_{2}}~~\text{for}~~~r\ge R_0.\]
\end{itemize}
\end{lemma}
\begin{proof}
We start by observing that since $z\geq0$ so in view of Remark \ref{monotone} we get $w^{\prime}\leq 0$ pointwise as the function is $C^{1}.$ Also in view of $(q-1)|t|^{2}+(2-q)|t|\ge|t|^{q}$ for all $t\in\mathbb{R},$ the first equation in \eqref{4.1} can be rewritten as 
\[-\theta(w''(r))w''(r)-\theta(w'(r))(n-1)\frac{w'(r)}{r}+(q-1)|w^{\prime}|^{2}+(2-q)|w^{\prime}|\geq c_{1} z^{p_1}~~\text{a.e. in}~(R_0,+\infty)\]
Now define $u(r):=\frac{\lambda}{q-1}\left(1-e^{-\frac{(q-1)w(r)}{\lambda}}\right),$ and setting 
\[U(r)=\frac{u''(r)}{1-\frac{(q-1)u(r)}{\lambda}}+\frac{q-1}{\lambda}\frac{|u'(r)|^{2}}
{\left(1-\frac{(q-1)u(r)}{\lambda}\right)^{2}},\] we get 
\begin{equation}\label{u-final}
-\theta\bigl(U(r)\bigr)u^{\prime\prime}(r)-\left(
\frac{(n-1)\theta\bigl(u^{\prime}(r)\bigr)}{r}+(2-q)\right)u'(r)\geq cz(r)^{p_{1}}
\left(1-\frac{(q-1)u(r)}{\lambda}\right)\ge0.
\end{equation}
So by applying the strong maximum principle, we find $u^{\prime}<0$ which is equivalent to $w^{\prime}<0.$ Similarly, we can also show that $z^{\prime}<0.$ Now, our aim is to show that
\begin{equation}\label{aim}
\begin{cases}
-w'(r)\ge\nu z(r)^{p_1/q}&\text{for}~~ r\ge R_1 \\
-z'(r)\ge\nu w(r)^{p_2/q}&\text{for}~~ r\ge R_1 .
\end{cases}
\end{equation}
In order to prove this we define the following
\[
\begin{cases}
\mathcal H_1(r)=cz(r)^{p_1}-a|w^{\prime}|^{q}\\
\mathcal H_2(r)=dw(r)^{p_2}-a|z^{\prime}|^{q}.\\
\end{cases}
\]
We want to show that there exists an $R_1$ such that for $r\geq R_{1},$ we have $\mathcal H_i(r)\leq 0$ for $i=1,2.$ Since the proof of both the cases are same. We prove the result for $\mathcal H_1(r)$ the result for $\mathcal H_2(r)$ follows on the same way. We first prove that if there exists $r_0>R_0$ such that $\mathcal H_{1}(R_0)\leq0,$ then
\[\mathcal H_{1}(r)\le0\qquad\text{for all}~r\ge R_0.\]
Suppose by contradiction that there exists $r_{1}>R_0$ such that $\mathcal H_{1}(r_{1})>0.$ Since $\mathcal H_{1}$ is continuous, the set
\[O:=\{r\in[R_0,r_{1}]:\mathcal H_{1}(r)>0\}\]
is a nonempty open subset of $[R_0,r_{1}].$ Also by the continuity of $\mathcal H_{1},$ we find that $\mathcal H_{1}(\rho)=0,$ where $\rho:=\inf~O.$  Since $O$ is open, there exists $\eta>0$ such that $(\rho,\rho+\eta)\subset~O.$ Consequently,
\[\mathcal H_{1}(r)>0\qquad\text{for all }r\in(\rho,\rho+\eta).\]
Set
\[\begin{aligned}
\begin{cases}
&Y(r):=-w'(r)>0,\\
&V(r):=\frac{z(r)^{p_{1}}}{Y(r)^{q}},\\
&\text{observe that}~\mathcal H_{1}(r)>0~~\Longleftrightarrow~~cV(r)>a.
\end{cases}
\end{aligned}
\]
Let us also assume that there exist constants $a^{*}>1$ and $\varepsilon>0$
depending only on the structural data such that whenever $cV(r)\ge a^{*},$ $\frac{V'(r)}{V(r)}\le-\frac{\varepsilon}{r}~~\text{a.e.}$ Choosing $a\ge a^{*}$, we obtain on $(\rho,\rho+\eta)$
\[\frac{V'(r)}{V(r)}\leq-\frac{\varepsilon}{r}\qquad\text{a.e.}\]
Since $V\in AC_{\mathrm{loc}},$ the fundamental theorem of calculus yields
\[\log\frac{V(r)}{V(s)}=\int_{s}^{r}\frac{V'(\tau)}{V(\tau)}d\tau\leq-\varepsilon\int_{s}^{r}\frac{d\tau}{\tau}<0\]
for every $\rho<s<r<\rho+\eta.$ Hence $V(r)<V(s)$ for $\rho<s<r<\rho+\eta,$ so $V$ is strictly decreasing
on $(\rho,\rho+\eta)$. Passing to the limit $s\downarrow\rho$, and using $cV(\rho)=a,$ we obtain
\[cV(r)<a\qquad\text{for all}~~r\in(\rho,\rho+\eta).\]
Therefore
\[\mathcal H_{1}(r)=c z(r)^{p_{1}}-aY(r)^{q}<0\]
on $(\rho,\rho+\eta),$ which contradicts the definition of $O$. The contradiction proves that
\[\mathcal H_{1}(r)\le0\qquad\text{for all }r\ge R_0.\]
\textbf{Proof of the assumption:}~Since $z\in W^{2,s}_{\mathrm{loc}},~~Y=-w'\in W^{1,s}_{\mathrm{loc}},$ and $Y(r)>0$ for $r>R_1,$ consequently, $V\in W^{1,s}_{\mathrm{loc}}((R_1,\infty)).$ Hence $V$ is absolutely continuous so for a.e. for $r>R_1$ we have
\[\frac{V'(r)}{V(r)}=p_1\frac{z'(r)}{z(r)}-q\frac{Y'(r)}{Y(r)}.\]
From the first equation of \eqref{4.1},
\[\theta(w''(r))Y'(r)+\lambda(n-1)\frac{Y(r)}{r}+Y(r)^q=cz(r)^{p_1},\]
and therefore $Y'(r)=\frac{cz(r)^{p_1}-Y(r)^q-\lambda(n-1)\frac{Y(r)}{r}}{\theta(w''(r))}.$  Since $cz(r)^{p_1}=cV(r)Y(r)^q,$ we obtain
\[Y'(r)=\frac{Y(r)^q(cV(r)-1)-\lambda(n-1)\frac{Y(r)}{r}}{\theta(w''(r))}.\]
Dividing by $Y(r),$
\[\frac{Y'(r)}{Y(r)}=\frac{Y(r)^{q-1}(cV(r)-1)-\lambda(n-1)r^{-1}}{\theta(w''(r))}.\]
Substituting into the logarithmic derivative gives
\begin{equation}\label{log-derivative-V}
\frac{V'(r)}{V(r)}=p_1\frac{z'(r)}{z(r)}-q\frac{Y(r)^{q-1}(cV(r)-1)-\lambda(n-1)r^{-1}}{\theta(w''(r))}.
\end{equation}
Since $\lambda\le\theta(w''(r))\le\Lambda,$ it follows that whenever $cV(r)-1>0,$ we have
\[-\frac{q\,Y(r)^{q-1}(cV(r)-1)}{\theta(w''(r))}\leq-\frac{q}{\Lambda}Y(r)^{q-1}(cV(r)-1),\]
and $q\frac{\lambda(n-1)r^{-1}}{\theta(w''(r))}\leq~q(n-1)\frac{1}{r}.$ Consequently,
\begin{equation}\label{log-derivative-upper}
\frac{V'(r)}{V(r)}\leq p_1\frac{z'(r)}{z(r)}-\frac{q}{\Lambda}Y(r)^{q-1}(cV(r)-1)+q(n-1)\frac{1}{r}.
\end{equation}
Noting that $\frac{z'(r)}{z(r)}<0$ and $Y^{q-1}\geq\frac{C}{r},$ we can find a constant $a^{*}$ such that the assumption becomes true.
Next, we claim that there exists $r_0>R_0$ such that $\mathcal H_1(R_0)\leq0.$ Indeed, from \eqref{log-derivative-upper} and the estimate $Y(r)^{q-1}\ge C_0 r^{-1}$, we obtain
\[\frac{V'(r)}{V(r)}\le\frac{1}{r}\Bigl[q(n-1)-\frac{qC_0}{\Lambda}(cV(r)-1)\Bigr].\]
Set $a^\ast:=1+\frac{\Lambda}{qC_0}\Bigl(q(n-1)+1\Bigr).$ Then whenever $cV(r)\ge a^*$, $\frac{V'(r)}{V(r)}\leq-\frac{1}{r}.$
Assume by contradiction that
\[V(r)\ge \frac{a^*}{c}\qquad\text{for all}~~r\ge R.\]
Integrating the above inequality over $(R,r)$ yields
\[\log\frac{V(r)}{V(R)}\le-\int_R^r \frac{ds}{s}=-\log\frac rR,\]
hence $V(r)\le V(R)\frac{R}{r}.$ Letting $r\to\infty$, we obtain $V(r)\to0$, which contradicts
$V(r)\ge a^*c^{-1}>0$.Therefore there exists $R_0\ge R$ such that $V(R_0)\le \frac{a^*}{c}.$ Equivalently, \[\mathcal H_1(R_0)=cz(R_0)^{p_1}-a^*Y(R_0)^q\le0.\]
By the argument established above, once $\mathcal H_1$ becomes
nonpositive it cannot return to positive values. Hence
\[\mathcal H_1(r)\leq0\qquad\text{for all}~r\ge R_0.\]
By taking $R_{1}=R_0$ we find that \eqref{aim} holds.\\
Now, we consider the change of variables defined by $W(t)=w(r)$ and $Z(t)=z(r)$, where
\[t=\mathcal{H}(r):=\int_r^{+\infty}\nu z^{\frac{p_1}{q}}(\rho)d\rho.\]
Then the function $\mathcal{H}$ is well-defined because
\begin{equation}\label{eq:5.3}
\int_r^{+\infty}\nu z^{\frac{p_1}{q}}(\rho)\ d\rho\le \int_r^{+\infty}-w'(\rho)\ d\rho=w(r)<+\infty.   
\end{equation}
Note that the functions $W$ and $Z$ are defined on $(0,\delta)$, where $\delta>0$, and satisfy $W(0)=Z(0)=0$. Moreover, by \eqref{eq:5.3}, it can be seen that $W(t)>t$. Moreover,
\[
\begin{aligned}
Z'(t)&=\frac{dz}{dr}.\frac{dr}{dt}= -\frac{z'(r)}{\nu z^{\frac{p_1}{q}}(r)}>\frac{w^{\frac{p_2}{q}}(r)}{z^{\frac{p_1}{q}}(r)}=\frac{W^{\frac{p_2}{q}}(t)}{ Z^{\frac{p_1}{q}}(t)}>\frac{t^{\frac{p_2}{q}}} {Z^{\frac{p_1}{q}}(t)},
\end{aligned}
\]
since $ W(t)=w(r)>t.$ Hence $Z^{\frac{p_1}{q}}(t) Z'(t) > t^{\frac{p_2}{q}},$ and integrating from $0$ to $t$ we have,
\[
\begin{aligned}
\begin{cases}
\int_{0}^{t}Z^{\frac{p_1}{q}}(\rho) Z'(\rho)d\rho=\frac{q}{p_1+q}\int_{0}^{t}\frac{d}{d\rho}\left(Z^{\frac{p_1+q}{q}}(\rho)\right)d\rho>\int_{0}^{t}t^{\frac{p_2}{q}}\\
\left(Z^{\frac{p_1+q}{q}}(t)-Z^{\frac{p_1+q}{q}}(0)\right)>\frac{p_1+q}{q}\frac{q}{p_2+q}.t^{\frac{p_2+q}{q}}\\
Z^{\frac{p_1+q}{q}}(t)>\underbrace{\frac{p_1+q}{p_2+q}}_{:=C}.t^{\frac{p_2+q}{q}}\\
Z(t)\geq Ct^{\frac{p_2+q}{p_1+q}},
 \end{cases}
\end{aligned}
\]
where we used $Z(0)=0$ in the second-to-last inequality. Returning to the original variables, we have $z(r)\geq C\mathcal{H}^{\frac{p_2+q}{p_1+q}}(r).$
Since $\mathcal{H}'(r)=-\nu w^{\frac{p_1}{q}}(r)$, we obtain
\[-\mathcal{H}'(r)=\nu w^{\frac{p_1}{q}}(r)\ge C\mathcal{H}^{\frac{p_1(P_2+q)}{q(p_1+q)}}(r),\]
which implies $-\mathcal{H}^{-{\frac{p_1(P_2+q)}{q(p_1+q)}}(r)}\mathcal{H}'(r)\geq C.$
Integrating between $R_0$ and $r$,
\begin{equation}\label{eq:5.4}
-\int_{R_0}^{r} \mathcal{H}^{-\frac{p_1(p_2+q)}{q(p_1+q)}}(\rho) \mathcal{H}'(\rho) d\rho\ge C (r-R_0).
\end{equation}
We consider three different cases according to the position of the exponent $\frac{p_1(p_2+q)}{q(p_1+q)}$ and $1.$\\
\begin{itemize}
 \item[\textbf{(a):}]~Assume that $\frac{p_1(p_2+q)}{q(p_1+q)}<1,$ which is equivalent to $p_1p_2<q^2$. Then \eqref{eq:5.4} gives  
\[
\begin{aligned}
\begin{cases}
-\int_{R_0}^{r} \mathcal{H}^{-\frac{p_1(p_2+q)}{q(p_1+q)}}(\rho) \mathcal{H}'(\rho)\, d\rho=-\frac{1}{1-\frac{p_1(p_2+q)}{q(p_1+q)}}\int_{R_0}^{r} \frac{d}{d\rho}\left(\mathcal{H}^{1-\frac{p_1(p_2+q)}{q(p_1+q)}}(\rho)\right) d\rho\ge C(r-R_0)\\
-\frac{1}{1-\frac{p_1(p_2+q)}{q(p_1+q)}}\left(\mathcal{H}^{1-\frac{p_1(p_2+q)}{q(p_1+q)}}(r)-\underbrace{\mathcal{H}^{1-\frac{p_1(p_2+q)}{q(p_1+q)}}(R_0)}_{C_1}\right)\ge C(r-R_0) \\
-\frac{\mathcal{H}^{1-\frac{p_1(p_2+q)}{q(p_1+q)}}(r)-C_1}{1-\frac{p_1(p_2+q)}{q(p_1+q)}}\ge C(r-R_0), \\
\end{cases}
\end{aligned}
\]
where $C_1>0$ is different positive constant. Letting $r\to+\infty$ and using the fact that $\mathcal{H}(r)\to 0,$ we obtain a contradiction. Hence, $(a)$ proved.
\item[\textbf{(b):}]~Assume that $\frac{p_1(p_2+q)}{q(p_1+q)}=1,$ which is equivalent to $p_1p_2=q^2$. Then \eqref{eq:5.4} gives 
\[
\begin{aligned}
\begin{cases}
-\int_{R_0}^{r} \mathcal{H}^{-1}(\rho) \mathcal{H}'(\rho)\, d\rho=-\int_{R_0}^{r} \frac{d}{d\rho}(\log \mathcal{H}(\rho)) d\rho\ge C(r-R_0)\\
-\log \mathcal{H}(r)+\log \mathcal{H}(R_0)\ge C(r-R_0)\\
\mathcal{H}(r)\le e^{-Cr}.
\end{cases}
\end{aligned}
\] 
Moreover,
\[
\mathcal{H}(r)
=\int_{r}^{+\infty} \nu z^{\frac{p_1}{q}}(\rho)\, d\rho\ge\nu\int_{r}^{2r}  z^{\frac{p_1}{q}}(\rho)\, d\rho\ge \nu r\, z^{\frac{p_1}{q}}(2r).
\]
so that $\nu r z^{\frac{p_1}{q}}(2r)\leq \mathcal{H}(r)\leq e^{-C_2r},$
 which implies $z^{\frac{p_1}{q}}(2r)\le C r^{-1} e^{-C_2r}.$ A similar argument provides an upper bound for $w.$ Hence, $(b)$ follows. 
\item[\textbf{(c):}]
 Assume that $\frac{p_1(p_2+q)}{q(p_1+q)}>1,$ which is equivalent to $p_1p_2>q^2$. Then \eqref{eq:5.4} gives  
\[
\mathcal{H}^{1-{\frac{p_1(p_2+q)}{q(p_1+q)}}(r)}\ge Cr,
\]
hence $\mathcal{H}(r)\le C r^{-\frac{q(p_1+q)}{p_1p_2-q^2}}.$ Since $\mathcal{H}(r)\ge C r z^{\frac{p_1}{q}}(2r)$, we deduce
\[
z^{\frac{p_1}{q}}(2r)\le Cr^{-\frac{q(p_1+q)}{p_1p_2-q^2}-1}=Cr^{-\frac{p_1(p_2+q)}{p_1p_2-q^2}},
\]
thus $z(r)\leq C r^{-\frac{q(p_2+q)}{p_1p_2-q^2}}= C r^{-\alpha_{2}},$ where $\alpha_{2}=\frac{q(p_2+q)}{p_1p_2-q^2}.$
The estimate for $w$ is obtained similarly. This completes the proof.\\
\end{itemize}
\end{proof}
\subsection{}\textbf{Upper estimates when $\bm{q>q_{c}}:$}
 In this subsection, we obtain upper estimates for positive solutions of \eqref{main} which do not blow up at infinity and $q>q_{c}.$ It will be achieved by proper adaptation of the doubling lemma introduced in \cite{Polavcik}. Due to the symmetric nature of the problem we assume $p_1 \geq p_2.$ In addition to this we also assume that $p_{1}p_{2}>1.$ Furthermore, the following notation will also be helpful in formulating the problem:
 \[\bar{\alpha}_{1}=\frac{2(p_1+1)}{p_{1}p_{2}-1},~~\bar{\alpha}_{2}=\frac{2(p_2+1)}{p_{1}p_{2}-1}.\]
We would also like to recall that $\beta=\frac{2-q}{q-1}.$ As we mention the main result of this subsection is proved by using the doubling lemma followed by the rescaling of solution. In one of the cases the limiting equation turns out to be 
\begin{equation}\label{5.31}
\begin{cases}
-\mathcal{M}^+(D^2\bar{u})=c\bar{v}^{p_1}&\text{in}~~\mathbb{R}^{n}\\
-\mathcal{M}^+(D^2\bar{v})+|\nabla \bar{v}|^{q}=d \bar{u}^{p_2}&\text{in}~~ \mathbb{R}^{n}.
\end{cases}.
\end{equation}
We seek the conditions on the exponents under which above system does not have non-trivial solution. The above situation will arise when $\alpha_{2}=\beta.$ In the case of Laplace equation, above system has been studied in \cite{burgos2018lioville}. So our next result is an adaptation in the setting of Pucci's extremal operator. 
\begin{theorem}\label{thm5}
Suppose $\widetilde{n}_{+}\geq2,$ $q>q_{c}$ and $p_1p_2>1.$ Let $\bar{\alpha}_{1},\bar{\alpha}_{2}$ be defined as in \eqref{1.7}. If $\bar{\alpha}_{1} \ge \tilde{n}_{+}-2$ and $\bar{\alpha}_{2}=\beta,$ then the system \eqref{5.31}
admits no positive viscosity supersolutions for any $c,d>0.$
\end{theorem}
\begin{proof}
First we notice that there is no harm in assuming $p_{1}p_{2}>q.$ For if $p_{1}p_{2}\le q,$ then $p_{1}p_{2}-1\le q-1,$ consequently, we have $(p_{1}p_{2}-1)(2-q)\le (q-1)(2-q).$
Also in view of $q>1,$ we  have $(q-1)(2-q)<q-1<2(q-1)(p_{2}+1).$ Therefore combining above two cases, we have
\[
\bar{\alpha}_{2}=\frac{2(p_{2}+1)}{p_{1}p_{2}-1}>\frac{2-q}{q-1}=\beta,\]
which is inconsistent with the assumption. On one hand by following the same calculation as in Lemma~\ref{lem2}, we get 
\begin{equation}\label{5.4}
\begin{cases}
m_v^{p_1}(2R) \le C\frac{m_u(R)}{R^{2}},\\
m_u^{p_2}(2R) \le C\left(\frac{m_v(R)}{R^{2}}+\frac{m_v^{q}(R)}{R^{q}}\right),
\end{cases}
\end{equation}
where $m_{u},~m_{v}$ are defined by \eqref{min} and $R>0.$ On the other hand in view of Corollary 3.1 in \cite{Alessandra} the function $R^{\widetilde{n}_{+}-2}m_{u}(R)$ is increasing. Thus we find that $m_{u}(2R)\geq Cm_{u}(R)$ for some $C>0.$ Replacing $2R$ by $R$ in the first inequality of \eqref{5.4}, we obtain
\begin{equation}\label{5.5}
\begin{cases}
m_{v}^{p_{1}}(R)\leq C\,\frac{m_{u}(R)}{R^{2}},\\
m_{u}^{p_{2}}(R)\leq C\left(\frac{m_{v}(R)}{R^{2}}+\frac{m_{v}^{q}(R)}{R^{q}}
\right),
\end{cases}
\end{equation}
for sufficiently large $R.$ Now set
\[
\begin{cases}
\tilde{m}_{u}(R)=R^{\bar{\alpha}_{1}}m_{u}(R),\\
\tilde{m}_{v}(R)=R^{\bar{\alpha}_{2}}m_{v}(R).
\end{cases}
\]
Using \eqref{5.5} together with
\[
\bar{\alpha}_{1}+2=p_{1}\bar{\alpha}_{2},~~\bar{\alpha}_{2}+2=p_{2}\bar{\alpha}_{1},
\]
we get $m_{v}^{p_{1}}(R)\leq C\frac{m_{u}(R)}{R^{2}},$
that is, $R^{-p_{1}\bar{\alpha}_{2}}\tilde{m}_{v}^{p_{1}}(R)\leq C\,R^{-2-\bar{\alpha}_{1}}\tilde{m}_{u}(R).$ Since $p_{1}\bar{\alpha}_{2}=\bar{\alpha}_{1}+2,$ thus it follows that $\tilde{m}_{v}^{p_{1}}(R)\leq C\tilde{m}_{u}(R).$ 
Similarly, $\tilde{m}_{u}^{p_{2}}(R)\leq C\left(\tilde{m}_{v}(R)+\tilde{m}_{v}^{q}(R)\right).$ Combining the above two estimates yields \[\tilde{m}_{v}^{p_{1}p_{2}}(R)\le C\left(\tilde{m}_{v}(R)+\tilde{m}_{v}^{q}(R)\right).\] Thus in view of $p_{1}p_{2}>q>1,$ it follows that $\tilde{m}_{v}(R)$ is bounded for large $R$, and therefore $\tilde{m}_{u}(R)$ is also bounded. Thus we can find a positive constant $C>0$ such that, for $R$ sufficiently large
\[
\begin{cases}
m_{u}(R)\leq CR^{-\bar{\alpha}_{1}},\\m_{v}(R)\leq CR^{-\bar{\alpha}_{2}}
\end{cases}.
\] 
Moreover, since $R^{\widetilde n_+-2}m_u(R)$ is nondecreasing, there exists $C>0$ such that $m_u(R)\ge CR^{-(\widetilde n_+-2)}$ for all sufficiently large $R$.
Since
\[
-\mathcal M_{\lambda,\Lambda}^{+}(D^{2}v)+|\nabla v|^{q}
=du^{p_2},
\]
and $u(x)\ge m_u(|x|)
\ge C|x|^{-(\widetilde n_+-2)},$ it follows that $u(x)\geq C|x|^{-(\widetilde n_+-2)}.$
Hence
\[
-\mathcal M_{\lambda,\Lambda}^{+}(D^{2}v)
+|\nabla v|^{q}
\ge C|x|^{-\bar{\alpha}_{2}-2}.
\]
By Lemma~\ref{lem6} we have, $m_v(R)\ge CR^{-\bar\alpha_2}$ for all sufficiently large $R$.
Hence
\begin{equation}\label{5.6}
\begin{cases}
CR^{-(\widetilde{n}_{+}-2)}\leq m_{u}(R)\leq CR^{-\bar{\alpha}_{1}},\\
CR^{-\bar{\alpha}_{2}\leq m_{v}(R)}\leq CR^{-\bar{\alpha}_{2}},
\end{cases}
\end{equation}
for all sufficiently large $R.$ Thus, if $\bar{\alpha}_{1}>\widetilde{n}_{+}-2,$ then \eqref{5.6} gives an immediate contradiction. Therefore, now let us consider the case $\bar{\alpha}_{1}=\widetilde{n}_{+}-2.$ Again from \eqref{5.6}, we already know that for sufficiently large $R,$ $m_{u}(R)\leq CR^{-(\widetilde{n}_{+}-2)}.$ From \eqref{5.6}, we already know that for sufficiently large $R$
\begin{equation}\label{uppercritical}
m_{u}(R)\leq CR^{-(\widetilde{n}_{+}-2)}
\end{equation}
Moreover, from \eqref{5.6}, we also have $m_{v}(R)\geq CR^{-\bar{\alpha}_{2}}.$ Since $v(x)\ge m_{v}(|x|),$  it follows that for sufficiently large $|x|$ we have  $v^{p_{1}}(x)\geq C|x|^{-p_{1}\bar{\alpha}_{2}}.$ Using the identity $p_{1}\bar{\alpha}_{2}=
\bar{\alpha}_{1}+2,$ and $\bar{\alpha}_{1}=\widetilde{n}_{+}-2,$ we deduce that $p_{1}\bar{\alpha}_{2}=\widetilde{n}_{+}.$ Hence 
\[v^{p_{1}}(x)\geq C|x|^{-\widetilde{n}_{+}}~~\text{for}~~|x|\gg1.\]
Thus in view of above inequality and the first equation of \eqref{5.31}, we get
\begin{equation}\label{criticalineq}
-\mathcal M^{+}_{\lambda,\Lambda}(D^{2}u)\geq C|x|^{-\widetilde{n}_{+}}~~\text{in}~~\mathbb R^{n}\setminus B_{R_{1}},
\end{equation}
for some sufficiently large $R_{1}>0.$ Now, we apply the fully nonlinear logarithmic improvement for the
critical exponent. More precisely, if $w\ge0$ satisfies
\[-\mathcal M^{+}_{\lambda,\Lambda}(D^{2}w)\geq C|x|^{-\widetilde{n}_{+}}~~\text{in }~~\mathbb R^{n}\setminus B_{R_{1}},\]
then the Hadamard-type principle for supersolutions of Pucci equations implies the sharper lower bound
\[m_{w}(R)\ge CR^{-(\widetilde{n}_{+}-2)}\log R~~\text{for}~~R\gg1.\]
Applying this to $w=u,$  we conclude
\begin{equation}\label{lowercritical}
m_{u}(R)\geq CR^{-(\widetilde{n}_{+}-2)}\log R
\end{equation}
for all sufficiently large $R.$ This contradicts the upper estimate \eqref{uppercritical}, since $\log R\to+\infty$ as $R\to+\infty.$ Thus this case is also impossible and concludes the proof.
\end{proof}
\begin{lemma}\label{lem9}
Let $n\geq2$ and $q>q_{c}.$ Assume $\bar{\alpha}_1\ge\tilde{n}_{+}-2$ and $\bar{\alpha}_{2}\ge\beta$. If $(u,v)$ is a positive viscosity solution of 
\begin{equation}\label{5.1}
\begin{cases}
-\mathcal{M}_{\lambda,\Lambda}^+(D^2u)+|\nabla u|^{q}=c v^{p_1}&\text{in}~~\mathbb{R}^{n} \setminus B_{R_0},\\
-\mathcal{M}_{\lambda,\Lambda}^+(D^2v)+|\nabla v|^{q}=d u^{p_2}&\text{in}~~\mathbb{R}^{n} \setminus B_{R_0},
\end{cases}
\end{equation}
satisfying $\lim_{|x|\to\infty} u(x)=\lim_{|x|\to\infty} v(x)=0,$ then there exists a constant $C>0$ such that 
\[u(x)\le C|x|^{-\bar{\alpha}_{1}}~~\text{and}~~v(x) \le C|x|^{-\bar{\alpha}_{2}} \]
for all $|x|$ sufficiently large.
\end{lemma}
\begin{proof}
We assume that both $u$ and $v$ tend to zero as $|x|\to\infty,$ that is,
\begin{equation}\label{5.2}
\lim_{|x|\to+\infty} v(x)=\lim_{|x|\to+\infty} u(x)=0,
\end{equation}
implies that $u,v$ are bounded. By using the local $C^{1,\alpha}$ estimate for bounded solution we get $|Du|$ and $|Dv|$ are bounded in $\mathbb{R}^n \setminus B_{2R_0},$ see \cite{nornberg2019c1}. Moreover, by repeating the argument as in the proof of Lemma~\ref{lem3}, we can show that 
\[
\lim_{|x|\to+\infty} |\nabla u(x)|=\lim_{|x|\to+\infty}| \nabla v(x)|=0.
\]
We claim that if $(u,v)$ is a positive solution of \eqref{5.1} satisfying \eqref{5.2}, then we can find a constant $C>0$ such that for all $z$ with $|z|>2R_0$ and $x\in B\bigl(z,|z|/2\bigr)$ we have
\[
\begin{cases}
u(x)\le Cd_z(x)^{-\bar{\alpha}_{1}}\\v(x)\le Cd_z(x)^{-\bar{\alpha}_{2}}
\end{cases}
~~\text{and}~~\begin{cases}
|\nabla u(x)|\le Cd_z(x)^{-(\bar{\alpha}_{1}+1)}\\
|\nabla v(x)|\le C\, d_z(x)^{-(\bar{\alpha}_{2}+1)},
\end{cases}
\]
where $d_z(x):=\operatorname{dist}(x,\partial B(z,|z|/2)).$ We define a function as
\[
N(x)=u(x)^{\frac{1}{\bar{\alpha}_{1}}}+v(x)^{\frac{1}{\bar{\alpha}_{2}}}+|\nabla u(x)|^{\frac{1}{\bar{\alpha}_{1}+1}}
+|\nabla v(x)|^{\frac{1}{\bar{\alpha}_{2}+1}}.\]
By the definition of $N$, it suffices to show that there exists a constant $C>0$ such that
\[
N(x)d(x)\le C,
\]
for all $|z|>2R_0$ and $x\in B\left(z,\frac{|z|}{2}\right).$ Now suppose, by contradiction, that the claim is false. Then there exist sequences
$\{z_k\}_{k\ge1}$ and $\{y_k\}_{k\ge1}$ such that
\[
|z_k|>2R_0,~~y_k\in B\left(z_k,\frac{|z_k|}{2}\right),~~\text{and}~~N(y_k)>2kd^{-1}_{z_k}(y_k),
\]
for every $k\ge1$.
By Lemma~5.1 in \cite{Polavcik}, there exists a sequence $\{x_k\}_{k\ge1}$ such that \[x_k\in B(z_k,|z_k|/2),~~ N(x_k)>2kd^{-1}_{z_k}(x_k),~~\text{and}~~ N(y)\leq 2 N(x_k)~~\text{for}~~|y-x_k|\leq kN^{-1}(x_k).
\]
Since $u$, $v$, $|\nabla u|$ and $|\nabla v|$ are bounded we have $2kd^{-1}_{z_k}(x_k)\le C,$ that is, $2k\le Cd_{z_k}(x_k),$  hence $d_{z_k}(x_k)\to+\infty$ as $k\to+\infty$.
Therefore $|x_k|\to+\infty$, which implies
$N(x_k)\to0$ as all the functions involved in the definition of $N$ goes to zero at infinity. Set
$\varrho_k:=N(x_k)^{-1}\to+\infty,$ and define for $y\in B_k:=B(0,k)$,
\[
\begin{cases}
\bar u_k(y)=\varrho_k^{\bar{\alpha}_{1}}u(x_k+\varrho_k y),\\
\bar v_k(w)=\varrho_k^{\bar{\alpha}_{2}}v(x_k+\varrho_k y).
\end{cases}
\]
These functions are well defined and satisfy
\begin{equation}\label{5.8}
\bar u_k^{\frac1{\bar{\alpha}_{1}}}
+\bar v_k^{\frac1{\bar{\alpha}_{2}}}
+|\nabla \bar u_k|^{\frac1{\bar{\alpha}_{1}+1}}
+|\nabla \bar v_k|^{\frac1{\bar{\alpha}_{2}+1}}
\le 2~~\text{in}~~B_k,
\end{equation}
since $N(x)\le 2N(y_k)$ and
\[
\bar u_k^{\frac1{\bar{\alpha}_{1}}}(0)
+\bar v_k^{\frac1{\bar{\alpha}_{2}}}(0)
+|\nabla \bar u_k(0)|^{\frac1{\bar{\alpha}_{1}+1}}
+|\nabla \bar v_k(0)|^{\frac1{\bar{\alpha}_{2}+1}}
=1.
\]
Moreover, $(\bar u_k,\bar v_k)$ satisfies
\begin{equation}\label{5.9}
\begin{cases}
-\mathcal{M}^+(D^2\bar u_k)+\varrho_k^{\bar{\alpha}_{1}+2-(\bar{\alpha}_{1}+1)q}|\nabla\bar u_k|^q=c\bar v_{k}^{p_1}&\text{in}~~B_{k},\\
-\mathcal{M}^+(D^2\bar v_k)+\varrho_k^{\bar{\alpha}_{2}+2-(\bar{\alpha}_{2}+1)q}|\nabla\bar v_{k}|^q=d\bar u_{k}^{p_2}&\text{in}~~B_{k}
\end{cases}
\end{equation}
Our remaining analysis is divided into two cases based on the inequalities satisfied by $\bar{\alpha}_{2}.$
\begin{enumerate}
\item{}\textbf{$\bar{\alpha}_{2}>\frac{2-q}{q-1}:$} In this case $\bar{\alpha}_{1}+2-(\bar{\alpha}_{1}+1)q<0,$~$\bar{\alpha}_{2}+2-(\bar{\alpha}_{2}+1)q<0.$ In view of \eqref{5.8} and by using $C^{1,\alpha}_{\mathrm{loc}}$-estimate in \cite{nornberg2019c1} we find that $(\bar u_k,\bar v_k)$ is locally bounded in $C^{1,{\alpha}}$. Thus passing to subsequence, if necessary, we find that
 $(\bar u_k,\bar v_k) \longrightarrow (\bar u,\bar v)~~\text{in}~~C^1_{\mathrm{loc}}(\mathbb{R}^n).$ Moreover, by stability results of viscosity solution, $(\bar u,\bar v)$ also satisfies
\begin{equation}\label{5.10}
\begin{cases}
-\mathcal{M}^+(D^2\bar u) = c\bar v^{p_1}~~~\text{in}~~ \mathbb{R}^{n},\\
-\mathcal{M}^+(D^2\bar v) = d \bar u^{p_2}~~~\text{in}~~ \mathbb{R}^{n},
\end{cases}
\end{equation}
and
\[
\bar u^{\frac1{\bar{\alpha}_{1}}}(0)
+\bar v^{\frac1\alpha_{2}}(0)
+|\nabla \bar u(0)|^{\frac1{\bar{\alpha}_{1}+1}}
+|\nabla \bar v(0)|^{\frac1{\bar{\alpha}_{2}+1}}=1.
\]
Thus $(\bar u,\bar v)$ is nontrivial. 
Since $p_{1}\ge p_{2}$ and
\[
\bar{\alpha}_{1}=\frac{2(p_{1}+1)}{p_{1}p_{2}-1}\ge \widetilde{n}_{+}-2,\]
the Liouville theorem for the Lane--Emden fully nonlinear system associated with the Pucci extremal operator implies that
\[
\begin{cases}
-\mathcal{M}^{+}_{\lambda,\Lambda}(D^{2}\bar u)=c\,\bar{v}^{p_{1}}&\text{in}~~\mathbb R^{n}\\
-\mathcal{M}^{+}_{\lambda,\Lambda}(D^{2}\bar v)=d\,\bar{u}^{p_{2}}&\text{in}~~\mathbb R^{n},
\end{cases}
\]
admits no positive entire solutions, see Theorem 1.1 \cite{quaas2009existence}.
This contradicts the fact that $(\bar u,\bar v)$ is nontrivial.
\item{}\textbf{$\bar{\alpha}_{2}=\frac{2-q}{q-1}:$} In this case, system \eqref{5.9} becomes
\begin{equation}\label{5.11}
\begin{cases}
-\mathcal{M}^+(D^2\bar u_k)+\varrho_k^{\bar{\alpha}_{1}+2-(\bar{\alpha}_{1}+1)q}|\nabla\bar{u}_k|^q=c\bar{v}_{k}^{p_1}&\text{in}~~B_{k}\\
-\mathcal{M}^+(D^2\bar{v}_k)+|\nabla \bar{v}_k|^q=d\bar u_k^{p_2}&\text{in}~~B_{k}.
\end{cases}
\end{equation}
Let $k\to+\infty$ as before, we get local $C^1$ convergence $(\bar u_n,\bar v_n) \to(\bar{u},\bar{v})$ in $C^1_{\mathrm{loc}}(\mathbb{R}^n),$
\[
\begin{cases}
-\mathcal{M}^+(D^2\bar{u})=c\bar{v}^{p_1}&\text{in}~~\mathbb{R}^n,\\
-\mathcal{M}^+(D^2\bar{v})+|\nabla\bar{v}|^q= d\bar{u}^{p_2}&\text{in}~~\mathbb{R}^n,
\end{cases}
\]
and by the same  non triviality condition holds. By Theorem~\eqref{thm5}, this is impossible, which leads to contradiction This conclude the proof of the claim and hence the lemma follows.
\end{enumerate}
\end{proof}
In the following Lemma \ref{lem10} also derive the similar estimate on the solution as in Lemma~\ref{lem9}. But here we allow the different range of exponents, which posses additional difficulties. Due to this fact limiting process becomes very difficult. For this reason, we confine our attention to radially symmetric solutions. In order to state our theorem we need two more notation. 
\[\hat{\alpha}_{1}=\frac{q(p_1+2)}{p_{1}p_{2}-q}, \quad \hat{\alpha}_{2}=\frac{q+2p_2}{p_{1}p_{2}-q}.\]
Before proceeding further we would like to make the following remark.
\begin{remark}
Note that the exponents $\hat{\alpha}_{1}$ and $\hat{\alpha}_{2}$ are not symmetric under the interchange of $p_1$ and $p_2$. This asymmetry is due to our standing assumption $p_1 \ge p_2$.
When $p_1<p_2$ the exponents
\[
\hat{\alpha}_{1}=\frac{q(p_1+2)}{p_{1}p_{2}-q},\quad
\hat{\alpha}_{2}=\frac{q+2p_2}{p_{1}p_{2}-q}
\]
must be replaced by
\[
\hat{\alpha}_{1}^{\prime}=\frac{q+2p_1}{p_{1}p_{2}-q},\quad
\hat{\alpha}_{2}^{\prime}=\frac{q(p_2+2)}{p_{1}p_{2}-q}.
\]
\end{remark}
\begin{lemma}\label{lem10}
Assume $\widetilde{n}_{+} \ge 2$, $q>q_{c}$, $\bar{\alpha}_1\geq\widetilde{n}_+-2,$ and $\bar{\alpha}_2<\beta.$ Suppose further that $\hat{\alpha}_2<\beta< \hat{\alpha}_1$. If $(u,v)$ is a bounded, positive, radially symmetric function satisfying
\begin{equation}\label{5.111}
\begin{cases}
-\mathcal{M}^+(D^2u)+|\nabla u|^{q}=cv^{p_1}&\text{in}~~\mathbb{R}^{n}\setminus B_{R_0}, \\
-\mathcal{M}^+(D^2v)+|\nabla v|^{q}=du^{p_2}&\text{in}~~\mathbb{R}^{n}\setminus B_{R_0},
\end{cases}
\end{equation}
and vanishing at infinity, then there exists $C>0$ such that 
\[u(x) \le C|x|^{-\hat{\alpha}_{1}}~~\text{and}~~v(x) \le C|x|^{-\hat{\alpha}_{2}}\]
for all $|x|$ sufficiently large.
\end{lemma}
\begin{proof}
Let us first observe that the assumptions imply $p_{1}p_{2}>q.$ Indeed, assume by contradiction that $p_{1}p_{2}\leq q.$ Then
$p_{1}p_{2}-1\leq q-1.$  Since $q>1$, multiplying by $(2-q)$ yields $(2-q)(p_{1}p_{2}-1)\leq(2-q)(q-1).$ Moreover,
\[(2-q)(q-1)<q-1<2(p_{2}+1)(q-1),\]
and therefore
\[(2-q)(p_{1}p_{2}-1)<2(p_{2}+1)(q-1).\]
Dividing by  $(q-1)(p_{1}p_{2}-1),$  we obtain $\frac{2-q}{q-1}<\frac{2(p_{2}+1)}{p_{1}p_{2}-1},$ that is, $\beta<\bar{\alpha}_{2},$ which contradicts the assumption $\bar{\alpha}_{2}<\beta.$ The argument leading to the upper bounds follows on the same line as the proof of Lemma~\eqref{lem9}, with the key difference that we now consider the radially symmetric form of system~\eqref{5.111}.
\begin{equation}\label{5.12}
\begin{cases}
-\theta\big(u^{\prime\prime}(r)\big)u^{\prime\prime}(r)-\theta\big(u^{\prime}(r)\big)(n-1)\frac{u^{\prime}(r)}{r}+|u^{\prime}|^q=cv^{p_1}&\text{in}~~r\geq R_0\\
-\theta\big(v^{\prime\prime}(r)\big)v^{\prime\prime}(r)-\theta\big(v^{\prime}(r)\big)(n-1)\frac{v'(r)}{r}+|v^{\prime}|^q=d u^{p_2}&\text{in}~~r\geq R_0.
\end{cases}
\end{equation}
Now by using $u(r),v(r)\to0$ as $r\to\infty$ and following the same argument as in the proof of Lemma \ref{lem9}, we can show that 
\[u^{\prime}(r),~v^{\prime}(r)\to0~~\text{as}~~r\to+\infty.\]
\textbf{Claim:} If  $(u,v)$ is a bounded, positive, and radial solution of \eqref{5.12}, then there exists a constant $C>0$ such that, for any $\tau>2R_0,$
\begin{equation}\label{eq:513}
\begin{cases}
u(r)\le Cd_\tau(r)^{-\hat{\alpha}_{1}},\\
|u'(r)|\le Cd_\tau(r)^{-(\hat{\alpha}_{1}+1)},\\
v(r)\le Cd_\tau(r)^{-\hat{\alpha}_{2}}.
\end{cases}
\end{equation}
where $r\in I_\tau:=\left(\tfrac{\tau}{2}, \tfrac{3\tau}{2}\right)$ and $d_\tau(r)=\min\left\{r-\tfrac{\tau}{2}, \tfrac{3\tau}{2}-r\right\}:=\mathrm{dist}(r,\partial I_\tau).$ As in Lemma \ref{lem9}, we prove this claim by method of contradiction. If no $C>0$ exists, then there exist two sequences
$\{\tau_k\}_{k\ge1}$ and $\{\sigma_k\}_{k\ge1}$ such that $\tau_k>2R_0$, $\sigma_k\in(\tau_k/2,3\tau_k/2)$ and the function 
\[N(\sigma_k):= u^{1/\hat{\alpha}_{1}}(\sigma_k)+ |u^{\prime}(\sigma_k)|^{1/(\hat{\alpha}_{1}+1)}+ v(\sigma_k)^{1/\hat{\alpha}_{2}}\] 
satisfies $N(\sigma_k)>2kd^{-1}_{\sigma_k}(\sigma_k).$ By Lemma 5.1 in \cite{Polavcik}, there exists a sequence $\{r_k\}_{k\ge1}$ with $r_k\in(\tau_k/2,3\tau_k/2)$ such that
\[N(r_k)> 2k\ d^{-1}_{\tau_k}(r_k)~~\text{and}~~N(r)~~\le 2N(r_k)~~\text{if}~~|r-r_k|\leq kN^{-1}(r_k).\]
Set $\varrho_{k}:=N^{-1}(r_k)$ and consider the following rescaled function
\[
\begin{cases}
\bar u_k(\sigma):=\varrho_{k}^{\hat{\alpha}_{1}}u(r_k+\varrho_{k}\sigma),\\
\bar{v}_k(\sigma):=\varrho_{k}^{\hat{\alpha}_{2}}v(r_k+\varrho_{k}\sigma).
\end{cases}
\]
for $|\sigma|\leq k.$ Boundedness of $N$ implies that $d_{\tau_k}(r_k)\to+\infty.$ Furthermore, as $u$, $u'$ and $v$ vanish at infinity implies $N(r_k)\to0$ as $k\to+\infty.$ Therefore, $\varrho_{k}\to+\infty$ as $k\to+\infty.$ It is also easy to see that $(\bar u_k,\bar v_k)$ satisfies 
\begin{equation}\label{5.14}
\begin{cases}
-\theta(\bar{u}^{\prime\prime}_k)\bar{u}^{\prime\prime}_k-\varrho_{k}\theta(\bar{u}^{\prime}_k)(n-1)\frac{\bar{u}^{\prime}_k}{r_k+\varrho_{k}\sigma}+\varrho_{k}^{\hat{\alpha}_1+2-q(\hat{\alpha}_1+1)}|\bar{u}^{\prime}_k|^q=c\bar {v}_{k}^{p_1}&\text{in}~~B_{k},\\
-\theta(\bar{v}^{\prime\prime}_k)\bar{v}^{\prime\prime}_k-\varrho_{k}\theta(\bar{v}^{\prime}_k)(n-1)\frac{\bar{v}^{\prime}_k}{r_k+\varrho_{k}\sigma}+\varrho_{k}^{\hat{\alpha}_2+2-q(\hat{\alpha}_2+1)}|\bar{v}^{\prime}_k|^q=d\bar {u}_{k}^{p_2}&\text{in}~~B_{k}.
\end{cases}
\end{equation}
Before taking the limit in \eqref{5.14}, we claim that $\frac{r_k}{\varrho_{k}}\longrightarrow+\infty.$  Indeed, from
\[d_{\tau_k}(r_k)=\min\left\{r_k-\frac{\tau_k}{2},\frac{3\tau_k}{2}-r_k\right\}\leq r_k~~\text{and}~~N(r_k)d_{\tau_k}(r_k)>2k\to +\infty\]
we obtain, $\frac{r_k}{\varrho_k} = N(r_k) r_k\geq N(r_k)d_{\tau_k}(r_k) \to +\infty,$ and claim follows. 
\[\bar{u}_{k}^{1/\hat{\alpha}_{1}}+\bar{}v_{k}^{1/\hat{\alpha}_{2}}+|\bar{u}^{\prime}_k|^{1/(\hat{\alpha}_{1}+1)} \le 2~~\text{for}~~ |\sigma|\leq k,\]
and
\begin{equation}\label{5.15}
\bar{u}_{k}^{1/\hat{\alpha}_{1}}(0)+\bar{v}_{k}^{1/\hat{\alpha}_{2}}(0)+|\bar{u}^{\prime}_k(0)|^{1/(\hat{\alpha}_{1}+1)}=1.
\end{equation}
From \eqref{5.14}, for every compact interval $F\subset\mathbb R,$  we have
\[0\leq\bar{u}_k(\sigma)\leq C,\quad|\bar{u}^{\prime}_k(\sigma)|\leq C,\quad 0\leq\bar{v}_k(\sigma)\leq C,\quad \sigma\in K,
\]
for some constant $C(F)$ but independent on $k.$. Now we use the first equation of \eqref{5.14}. Since $\frac{r_k}{\varrho_{k}}\to+\infty$
it follows that $\frac{\varrho_{k}}{r_k+\varrho_{k}\sigma}$  and $\frac{1}{r_k/\varrho_{k}+\sigma}\to0$ uniformly on compact subsets of $\mathbb{R}.$ Moreover, as $\hat{\alpha}_1+2-q(\hat{\alpha}_1+1)<0,$ so we get $\varrho_{k}^{\hat{\alpha}_1+2-q(\hat{\alpha}_1+1)}
\longrightarrow 0.$ Using also the uniform boundedness of $\bar{u}^{\prime}_k$ and $\bar v_k,$  the first equation of \eqref{5.14} yields
\[|\bar u_k''(\sigma)|\le C\quad\text{for }\sigma\in K,
\]
where $C$ is independent of $k.$ Hence $\bar{u}^{\prime}_k$  is locally equibounded and equicontinuous. By the Arzelà–Ascoli theorem, there exists $\bar u\in C^{1}(\mathbb R)$  such that, up to a subsequence, $\bar{u}_k\to\bar{u}~~\text{in}~~C^{1}_{\mathrm{loc}}(\mathbb R).$ Getting a uniform convergence subsequence of $\bar{v}_k$ is little hard. Because in view of \eqref{5.14} we had uniform control on $\bar{u}^{\prime}_k$ but it is not available on $\bar{v}^{\prime}_k.$ But we have uniform control on the coefficients of second equation of system \eqref{5.14}. So we can use Lemma \ref{appen01} to get the uniform control on the $\bar{v}^{\prime}_k.$ So we have 
\[\lim_{k\to+\infty}\bar v_k = \bar v~~\text{in}~~ C_{\mathrm{loc}}(\mathbb{R})\]
for some $\bar{v}\in C(\mathbb{R})$. Passing to the limit in \eqref{5.15}, we obtain
\begin{equation}\label{5.16}
\bar u^{1/\hat{\alpha}_{1}}(0)+\bar v^{1/\hat{\alpha}_{2}}(0)+|\bar u'(0)|^{1/(\hat{\alpha}_{1}+1)}=1,
\end{equation}
so $(\bar u,\bar v)$ is non trivial. Letting $k\to+\infty$ in the system \eqref{5.14} of first equation and by using
$\hat{\alpha}_{1}+2-(\hat{\alpha}_{1}+1)q<0$, we obtain
\begin{equation}\label{5.17}
-\bar u^{\prime\prime}=c \bar v^{p_1}~~\text{in}~~\mathbb{R},
\end{equation}
in the viscosity sense, by the standard regularity yields classical validity. Similarly, by passing to the limit in the second relation of \eqref{5.14} and using $(\hat{\alpha}_{2}+1)q-\hat{\alpha}_{2}-2<0$, we get
\begin{equation}\label{5.18}
|\bar v'|^q= d\bar u^{p_2}~~\text{in}~~\mathbb{R},
\end{equation}
in the viscosity sense(c.f \cite{KoikeViscosity}). We claim that this is impossible. By Equation \eqref{5.17} implies that $\bar u$, since $\bar u$ is nonnegative and bounded, so it is concave.  And it must be constant also. Hence $\bar v=0$ using \eqref{5.17}, and then \eqref{5.18} yields $\bar u=0$, which contradicts \eqref{5.16}. Therefore, \eqref{eq:513} holds. Finally, taking $r=\tau$ in \eqref{eq:513}, we conclude,
\[u(r)\le Cr^{-\hat{\alpha}_{1}},~~v(r)\le Cr^{-\hat{\alpha}_{2}},\]
for all sufficiently large $r$ and some constant $C>0$ independent of $r$.
\end{proof}
\begin{remark}\label{rem:gradient_decay}
The proofs of Lemmas~\ref{lem9} and \ref{lem10} contain, besides the decay estimates for the solutions themselves, corresponding decay estimates for their gradients. Since these estimates will be used in the
proof of Theorem~\ref{thm3}, we record them explicitly. Under the assumptions of Lemma~\ref{lem9}, the contradiction quantity is chosen as
\[N(x)=u(x)^{1/\bar\alpha_1}+v(x)^{1/\bar\alpha_2}+|\nabla u(x)|^{1/(\bar\alpha_1+1)}+|\nabla v(x)|^{1/(\bar\alpha_2+1)}.\]
Therefore the doubling argument yields
\begin{equation}\label{gradient_1_used}
\begin{cases}
|\nabla u(x)|\leq Cd_z(x)^{-(\bar\alpha_1+1)},\\
|\nabla v(x)|\leq Cd_z(x)^{-(\bar\alpha_2+1)},
\end{cases}
\end{equation}
where $d_z(x)=\operatorname{dist}\bigl(x,\partial B(z,|z|/2)\bigr).$ Taking $x=z$, we have $d_z(z)=\frac{|z|}{2},$
and consequently, after enlarging the constant if necessary,
\begin{equation}\label{gradient_2_used}
\begin{cases}
|\nabla u(z)|\leq C|z|^{-(\bar\alpha_1+1)},\\
|\nabla v(z)|\leq C|z|^{-(\bar\alpha_2+1)}
\end{cases}
\end{equation}
for all sufficiently large $|z|$. Similarly, the proof of Lemma~\ref{lem10} gives
$|u'(r)|\leq Cr^{-(\hat\alpha_1+1)}$
for every sufficiently large $r.$ Suppose now that $(u,v)$ is radial and satisfies
$u(r)\to0,~~v(r)\to0~~ \text{as}~~r\to\infty.$
By Lemma~\ref{lem4}, both components are eventually monotone decreasing. Hence there exists $R>0$ such that $u'(r)\le0,~~v'(r)\le0~~\text{for}~~r\ge R.$
Therefore, under the assumptions of Lemma~\ref{lem9},
\[-u'(r)\leq Cr^{-(\bar\alpha_1+1)},~~-v'(r)\leq Cr^{-(\bar\alpha_2+1)},~~\text{for}~~r\ge R,\]
while under the assumptions of Lemma~\ref{lem10}, $-u'(r)\leq Cr^{-(\hat\alpha_1+1)}~~\text{for}~~r\ge R.$
Finally, since $q>q_c=\frac{\widetilde n_+}{\widetilde n_+-1},$ we have $\widetilde n_+-1>\frac1{q-1}.$ Consequently,
\begin{equation}\label{gradient_3_used}
\begin{cases}
\bar\alpha_1\ge\widetilde n_+-2\quad\Longrightarrow\quad\bar\alpha_1+1>\frac1{q-1},\\
\text{and}\\
\hat\alpha_1\ge\widetilde n_+-2\quad\Longrightarrow\quad\hat\alpha_1+1>\frac1{q-1}.
\end{cases}
\end{equation}
Therefore, in each of the situations covered by Theorem~\ref{thm3}, there exists an exponent $\beta>\frac1{q-1}$ such that
\begin{equation}\label{gradient_4_used}
-u'(r)\le Cr^{-\beta}
\end{equation}
for all sufficiently large $r$. The analogous estimate also holds for $v$ whenever the corresponding
hypotheses of Lemma~\ref{lem9} are satisfied. These gradient estimates constitute the only information from the doubling arguments needed in the proof of the eventual convexity theorem proved in the next subsection.
\end{remark}

\section{Proof Theorem~\ref{thm1}, Theorem~\ref{thm2},~Theorem~\ref{thm3} and Theorem \ref{thm4}}
This section deals with the proof of main theorem by using all the results established up to now. The proof of these results are established by method of contradiction. We reach contradiction by using the upper and lower estimates established in the previous sections. In order to achieve this we use the following elementary lemma which is proved in \cite{burgos2018lioville}. But as the solutions involves here are not $C^{2}$ so we need a refined version of the mentioned lemma.
\begin{lemma}\label{lem11}
Assume $w\in W^{2,s}(R_0,+\infty)\cap C^{1}(R_0,+\infty),$ where $R_0\ge0$ and $z(r)=r^{\alpha_{1}} w(r),$ for $\alpha_{1}\in\mathbb R$ is bounded for large $r.$ Then for every measurable set $E\subset(R_0,\infty)$ with full measure, we can find a sequence $r_{k}\in E$ such that $
r_k\to\infty$ and following hold:
\begin{enumerate}
\item $\lim_{k\to+\infty} z(r_k)=\Sigma:=\liminf_{r\to+\infty} z(r),$
\item $\lim_{k\to+\infty} r_k^{\alpha_{1}+1}w'(r_k)=-\alpha_{1}\Sigma,$
\item $\lim_{k\to+\infty}r_k^{\alpha_{1}+2}w''(r_k)\ge \alpha_{1}(\alpha_{1}+1)\Sigma.$
\end{enumerate}
\end{lemma}
\begin{proof} Since $w\in W^{2,s}_{\rm loc}(R_0,\infty)\cap C^1(R_0,\infty),$ so $z(r)=r^{\alpha_{1}} w(r)\in W^{2,s}_{\mathrm{loc}}(R_0,\infty)\cap C^1(R_0,\infty).$ Therefore $z^{\prime}$ is locally absolutely continuous, consequently we have $z^{\prime\prime}$ exists almost everywhere (on Lebesgue points) and we can evaluate $z^{\prime\prime}$ on this set as well as it has full measure in view of $z^{\prime\prime}\in L^{s}_{\mathrm{loc}}.$ Since $\Sigma=\liminf_{r\to\infty}z(r)>-\infty,$ there exists $R_\Sigma>R_0$ such that \[z(r)\ge \Sigma-1\forall~~r\geq R_{\Sigma}.\] For each integer $k\geq1,$ define 
\[z_k(r)=z(r)+\frac{r}{k}.\] 
For $r\geq R_\Sigma,$ $z_k(r)\geq\Sigma-1+\frac{r}{k}.$ As $\frac{r}{k}\to+\infty,$ so $z_k(r)\to+\infty$ with $r\to\infty.$ Thus, $z_k$ is coercive at infinity and $z_k\in C(R_0,\infty).$ Consequently, there exists $\rho_k\in(R_0,\infty)$ such that
\[z_k(\rho_k)=\min_{r\ge R_0}z_k(r).\]
Moreover, $\rho_k\to\infty.$ If not then, after passing to a subsequence, $\rho_k\to\rho<\infty.$
Fix any \(R>\rho+1\). By definition of $\Sigma,$ there exists a sequence $s_k\to\infty$ such that $z(s_k)\to\Sigma.$ For every $k,$ 
\[z_k(\rho_k)\leq z_k(s_k)=z(s_k)+\frac{s_k}{k}.\]
Choose $s_k$ so that
\[z(s_k)\le\Sigma+\frac{1}{k},\quad s_k=o(k).\]
Then $z_k(\rho_k)\leq\Sigma+\frac{1}{k}+\frac{s_k}{k}.$ Hence
\[\limsup_{k\to\infty}z_k(\rho_k)\leq\Sigma.\]
But $z_k(\rho_k)=z(\rho_k)+\frac{\rho_k}{k},$ and $\rho_k\to\rho,$ thus
$z(\rho)\leq\Sigma.$ This contradicts the definition of $\Sigma,$ since $\rho$ is finite. Thus, $\rho_k\to\infty.$
As $z_k\in C^1 $ and $\rho_k$ is an interior minimizer so $z_k'(\rho_k)=0.$ Consequently, 
\[z'(\rho_k)=-\frac{1}{r^2k}.\]
Also as $z_k(\rho_k)=z(\rho_k)+\frac{\rho_k}{k}$ remains bounded and $z(\rho_k)\to\Sigma,$ we necessarily have $\frac{\rho_k}{k}\to0.$ Thus we also have $\rho_kz'(\rho_k)\to0.$ Now differentiating $z=r^\alpha_{1} w$ and setting $r=\rho_{k}$ and rearranging the term we have
\[\rho_k^{\alpha_{1}+1}w'(\rho_k)=\rho_kz'(\rho_k)-\alpha_{1} z(\rho_k).\]
Taking $k\to \infty,$ we have
\[\rho_k^{\alpha_{1}+1}w'(\rho_k)\to-\alpha_{1}\Sigma.\]
As $z$ is Sobolev function so $z^{\prime\prime}$ may not be defined at $\rho_{k}.$ Set 
\[F=\{r:z^{\prime\prime}(r)~~\text{exists}\}.\]  
 As $F$ has full measure. So $E\cap F$ also has full measure as $z''\in L^s_\mathrm{loc}$ by Lebesgue differentiation theorem. Since $\rho_k$ minimizes $z_k,$ the generalized second-order optimality condition yields
\[\operatorname*{ess\,liminf}_{r\to\rho_n}z_k''(r)\ge0.\]
In other word, for every $\epsilon>0,$
\[\left|\left\{r:|r-\rho_n|<\delta,z_k''(r)\ge-\varepsilon\right\}\right|>0\]
for sufficiently small $\delta.$ As $E\cap F$ has full measure, so we can choose $r_k\in(E\cap F)\cap(\rho_k-1/k,\rho_k+1/k)$
such that 
\[z^{\prime\prime}(r_k)\geq-\frac{1}{k}.\]
Moreover, as a consequence of continuity of $z$ and $z^{\prime}$ we also have 
\[z(r_k)-z(\rho_k)\to0,~\quad~\text{and}~\quad r_kz'(r_k)-\rho_kz'(\rho_k)\to0.\]
Thus, $z(r_k)\to\Sigma,$ and $r_kz^{\prime}(r_k)\to0.$\\
Now differentiating $z=r^{\alpha}w$ two times, rearranging the term and setting $r=r_{k},$ we have
\[r_k^{\alpha_{1}+2}w''(r_k)=r_k^2z''(r_k)-2\alpha_{1} r_kz'(r_k)+\alpha_{1}(\alpha_{1}+1)z(r_k).\]
Finally, using 
\[z''(r_k)\ge-\frac{1}{k},~~~r_kz'(r_k)\to0~~\text{and}~~z(r_k)\to\Sigma,\]
we conclude
\[\liminf_{k\to\infty}r_k^{\alpha_{1}+2}w''(r_k)\geq\alpha_{1}(\alpha_{1}+1)\Sigma.\]
This completes the proof.
\end{proof}
\subsubsection{Proof Theorem~\ref{thm1}}
\begin{proof}
Suppose that system \eqref{main} admits a positive supersolution $(u,v)$.  
Fix $\varepsilon>0$. By Lemma \eqref{lem4}, there exists a positive, bounded and radially symmetric solution $(w,z)$ of
\begin{equation}\label{6.3}
\begin{cases}
-\mathcal{M}^+(D^2 w)+|\nabla w|^{q}=\lambda_1(\mu_1-\varepsilon)z^{p_1}&\text{in}~~\mathbb{R}^{n}\setminus B_{R_1},\\
-\mathcal{M}^+(D^2 z)+|\nabla z|^{q}=\lambda_2(\mu_2-\varepsilon)w^{p_2}&\text{in}~~\mathbb{R}^{n}\setminus B_{R_1},
\end{cases}
\end{equation}
for  $R_1>R_0$.
\begin{enumerate}
\item{} As this part is concerned with $p_{1}p_{2}\le q^{2}.$  If $p_{1}p_{2}<q^{2},$ then by Lemma~\ref{lem8}(a), the auxiliary radial problem admits no positive bounded solutions, which is not possible. Thus we assume $p_{1}p_{2}=q^{2}.$ In this critical case, Lemma~\ref{lem5} gives the lower estimate $w(r)\ge Cr^{-\beta}$ for all sufficiently large $r,$ where $\beta=\frac{2-q}{q-1}.$ On the other hand, by Lemma~\ref{lem8}(b), we obtain an exponential upper decay of the form
\[w(r)\leq C_{1}r^{-q}e^{-C_{2}r}~~\text{for sufficiently large}~~~r,\]
and some constants $m>0,$  $C_{1}>0$ and $C_{2}>0.$  Therefore,
\[Cr^{-\beta}\leq w(r)\leq C_{1}r^{-q}e^{-C_{2}r}~~\text{for large values of}~r.\]
This is impossible, since the right-hand side decays exponentially as $r\to+\infty,$  while the left-hand side decays only polynomially. More precisely, $r^{\beta-q}e^{-C_{2}r}\to0$ as~~$r\to+\infty,$ 
and hence for $r$ sufficiently large,
\[C_{1}r^{-q}e^{-C_{2}r}<Cr^{-\beta},\]
which contradicts the previous estimate. Thus the case $p_{1}p_{2}=q^{2}$  is also impossible and completes the proof of part $(a).$
\item{} Assume that $p_{1}p_{2}>q^{2}$ and 
$\alpha_{1}>\beta.$ Suppose by contradiction that system \eqref{main} admits a positive supersolution $(u,v)$ which does not blow up at infinity. Since $(u,v) $is only a supersolution, we cannot directly apply the radial estimates of Lemma~\ref{lem5} and Lemma~\ref{lem8}, which require bounded positive radial solutions of the associated auxiliary problem. To overcome this difficulty, we use Lemma~\ref{lem4}. It ensures that, for every fixed $\varepsilon>0,$ there exists a positive, bounded,
radially symmetric solution \((w,z)\) of the auxiliary system
\begin{equation}\label{6.3b}
\begin{cases}
-\mathcal{M}^{+}(D^{2}w)+|\nabla w|^{q}=\lambda_{1}(\mu_{1}-\varepsilon)z^{p_{1}}
&\text{in }\mathbb{R}^{n}\setminus B_{R_{1}},\\
-\mathcal{M}^{+}(D^{2}z)+|\nabla z|^{q}=\lambda_{2}(\mu_{2}-\varepsilon)\,w^{p_{2}}&\text{in}~~\mathbb{R}^{n}\setminus B_{R_{1}},
\end{cases}
\end{equation}
for some $R_{1}>R_0$. Since $(w,z)$ is positive, bounded, and radially symmetric, we may now apply the ODE-type estimates which already established earlier. First, by Lemma~\ref{lem5}, we obtain the universal lower estimate $w(r)\ge C_{1}\,r^{-\beta}$ for all sufficiently large $r,$ where $\beta=\frac{2-q}{q-1}.$ On the other hand, since \(p_{1}p_{2}>q^{2}\), Lemma~\ref{lem8}(c)
applies and yields the upper decay estimate $w(r)\le C_{2}\,r^{-\alpha_{1}}$ for all sufficiently large $r.$ Therefore, 
\[C_{1}r^{-\beta}\leq w(r)\leq C_{2}r^{-\alpha_{1}},\] for sufficiently large values of $r.$ Since $\alpha_{1}>\beta,$ the function \(r^{-\alpha_{1}}\) decays strictly faster than $r^{-\beta}.$  Indeed, $r^{\beta-\alpha_{1}}\to0~~\text{as}~~r\to+\infty.$ Hence, for sufficiently large $r,$ $C_{2}r^{-\alpha_{1}}<C_{1}r^{-\beta},$ which contradicts the previous double inequality.This contradiction shows that no positive supersolution of \eqref{main} which does not blow up at infinity can exist.
\item{} Assume $p_1p_2>q^2,$ $\alpha_{1}=\alpha_{2}=\beta$ hold. Despite that system \eqref{main} admits a positive supersolution
$(u,v) $ which does not exhibit blow-up behaviour at infinity. By Lemma~\ref{lem4}, for every fixed $\varepsilon>0,$
there exists a positive bounded radially symmetric supersolution
$(w,z)$ satisfying
\[
\begin{cases}
-\mathcal M^+_{\lambda,\Lambda}(D^2w)+|\nabla w|^q=\lambda_1(\mu_{1}-\varepsilon)z^{p_1},\\
-\mathcal M^+_{\lambda,\Lambda}(D^2z)+|\nabla z|^q=\lambda_2(\mu_{2}-\varepsilon)w^{p_2},
\end{cases}
\]
for sufficiently large $r.$ Furthermore we have $w,z\in W^{2,s}_{loc}(R_1,\infty)\cap C^{1}(R_1,\infty).$ Thus $w$ is twice differentiable on a full measure set. So while applying Lemma~\ref{lem11} we will be choosing the sequence from the set where $w$ is twice differentiable. By Lemma~\ref{lem5} and Lemma~\ref{lem8},
there exist positive constants $C_1,C_2$ such that
\[
\begin{cases}
C_1r^{-\alpha_{1}}\leq w(r)\le C_2r^{-\alpha_{1}},\\
C_1r^{-\alpha_{2}}\leq z(r)\le C_2r^{-\alpha_{2}},
\end{cases}
\]
for sufficiently large $r.$ So we can define 
\[
t_1:=\liminf_{r\to\infty}r^{\alpha_1}_{1}w(r),~~
t_2:=\liminf_{r\to\infty}r^{\alpha_1}_{2}z(r),
\]
and as $w,z$ decays power like we also have $0<t_1,t_2<\infty.$
Now apply Lemma~\ref{lem11} to $w.$ This gives a sequence
$\{r_k\}_{k\geq1}\to\infty$ such that
\begin{enumerate}
\item{}$r_k^{{\alpha}_{1}}w(r_k)\to t_{1},$
\item{}$r_k^{\alpha_{1}+1}w'(r_k)\to-\alpha_{1}t_1,$
\item{}$\liminf_{k\to\infty}r_k^{\alpha_{1}+2}w''(r_k)\geq\alpha_{1}(\alpha_{1}+1)t_1.$
\end{enumerate}
Notice that last point implies that there is a natural number such that for $k\geq n,$ we have $w^{\prime\prime}(r_{k})>0.$ Also as we already observed that $w^{\prime}(r)<0.$ Thus the following hold:
\[
\begin{cases}
\theta(w^{\prime}(r_{k}))=\lambda,\\
\theta(w^{\prime\prime}(r_{k}))=\Lambda,
\end{cases}
\]
along these sequences. Moreover, $w$ is radial and twice differentiable almost everywhere, so without loss of generality, we may assume that it is twice differentiable at $r_{k}$ so we have 
\[-\Lambda w^{\prime\prime}(r_{k})-\lambda(n-1)\frac{w^{\prime}(r_{k})}{r_{k}}+|w^{\prime}(r_{k})|^q\geq\lambda_1(\mu_{1}-\varepsilon)z^{p_1}(r_{k}).\]
Multiply both sides by $r_k^{\alpha_{1}+2}$ we get
\begin{equation}\label{final1.1}
-\Lambda r_{k}^{\alpha_{1}+2} w^{\prime\prime}(r_{k})-r_{k}^{\alpha_{1}+1}\lambda(n-1)w^{\prime}(r_{k})+r_{k}^{\alpha_{1}+2}|w^{\prime}(r_{k})|^q\geq\lambda_1(\mu_{1}-\varepsilon)r_{k}^{\alpha_{1}+2}z^{p_1}(r_{k}). 
\end{equation}
Now, considering each term separately we have
\begin{enumerate}
\item $r_k^{\alpha_{1}+2}w^{\prime\prime}\geq\alpha_{1}(\alpha_{1}+1)t_1+o(1),$
\item $r_k^{\alpha_{1}+1}w^{\prime}\to-\alpha_{1}t_1,$
\item$r_k^{\alpha_{1}+2}|w^{\prime}|^q=\left(r_k^{\alpha_{1}+1}|w^{\prime}|\right)^qr_k^{\alpha_{1}+2-q(\alpha_{1}+1)}.$
\end{enumerate}
In view of $\alpha_{1}=\beta=\frac{2-q}{q-1},$ we have $\alpha_{1}+2-q(\alpha_{1}+1)=0.$ Therefore, 
\[ r_k^{\alpha_{1}+2}|w\prime|^q=\left(r_k^{\alpha_{1}+1}|w^{\prime}|\right)^qr_k^{\alpha_{1}+2-q(\alpha_{1}+1)}\to\alpha_{1}^qt_1^q.\]
Thus passing to the limit in \eqref{final1.1}, we get 
\[\Lambda\alpha_{1}(\widetilde{n}_{+}-2-\alpha_{1})t_1+\alpha_{1}^qt_1^q\geq \lambda_1\mu_{1}t_2^{p_1},\]
which after putting the value of $\widetilde{n}_{+}$ we get contradiction to the first inequality in \eqref{1.5}. Same result also holds for $z$ and this part is proved.
\item{} Assume $\alpha_{2}<\alpha_{1}=\beta$ and \eqref{1.5}.  
By Lemma \ref{lem5} and Lemma \ref{lem8}(c),
\[
\begin{cases}
C_1 r^{-\alpha_{1}}\leq w(r)\le C_2 r^{-\alpha_{1}},\\
C_1 r^{-\beta}\le z(r)\le C_2 r^{-\alpha_{2}}.
\end{cases}
\]
As a consequence we have,
\[
\begin{aligned}
-\mathcal{M}^+(D^2 z)+|\nabla z|^{q}=dw^{p_2}\geq Cr^{-p_2\alpha_{1}}=Cr^{-q(\alpha_{2}+1)},\\  
\end{aligned}
\]
which in view of Lemma~\eqref{lem7} gives $z(r)\ge Cr^{-\alpha_{2}}$. Therefore, it is natural to define
\[
t_1:=\liminf_{r\to+\infty} r^{\alpha}_{1}w(r),~~
t_2:=\liminf_{r\to+\infty} r^{\alpha}_{2}z(r),
\]
 and observe that both are finite and positive. By Lemma~\eqref{lem11}, there is two sequences $\{r_{k}\}_{k\geq1},\{s_{k}\}_{k\geq1}$ together $r_{k},s_{k}\to+\infty$ satisfying 
\[
\left\{
\begin{aligned}
&\lim_{k\to\infty}r_k^{{\alpha}_{1}}w(r_k)=t_1,\\
&\lim_{k\to\infty}r_k^{{\alpha}_{1}+1}w'(r_k)=-\alpha_{1}t_{1}\\
&\lim_{k\to\infty}r_k^{{\alpha}_{1}+2}w''(r_k)\geq\alpha_{1}(\alpha_{1}+1)t_1,
\end{aligned}
\right.
\text{and}~~~ 
\left\{
\begin{aligned}
&\lim_{k\to\infty}s_{k}^{{\alpha}_{2}}z(s_{k})=t_2,\\
&\lim_{k\to\infty}s_{k}^{{\alpha}_{2}+1}z'(s_{k})=-\alpha_{2} t_2,\\
&\lim_{k\to\infty}s_{k}^{{\alpha}_{2}+2}z''(s_{k})\ge \alpha_{2}(\alpha_{2}+1)t_2.
\end{aligned}
\right.
\]

Arguing similarly as before we obtain
\begin{equation}\label{6.4}
\alpha_{1}(\Lambda(n-1)-\lambda(\alpha_{1}+1))t_1+\alpha_{1}^q t_1^q\geq\lambda_1(\mu_{1}-\epsilon)t_2^{p_1},
\end{equation}
Now, we consider the radial form of the second equation in \eqref{6.3},
\[
-\theta(z^{\prime\prime}(r))-\theta(z^{\prime}(r))\frac{n-1}{r}z^{\prime}(r)+|z^{\prime}(r)|^{q}\geq\lambda_{2}(\mu_{2}-\epsilon)w(r)^{p_{2}}.
\]
Now evaluating the above equation at $r=s_{k},$ noting that $z^{\prime\prime}(s_{k})>0,~z'(s_{k})<0$ and multiplying by $s_{k}^{q(\alpha_{2}+1)}$
\begin{equation}
-s_{k}^{q({\alpha}_{2}+1)}\Lambda z''(s_{k})-\lambda(n-1)s_{k}^{q({\alpha}_{2}+1)-1} z'(s_{k})+ s_{k}^{q({\alpha}_{2}+1)} |z'(s_{k})|^{q}\ge\lambda_2(\mu_{2}-\varepsilon)s_{k}^{q({\alpha}_{2}+1)} w(s_{k})^{p_2}.
\end{equation}
As a consequence of conditions satisfied by the exponent in case (4), we have $q({\alpha}_{2}+1)=p_2\alpha_{1}$ and  
$q({\alpha}_{2}+1)-({\alpha}_{2}+2)<0.$ Thus the above equation can be rewritten as 
\begin{equation}\label{6.5}
s_{k}^{q(\alpha_{2}+1)-(\alpha_{2}+2)}\left(-s_{k}^{\alpha_{2}+2}\Lambda z''(s_{k})-\lambda (n-1)s_{k}^{\alpha_{2}+1} z'(s_{k})\right)+\big|s_{k}^{\alpha_{2}+1}z'(s_{k})\big|^{q}
\ge\lambda_2(\mu_{2}-\varepsilon)\big(s_{k}^{\alpha_{1}} w(s_{k})\big)^{p_2}.
\end{equation}
Note that $q({\alpha}_{2}+1)-({\alpha}_{2}+2)<0.$ Taking the limit $k \to \infty$ in \eqref{6.5}, we arrive at $s_{k}^{q({\alpha}_{2}+1)-({\alpha}_{2}+2)}\to 0~~\text{as}~~ k\to\infty.$
Hence
\[
\begin{aligned}
&\lim_{k\to\infty}\bigl|s_{k}^{{\alpha}_{2}+1}z'(s_{k}) \bigr|^{q}\ge\lim_{k\to\infty}\lambda_2(\mu_{2}-\varepsilon)
(s_{k}^{{\alpha}_{1}} w(s_{k}))^{p_2},
\end{aligned}
\]
implies
\begin{equation}\label{6.6}
{\alpha}_{2}^{q}t_2^{q}\ge \lambda_2(\mu_2-\varepsilon)t_1^{p_2}.    
\end{equation}
Now, letting $\varepsilon\to0$ in the above inequalities contradicts \eqref{1.5}. This completes the proof.\\
\end{enumerate}
\end{proof}
\subsubsection{Proof of Theorem \ref{thm2}}
\begin{proof}
Assume by contradiction that system \eqref{main} admits a positive supersolution $(u,v)$ which does not exhibit blow-up behavior at infinity.
\begin{enumerate}
\item[(a)] First assume that $p_1p_2<q_c^2.$ Choose $q'<q_c$ sufficiently close to $q_c$ such that $p_1p_2<(q')^2.$ Since $q'<q_c$, Theorem~\ref{thm1}(a) applies and yields a contradiction. Now assume that
$p_1p_2=q_c^2.$ Let $q'=q_c-\varepsilon,$ for some $\epsilon>0.$ Then $q'<q_c$ and for $\varepsilon$ sufficiently small $p_1p_2>(q')^2.$
Define
\[
\alpha_{1}'=\frac{q'(p_1+q')}{p_1p_2-(q')^2},~~\beta'=\frac{2-q'}{q'-1}.
\]
Notice that
\[\alpha_{1}'=\frac{(q_c-\varepsilon)(p_1+q_c-\varepsilon)}{\varepsilon(2q_c-\varepsilon)}=\frac{(q_c-\varepsilon)(p_1+q_c-\varepsilon)}{q_c^2-(q_c-\varepsilon)^2},\]
where we have used $p_1p_2=q_c^2$ in the first equality and $q_c^2-(q_c-\varepsilon)^2=\varepsilon(2q_c-\varepsilon)$ in the second. Thus, $\alpha_{1}'\to+\infty$ as $\varepsilon\to0^+.$ On the other hand,
\[\beta'\to\frac{2-q_c}{q_c-1}=\widetilde{n}_{+}-2,\]
where we have used $q_c=\frac{\widetilde{n}_{+}}{\widetilde{n}_{+}-1}.$
Which shows $\beta'$ stay bounded as $\varepsilon\to0^+$, while
$\alpha_{1}'\to+\infty$. Therefore, for $\varepsilon>0$ sufficiently small, $\alpha_{1}'>\beta'.$ Which leads to contradiction in view of Theorem~\ref{thm1}(b).
\item[(b)] Assume now that $p_1p_2>q_c^2$ and $\alpha_{1}>\widetilde{n}_{+}-2.$ By continuity of the functions
\[
q\mapsto \frac{q(p_1+q)}{p_1p_2-q^2},~~\text{and}~~q\mapsto \frac{2-q}{q-1}.
\]
Then there exists $q'<q_c$ sufficiently close to $q_c$ such that $p_1p_2>(q')^2$ and
\[\frac{q'(p_1+q')}{p_1p_2-(q')^2}>\frac{2-q'}{q'-1}.\]
Hence the assumptions of Theorem~\ref{thm1}(b) are satisfied, which again leads to a contradiction.
\end{enumerate}
\end{proof}
When $q>\frac{\widetilde{n}_{+}}{\widetilde{n}_{+}-1}$ and $p_1\ge p_2$, the related exponents are given as follows
\[
\beta=\frac{2-q}{q-1},~~
\bar\alpha_{1}=\frac{2(p_1+1)}{p_1p_2-1},~~ \bar{\alpha}_{2}=\frac{2(p_2+1)}{p_1p_2-1},~~\hat\alpha_{1}=\frac{q(p_1+2)}{p_1p_2-q},~~
\hat{\alpha}_{2}=\frac{q+2p_2}{p_1p_2-q}.
\]
It is easy to observe that these are related by the following set of relations
\[
\begin{cases}
\bar\alpha_{1}+2=p_1\bar{\alpha}_{2},\\
\bar{\alpha}_{2}+2 = p_2\bar\alpha_{1},
\end{cases}~~\text{and}~~
\begin{cases}
\hat{\alpha}_{1}+2 = p_1\hat{\alpha}_{2},\\
q(\hat{\alpha}_{2}+1) = p_2\hat\alpha_{1}.
\end{cases}
\]
\subsubsection{Proof of Theorem \ref{thm3}} 
\begin{proof}
Assume by contradiction that $(u,v)$ is a positive supersolution of \eqref{main} which do not blow up at infinity. By Lemma~\ref{lem4}, there exists a positive radial function $(w,z)$ satisfying
\begin{equation}\label{fi1}
\begin{cases}
-\theta\big(w''(r)\big)w''(r)-\lambda(n-1)\dfrac{w'(r)}{r}+|w'(r)|^{q}
=\lambda_{1}\varepsilon_{1}z^{p_{1}}(r),\\
-\theta\big(z''(r)\big)z''(r)-\lambda(n-1)\dfrac{z'(r)}{r}+|z'(r)|^{q}
=\lambda_{2}\varepsilon_{2}w^{p_{2}}(r).
\end{cases}
\end{equation}
Note that $w,z$ are positive radial solutions converging to zero at infinity. Now we claim that 
\begin{enumerate}
\item under hypothesis Theorem \ref{thm3}(b), $w,z$ are eventually convex.
\item under hypothesis Theorem \ref{thm3}(c), $w$ is eventually convex.
\end{enumerate}
We prove these claim as a separate result below, see Corollary\ref{evencon1}. Having assumed these claim we want to establish the lower estimate for $w$ and $z$ as follows. 
As $w^{\prime}(r)<0$ and $w$ is eventually convex therefore we also have $\theta(w^{\prime\prime})=\Lambda.$ Moreover, by noting  that $z\geq0$ the first equation  in \eqref{fi1} becomes:
\begin{equation}\label{3.s1}
-w''(r)-\frac{\widetilde n_{+}-1}{r}w'(r)+\frac1\Lambda |w'(r)|^q\ge0.~~\text{a.e. in}~~(R_0,+\infty),
\end{equation}
where we also have used  $\widetilde n_{+}-1=\frac{\lambda}{\Lambda}(n-1).$
By setting $v(r):=-w'(r),$ \eqref{3.s1} can be rewritten as follows
\begin{equation}\label{3.sf1}
v'(r)+\frac{\widetilde n_{+}-1}{r}v(r)\geq-\frac1\Lambda v(r)^q.
\end{equation}
By multiplying \eqref{3.sf1} by $-(q-1)v^{-q}$ and setting 
$Y(r)=v(r)^{1-q},$ we get
\[Y'(r)\le(q-1)\frac{\widetilde n_{+}-1}{r}Y(r)+\frac{q-1}{\Lambda}.\]
Set $\gamma:=(q-1)(\widetilde n_{+}-1).$ Since
$q>q_c=\frac{\widetilde n_{+}}{\widetilde n_{+}-1},$ we have $\gamma>1.$ Multiplying the previous inequality by the integrating factor $r^{-\gamma}$ gives
\[\left(r^{-\gamma}Y(r)\right)'\le\frac{q-1}{\Lambda}r^{-\gamma}.\]
Integrating over $[R_1,r]$ yields
\[r^{-\gamma}Y(r)-R_1^{-\gamma}Y(R_1)\le\frac{q-1}{\Lambda}
\int_{R_1}^{r}s^{-\gamma}ds\leq \frac{R_1^{1-\gamma}}{\gamma-1}=\tilde{C}_{1}. \]
Rewriting the above inequality we obtain we obtain $Y(r)\le C_1r^\gamma.$ Now noting that $1-q<0$ and $\frac{\gamma}{q-1}=\widetilde n_{+}-1$ we get
\[v(r)\geq C_2r^{-(\widetilde n_{+}-1)}.\]
Recalling that $v=-w'$, we obtain
\[|w'(r)|=-w'(r)\ge C_2r^{\,1-\widetilde n_{+}},\qquad r\ge R_1.\]
Which after integration gives 
\[w(r)\geq C_3r^{\,2-\widetilde n_{+}},\]
where $C_3=C~{\widetilde n_{+}-2}>0.$ Since the interval $[R_0,R_1]$ is compact, by decreasing the constants if necessary, the above estimates extend to all
$r\ge R_0$. Hence
\begin{equation}\label{fi2}
|w'(r)|\geq Cr^{1-\widetilde n_{+}},~~ w(r)\geq Cr^{2-\widetilde n_{+}},
\end{equation}
for every $r\ge R_0.$  Similar result also hold for $z.$ Notice that in the proof of part $(a)$ we do not need these lower bound. While in the part of $(b)$ we need these lower bound for $w,z$ both and in the part of $(c)$ we need the lower estimate for $z$ only.   
\begin{itemize}
\item[\textbf{(a).}] In this case we have $p_1p_2\le1.$ Now, choose $t>0$ such that $(p_1+t)p_2>1.$ It also implies that $z^{p_1}(r)\ge z^{p_1+t}(r)$ for sufficiently large $r,$ as $z$ is very small near infinity. Hence $(w,z)$ is also a supersolution of
\[
\begin{cases}
-\mathcal M^+_{\lambda,\Lambda}(D^2w)+|w'|^q\geq c z^{p_1+t},\\
-\mathcal M^+_{\lambda,\Lambda}(D^2z)+|z'|^q\geq dw^{p_2}.
\end{cases}
\]
Define $p_1'=p_1+t,$ and consider the corresponding exponents
\[
\bar{\alpha}_{1}'=\frac{2(p_1'+1)}{p_1'p_2-1},~~\bar{\alpha}_{2}'=\frac{2(p_2+1)}{p_1'p_2-1}.\]
Since $p_1'p_2-1>0,$ we may choose $t>0$ sufficiently small such that
\[\bar{\alpha}_{1}'>\widetilde{n}_{+}-2,~~\bar{\alpha}_{2}'>\beta.\]
Therefore part (b) applies to the perturbed system, and we get a contradiction. 
\item[\textbf{(b).}] Assume that $\bar{\alpha}_{1}\ge \widetilde{n}_{+}-2,$
and $\bar{\alpha}_{2}\ge\beta.$ So by Lemma~\ref{lem9},
\[w(r)\le C r^{-\bar{\alpha}_{1}},~~z(r)\le Cr^{-\bar{\alpha}_{2}}.\]
Combining this with \eqref{fi2} we have
\[C r^{-(\widetilde{n}_{+}-2)}\leq w(r)\leq Cr^{-\bar{\alpha}_{1}}.\]
If $\bar{\alpha}_{1}>\widetilde{n}_{+}-2,$ we immediately obtain a contradiction for large $r.$ Hence we may assume $\bar{\alpha}_{1}=\widetilde{n}_{+}-2.$
Since $p_2\bar{\alpha}_{1}=\bar{\alpha}_{2}+2,$ the second equation in \eqref{fi2} gives
\[-\mathcal M^+_{\lambda,\Lambda}(D^2z)+|z'|^q\geq Cr^{-p_2\bar{\alpha}_{1}}=
Cr^{-(\bar{\alpha}_{2}+2)}.\]
Thus Lemma~\ref{lem6} implies $z(r)\geq Cr^{-\bar{\alpha}_2},$ consequently, we have
\[
\begin{cases}
C_1 r^{-\bar{\alpha}_{1}}\leq w(r)\leq C_2 r^{-\bar{\alpha}_{1}},
\\
\text{and}\\
C_1 r^{-\bar{\alpha}_{2}}\leq z(r)\leq C_2 r^{-\bar{\alpha}_{2}}.
\end{cases}
\]
Thus as in the proof of Theorem \ref{thm1}, we define 
\[\Sigma=\liminf_{r\to\infty}r^{\bar{\alpha}_{1}}w(r).\]
By Lemma~\ref{lem11}, there exists a sequence $r_k\to\infty$ belonging to the differentiability set of $w^{\prime\prime}$ such that
\[
\left\{
\begin{aligned}
&r_k^{\bar{\alpha}_{1}}w(r_k)\to\Sigma>0,\\
&r_k^{\bar{\alpha}_1+1}w^{\prime}(r_k)\to-\bar{\alpha}_{1}\Sigma,\\
&\liminf_{k\to\infty}r_k^{\bar{\alpha}_{1}+2}w^{\prime\prime}(r_k)\geq \bar{\alpha}_{1}(\bar{\alpha}_{1}+1)\Sigma.
\end{aligned}
\right.
\]
Evaluating the first equation of \eqref{fi1}
at $r_k$ and multiplying by $r_k^{\bar{\alpha}_{1}+2}$ we get
\begin{equation}\label{fi3}
-\Lambda r_k^{\bar{\alpha}_{1}+2}w^{\prime\prime}(r_k)-\lambda(n-1)r_k^{\bar{\alpha}_{1}+1}w^{\prime}(r_k)
+r_k^{\bar{\alpha}_{1}+2}|w^{\prime}(r_k)|^q\geq cr_k^{\bar{\alpha}_{1}+2}z(r_k)^{p_1}.
\end{equation}
 Now passing to the limit as in \eqref{fi3}, as in the proof of Theorem \ref{thm1},  we get
 \begin{equation}\label{fi4}
-\Lambda\bar{\alpha}_{1}(\bar{\alpha}_{1}+1)\Sigma+\lambda(n-1)\bar{\alpha}_{1}\Sigma\geq c\,T^{p_{1}},
\end{equation}
where $T:=\liminf_{r\to\infty}r^{\bar{\alpha}_{2}}z(r)>0.$ Finally putting the values of all the exponents in \eqref{fi3}, we get that the left hand side is zero while the right hand side is positive, which leads to a contradiction.
\item[\textbf{(c).}] In this case we have $\bar{\alpha}_{2}<\beta$ and $\hat{\alpha}_{1}\ge\widetilde{n}_{+}-2.$ So by Lemma \ref{lem10}, we have
\[w(r)\le C r^{-\hat{\alpha}_{1}},\]
which in view of \eqref{fi2}, give $C r^{-(\widetilde{n}_{+}-2)}\leq w(r)\le Cr^{-\hat{\alpha}_{1}}.$ So there are two possibilities. If
$\hat{\alpha}_{1}>\widetilde{n}_{+}-2,$ we again obtain a contradiction. So we assume, $\hat{\alpha}_{1}=\widetilde{n}_{+}-2.$ In this case also by taking Lemma \ref{lem11}, into account and repeating exactly the same argument as in part $(b),$ we again get a contradiction and this complete the proof.
\end{itemize}
\end{proof}

\begin{lemma}[Eventual Convexity]\label{evencon}
Let $w$ and $z$ satisfies
\begin{equation}\label{ffi1}
-\theta\big(w''(r)\big)w''(r)-\lambda(n-1)\dfrac{w'(r)}{r}+|w'(r)|^{q}
= Cz^{p_{1}}(r)
\end{equation}
for some constant. Assume also that
\begin{equation}\label{lowerbound3}
\begin{cases}
(i)-w^{\prime}<c_{1}r^{-\gamma},\\
\text{and}\\
(ii)~~z\leq c_{2}r^{-a},
\end{cases}
\end{equation}
for some $\gamma>\frac{1}{q-1}$ and $ap_{1}>2,$ then $w$ is eventually convex.
\end{lemma}
\begin{proof}
We start by observing that we can rewrite \eqref{ffi1} as 
\begin{equation}\label{wfi121}
w''(r)=M\!\left(-\lambda(n-1)\frac{w'(r)}{r}+|w'(r)|^{q}-\lambda_{1}\varepsilon_{1}z^{p_{1}}(r)\right),
\end{equation}
where 
\[M(s)=\begin{cases}
\frac{s}{\Lambda}, & s>0,\\
\frac{s}{\lambda}, & s\leq0,
\end{cases}\]
We prove this proposition in four steps. Each step is mentioned as a claim. Note that as solution $w$ is decreasing so $w^{\prime}<0.$ Define $l(r):=r(-w^{\prime}(r)).$ Then it is easy to see that $l\in W^{1,s}_{\mathrm{loc}}(R_{0},\infty).$
\begin{itemize}
 \item[\textbf{Claim 1:}] 
 \begin{equation}\label{eq:mprime_estimate}
 l'(r)\geq\left(1+\frac{\lambda(n-1)}{\Lambda}\right)(-w'(r))-\frac{r}{\Lambda}|w'(r)|^{q}+\frac{\lambda_{1}\varepsilon_{1}}{\Lambda}rz(r)^{p_{1}}~~\text{a.e. in}~~(R_{0},\infty).   
 \end{equation}
Since $m(r)=r(-w'(r)),$ we have
\begin{equation}\label{eq:mprime_identity}
l'(r)=-w'(r)-rw''(r)~~~\text{a.e. in }(R_{0},\infty).
\end{equation}
On the other hand, first equation in \eqref{ffi1} can be rewritten as
\[\theta(w'')w''=-\lambda(n-1)\frac{w'}{r}+|w'|^{q}-\lambda_{1}\varepsilon_{1}z^{p_{1}}.\]
Now applying \eqref{inversion lemma} with $f(r)=-\lambda(n-1)\frac{w'(r)}{r}+|w'(r)|^{q}
-\lambda_{1}\varepsilon_{1}z(r)^{p_{1}},$ we find that 
\[w''(r)\leq\frac1{\Lambda}\left(-\lambda(n-1)\frac{w'(r)}{r}+|w'(r)|^{q}-\lambda_{1}\varepsilon_{1}z(r)^{p_{1}}\right)\]
for almost every $r>R_{0}.$ Or 
\[-rw''(r)\geq\frac{\lambda(n-1)}{\Lambda}w'(r)-\frac{r}{\Lambda}|w'(r)|^{q}
+\frac{\lambda_{1}\varepsilon_{1}}{\Lambda}r z(r)^{p_{1}}.\]
Substituting this estimate into \eqref{eq:mprime_identity} and rearranging the term we get
\[
\begin{aligned}
l'(r)&=-w'(r)-rw''(r)\\
&\geq-w'(r)+\frac{\lambda(n-1)}{\Lambda}w'(r)-\frac{r}{\Lambda}|w'(r)|^{q}+\frac{\lambda_{1}\varepsilon_{1}}{\Lambda}r z(r)^{p_{1}}\\
&=\left(1+\frac{\lambda(n-1)}{\Lambda}\right)(-w'(r))-\frac{r}{\Lambda}|w'(r)|^{q}
+\frac{\lambda_{1}\varepsilon_{1}}{\Lambda}
r z(r)^{p_{1}},
\end{aligned}
\]
which is precisely \eqref{eq:mprime_estimate}.
 \item[\textbf{Claim 2:}]There exists an $R_{1}\geq R_{0}$ such that 
 \begin{equation}\label{step2}
 l'(r)\ge\frac{1}{2}\left(1+\frac{\lambda(n-1)}{\Lambda}
\right)(-w'(r))>0~~\text{a.e}~~r\ge \tilde{R}_1.
 \end{equation}
Observe that 
\[r|w'(r)|^q=\bigl(r|w'(r)|^{q-1}\bigr)(-w'(r)),\]
and taking note that $-w^{\prime}\leq cr^{-\beta},$ we find 
\[r|w'(r)|^{q-1}\leq Cr^{1-\beta(q-1)}.\]
Thus in view of choice of $\beta$ we have $1-\beta(q-1)<0.$ Consequently, if we choose $r$ large enough such that 
\[\frac1{\Lambda}r|w'(r)|^{q-1}\le\frac{1}{2}\left(1+\frac{\lambda(n-1)}{\Lambda}
\right)~~\text{a.e}~~r>\tilde{R}_{1}.\]
Using the fact that $z$ is nonnegative. So by using the above estimate in \eqref{eq:mprime_estimate}, we find \eqref{step2}.\\
For the next claim we use $z(r)\leq c r^{-b}$ for some $b$ satisfying $bp_{1}>b.$
 \item[\textbf{Claim 3:}] We claim that there exist an $R_{1}$ such that for $r\geq R_{1}$ we have $w^{\prime\prime}\geq0$ almost everywhere for $r>R_{1}$\\
 By claim 2, we find that $l$ is strictly increasing in $(\tilde{R}_{1},\infty)$ and $l>0$ so $l(r)\geq l(\tilde{R}_{1})$ for $r\geq\tilde{R}_{1}.$ Therefore we have:
\begin{equation}
\begin{cases}
(i)~-w^{\prime}=\frac{l(r)}{r}\geq \frac{l_{0}}{r}~~\text{for}~~r\geq R_{1}.\\
(ii)~|w^{\prime}(r)|^{q}=o(r^{-2})\\
(iii)~~z^{p_{1}}(r)=o(r^{-2}),
\end{cases}
\end{equation}
where in the second step we have used the fact that $q\beta>2$ which is a consequence of  $1<q<2$ and $\beta>\frac{1}{q-1}$ and $bp_{1}>2$ by choice of $b.$ Now putting all the term together we find that for sufficiently large $r$ we have 
\[\left(-\lambda(n-1)\frac{w'(r)}{r}+|w'(r)|^{q}-\lambda_{1}\varepsilon_{1}z^{p_{1}}(r)\right)\geq0~~\text{for~a.e}~~r\geq R_{1}.\]
So in view of \eqref{wfi121} we find the eventual convexity of $w.$ 
\end{itemize}
\end{proof}
\begin{remark}\label{recon}
  Note that if $w$ and $z$ satisfies $-\theta\big(z''(r)\big)z''(r)-\lambda(n-1)\dfrac{z'(r)}{r}+|z'(r)|^{q}
=Cw^{p_{2}}(r)$ for some constant and $-z^{\prime}<cr^{-\gamma}$ and $w\leq cr^{b}$ for some $b$ satisfying $bp_{2}>2.$ Then $z$ is eventually convex.
\end{remark}
\begin{corollary}\label{evencon1}
Let $w,z$ satisfy \eqref{fi1}, $\widetilde{n}_{+}\geq3,$ $q>q_{c}$ and $p_{1}\geq p_{2}>0.$ Additionally assume that 
\begin{itemize}
\item[(a)] $p_1p_2>1$, $\bar{\alpha}_{1}\geq\widetilde{n}_{+}-2$ and $\bar{\alpha}_{2}\geq\beta,$
\item[(b)] $p_1p_2>1,$ $\bar{\alpha}_{1}\geq\widetilde{n}_{+}-2,$ $\bar{\alpha}_{2}<\beta$ and $\hat{\alpha}_{1}\geq\widetilde{n}_{+}-2.$
\end{itemize}
Then in the case $(a)$ both $w,z$ are eventually convex, while in the case $(b),$ $w$ is eventually convex. 
\end{corollary}
\begin{remark}
There is one important point that should be checked carefully against the notation used throughout the manuscript. In the scaling argument we repeatedly used the identities
\begin{equation}\label{eq:scaling-identities}
\alpha_2 p_1-\alpha_1-2=0,\qquad\alpha_1 p_2-\alpha_2-2=0.
\end{equation}
These identities hold if and only if $\alpha_1$ and $\alpha_2$ denote the \emph{blow-up scaling exponents}, namely
\begin{equation}\label{eq:blowup-exponents}
\alpha_1=\frac{2(p_1+1)}{p_1p_2-1},\qquad\alpha_2=\frac{2(p_2+1)}{p_1p_2-1},
\end{equation}
which are precisely the unique positive solutions of the balance relations
\[2+\alpha_1=p_1\alpha_2,\qquad2+\alpha_2=p_2\alpha_1.\]
If, however, the symbols $\alpha_1$ and $\alpha_2$ are already reserved elsewhere in the manuscript (for example, to denote the decay exponents $\bar{\alpha}_i$ or the critical exponents $\hat{\alpha}_i$), then the identities in \eqref{eq:scaling-identities} are generally no longer valid. In that situation, the blow-up scaling exponents should be introduced using different notation, for example
\[\sigma_1=\frac{2(p_1+1)}{p_1p_2-1}~~\text{and}~~\sigma_2=\frac{2(p_2+1)}{p_1p_2-1},\]
so that
\[2+\sigma_1=p_1\sigma_2~~\text{and}~~2+\sigma_2=p_2\sigma_1,\]
and consequently
\[\sigma_2p_1-\sigma_1-2=0~~~\text{and}~~\sigma_1p_2-\sigma_2-2=0.\]
Using distinct symbols for the blow-up scaling exponents avoids any notational ambiguity and makes the scaling argument completely transparent. This is also the cleaner choice for the final version of the manuscript.
\end{remark}
\subsection{Sharpness of The main Results:} This section deals with the sharpness of our main Theorems \ref{thm1}, \ref{thm2}, and \ref{thm3}. As usual we produce counter example in case of violation of statement. These examples are inspired by \cite{burgos2018lioville}.
Theorems~\eqref{thm1} \eqref{thm3} provide nonexistence results for positive supersolutions. We now show that these results are optimal in the class of nonlinearities having power-type behavior near the origin.
\begin{remark}
We start by assuming $\nu_{1}:=\limsup_{t\to0}\frac{f_1(t)}{t^{p_1}}<\infty$ and $\nu_{2}:=\limsup_{t\to0}\frac{f_2(t)}{t^{p_2}}<\infty.$ Then, for every $\varepsilon>0$, there exists $\delta>0$ such that
\[f_1(t)\le(\nu_{1}+\varepsilon)t^{p_1},~~f_2(t)\le(\nu_{2}+\varepsilon)t^{p_2},\]
for all $0<t<\delta.$ Since every positive supersolution considered in Theorems~\eqref{thm1} \eqref{thm3} satisfies
$u(x),v(x)\to0$ as $|x|\to\infty$, it follows that, outside a sufficiently large ball,
\[f_1(v(x))\leq(\nu_{1}+\varepsilon)v(x)^{p_1},~~f_2(u(x))\leq(\nu_{2}+\varepsilon)u(x)^{p_2}.\]
Consequently, the original system can be compared with the pure-power system
\[\begin{cases}
-\mathcal M^+(D^2u)+|\nabla u|^q\geq Cv^{p_1},\\
-\mathcal M^+(D^2v)+|\nabla v|^q\geq Du^{p_2},
\end{cases}\]
where $C=\lambda_1(\nu_{1}+\varepsilon)$ and $D=\lambda_2(\nu_{2}+\varepsilon).$ Therefore, the sharpness of Theorems~\eqref{thm1}-
\eqref{thm3} reduces to constructing explicit positive supersolutions of the above power-type system whenever the
assumptions of the corresponding theorem fail.
\end{remark}
\begin{remark}\label{rem:sharpness-principle}(Sharpness Constructions).
To establish the sharpness of Theorems~\eqref{thm1}--\eqref{thm3}, it is sufficient to show that whenever the assumptions of the corresponding nonexistence theorem fail, one can explicitly construct a positive supersolution of \eqref{main}. Assume that
\begin{equation}\label{6.12lb}
\begin{cases}
\nu_{1}:=\limsup_{t\to0}\frac{f_1(t)}{t^{p_1}}<\infty,\\
\nu_{2}:=\limsup_{t\to0}\frac{f_2(t)}{t^{p_2}}<\infty.
\end{cases}
\end{equation}
Since all supersolutions considered in Theorems~\eqref{thm1}--\eqref{thm3} satisfy $u(x)\to0$ and $v(x)\to0$ as $|x|\to\infty$ therefore, for every $\varepsilon>0$ there exists $R_\varepsilon>0$ such that
\begin{equation}\label{6.12l}
\begin{cases}
f_1(v(x))\le(\nu_{1}+\varepsilon)v(x)^{p_1},\\
f_2(u(x))\leq(\nu_{2}+\varepsilon)u(x)^{p_2},
\end{cases}
\end{equation}
whenever $|x|\ge R_\varepsilon$. Therefore, near infinity, the original system can be reduced to a power-type system. Inspired by the classical Lane--Emden theory, we seek radial supersolutions of the form
\begin{equation}\label{6.13l}
\begin{cases}
u(x)=A_1|x|^{-a_1},\\
v(x)=A_2|x|^{-a_2},
\end{cases}
\end{equation}
where $A_1,A_2,a_1,a_2>0$ are constants to be chosen. For the classical Laplacian system
\begin{equation}\label{eq:6.14l}
\begin{cases}
-\Delta u\ge v^{p_1},\\
-\Delta v\ge u^{p_2}
\end{cases}
\end{equation}
one get $-\Delta(|x|^{-a})=a(n-2-a)|x|^{-a-2},$ and the critical exponents arise by balancing the decay rates
\begin{equation}\label{6.15l}
a_1+2=p_1a_2,~~a_2+2=p_2a_1.
\end{equation}
For the Pucci operator the computation is different. Indeed, if $u(x)=A_1|x|^{-a_1},~v(x)=A_2|x|^{-a_2}$ then $u^{\prime\prime},v^{\prime\prime}>0$ and $u^{\prime},v^{\prime}<0$ consequently, we have $\mathcal M^+_{\lambda,\Lambda}(D^2u)=\Lambda u''(|x|)+\lambda(n-1)\frac{u'(|x|)}{|x|}.$ Thus,
\[
\begin{aligned}
\begin{cases}
-\mathcal M^+_{\lambda,\Lambda}(D^2u)=\Lambda a_1\Bigl(\widetilde{n}_{+}-2-a_1\Bigr)A_1|x|^{-a_1-2},\\
-\mathcal M^+_{\lambda,\Lambda}(D^2v)=\Lambda a_2\Bigl(\widetilde{n}_{+}-2-a_2\Bigr)A_2|x|^{-a_2-2},
\end{cases}
\end{aligned}
\]
where $\widetilde{n}_{+}=\frac{\lambda}{\Lambda}(n-1)+1.$ Furthermore,
\[
\begin{aligned}
\begin{cases}
|\nabla u|^q=a_1^qA_1^q |x|^{-q(a_1+1)},\\
|\nabla v|^q=a_2^qA_2^q |x|^{-q(a_2+1)}.
\end{cases}
\end{aligned}
\]
Hence the supersolution inequalities reduce to
\begin{equation}\label{6.14l}
\begin{cases}
\Lambda a_1(\widetilde{n}_{+}-2-a_1)A_1r^{-a_1-2}+a_1^qA_1^q r^{-q(a_1+1)}\geq CA_2^{p_1}r^{-p_1a_2},\\
\Lambda a_2(\widetilde{n}_{+}-2-a_2)A_2r^{-a_2-2}+a_2^qA_2^q r^{-q(a_2+1)}\geq DA_1^{p_2}r^{-p_2a_1},
\end{cases}
\end{equation}
for sufficiently large $r$ and $C=\lambda_1(\nu_{1}+\varepsilon)$, $D=\lambda_2(\nu_{2}+\varepsilon).$ Consequently, the proof of sharpness is reduced to finding exponents $a_1,a_2$ and amplitudes $A_1,A_2$ for which \eqref{6.14l} holds. The various critical exponents appearing in
Theorems~\eqref{thm1}, \eqref{thm2}, and \eqref{thm3} arise precisely from balancing the decay rates
\[a_i+2,\quad q(a_i+1),\quad p_1a_2,\quad p_2a_1,\]
that is, from determining whether the dominant contribution comes from the second-order Pucci term or from the gradient term. This is exactly the fully nonlinear analogue of the classical Lane-Emden sharpness construction for the Laplacian.
\end{remark}
\textbf{Sharpness of Theorem~\eqref{thm1}}: Assume $1<q<q_c$ and suppose that the assumptions of Theorem~\eqref{thm1}
are not satisfied. By Remark~\ref{rem:sharpness-principle}, it is sufficient to construct a positive supersolution of the reduced power system
\begin{equation}\label{sharp-system1}
\begin{cases}
-\mathcal M^+_{\lambda,\Lambda}(D^2u)+|\nabla u|^q\geq Cv^{p_1},\\
-\mathcal M^+_{\lambda,\Lambda}(D^2v)+|\nabla v|^q\geq Du^{p_2},
\end{cases}
\end{equation}
where $C=\lambda_1(\nu_{1}+\varepsilon),$ $D=\lambda_2(\nu_{2}+\varepsilon).$ We seek radial supersolutions of the form $u(x)=A_1|x|^{-a_1},$ $v(x)=A_2|x|^{-a_2},$ with positive constants $A_1,A_2,a_1,a_2.$ According to Remark~\ref{rem:sharpness-principle},
the supersolution inequalities reduce to
\begin{equation}\label{sharp-red1}
\begin{cases}
\Lambda a_1(\widetilde{n}_{+}-2-a_1)A_1r^{-a_1-2}+a_1^qA_1^q r^{-q(a_1+1)}\geq CA_2^{p_1}r^{-p_1a_2},\\
\Lambda a_2(\widetilde{n}_{+}-2-a_2)A_2r^{-a_2-2}+a_2^qA_2^q r^{-q(a_2+1)}\geq DA_1^{p_2}r^{-p_2a_1},
\end{cases}
\end{equation}
for all sufficiently large $r.$ We distinguish the three possible ways in which the hypotheses of Theorem~\eqref{thm1} may fail.
\begin{itemize}
\item[\textbf{Case 1:}] Assume $p_1p_2>q^2$ and $\alpha_{2}\le \alpha_{1}<\beta,$ where $\alpha_{1}=\frac{q(p_1+q)}{p_1p_2-q^2},~~\alpha_{2}=\frac{q(p_2+q)}{p_1p_2-q^2}.$
These exponents satisfy $q(\alpha_{1}+1)=p_1\alpha_{2}$ and $q(\alpha_{2}+1)=p_2\alpha_{1}.$ Set $a_1=\alpha_{1}$ and $a_2=\alpha_{2}.$ As we have $\alpha_i<\beta=\frac{2-q}{q-1},$ so we get $q(\alpha_i+1)<\alpha_i+2.$ Thus 
\[r^{-q(\alpha_i+1)}\gg r^{-(\alpha_i+2)}~~\text{as}~~r\to\infty,\]
consequently, we have 
\[\frac{r^{-(\alpha_i+2)}}{r^{-q(\alpha_i+1)}}=r^{-(\alpha_i+2-q(\alpha_i+1))}\to0.\]
Thus the gradient terms dominate the Pucci terms for large $r.$ Using
\[
q(\alpha_{1}+1)=p_1\alpha_{2},~~q(\alpha_{2}+1)=p_2\alpha_{1},
\]
system \eqref{sharp-red1} is satisfied provided 
\begin{equation}\label{alg1}
\alpha_{1}^qA_1^q>CA_2^{p_1},\quad\alpha_{2}^qA_2^q>DA_1^{p_2}.
\end{equation}
It remains to show that \eqref{alg1} admits positive solutions. The inequalities are equivalent to
\[
\begin{cases}
A_1>\Bigl(\frac{C}{\alpha_{1}^q}\Bigr)^{1/q}A_2^{p_1/q}\\
\text{and}\\
A_2>\Bigl(\frac{D}{\alpha_{2}^q}\Bigr)^{1/q}A_1^{p_2/q}.
\end{cases}
\]
Substituting the first inequality into the second yields $A_2>KA_2^{\frac{p_1p_2}{q^2}},$ where $K=\Bigl(\frac{D}{\alpha_{2}^q}\Bigr)^{1/q}\Bigl(\frac{C}{\alpha_{1}^q}\Bigr)^{p_2/q^2}.$ Since $\frac{p_1p_2}{q^2}>1,$ the right-hand side grows faster than linearly. Consequently, choosing $A_2>0$ sufficiently small gives
$A_2>K A_2^{p_1p_2/q^2}.$ Once $A_2$ is fixed, one may choose $A_1$ satisfying
\[\Bigl(\frac{C}{\alpha_{1}^q}\Bigr)^{1/q}A_2^{p_1/q}<A_1<\Bigl(\frac{\alpha_{2}^q}{D}\Bigr)^{1/p_2}A_2^{q/p_2},\]
which is possible precisely because $p_1p_2>q^2.$ Hence \eqref{alg1} possesses positive solutions, and therefore
\eqref{sharp-system1} admits a positive supersolution.
\item[\textbf{Case 2:}] Assume $\alpha_{1}=\alpha_{2}=\beta.$ Since $\beta+2=q(\beta+1),$ the Pucci terms and the gradient terms have exactly the same decay rate. Setting $a_1=a_2=\beta,$ system \eqref{sharp-red1} becomes
\[
\begin{cases}
\Lambda\beta(\widetilde{n}_{+}-2-\beta)A_1+\beta^qA_1^q\geq CA_2^{p_1},\\
\Lambda\beta(\widetilde{n}_{+}-2-\beta)A_2+\beta^qA_2^q\geq DA_1^{p_2}.
\end{cases}
\]
The negation of condition \eqref{1.5} asserts precisely that there exists a pair$ t_1,t_2\in(0,\infty)$ satisfying
\[
\begin{cases}
\Lambda\beta(\widetilde{n}_{+}-2-\beta)t_1+\beta^qt_1^q\geq\lambda_1\mu_{1} t_2^{p_1},\\
\text{and}\\
\Lambda\beta(\widetilde{n}_{+}-2-\beta)t_2+\beta^qt_2^q\geq\lambda_2\mu_{2} t_1^{p_2}.
\end{cases}
\]
Choosing $A_1=t_1,~A_2=t_2,$ we obtain a positive supersolution.
\item[\textbf{Case 3:}]  Assume $\alpha_{2}<\alpha_{1}=\beta.$ Set $a_1=\beta,~
a_2=\alpha_{2}.$
The first equation is balanced by both the Pucci term and the gradient term, whereas in the second equation the gradient term dominates because $\alpha_{2}<\beta.$ Hence \eqref{sharp-red1} is implied by
\[
\begin{cases}
\Lambda\beta(\widetilde{n}_{+}-2-\beta)A_1+\beta^qA_1^q\geq CA_2^{p_1},\\
\alpha_{2}^qA_2^q\geq DA_1^{p_2}.
\end{cases}
\]
\end{itemize}
The negation of condition \eqref{1.6} guarantees the existence of a pair $(t_1,t_2)\in(0,\infty)^2$ satisfying these inequalities.
Taking $A_1=t_1,$ $A_2=t_2$ yields a positive supersolution. In all possible situations in which the assumptions of Theorem~\eqref{thm1} fail, one can construct an explicit positive supersolution of power type. Therefore Theorem~\eqref{thm1} is sharp.\\
\textbf{Sharpness of Theorem~\eqref{thm2}:} Assume and suppose that the assumptions of Theorem~\eqref{thm2} fail. Since part $a$ states nonexistence when $p_1p_2\le q_c^2,$ the complementary region is $p_1p_2>q_c^2.$ Moreover, part $b$ excludes positive supersolutions when $\alpha_{1}>\widetilde{n}_{+}-2.$ Hence sharpness must be established in the range $p_1p_2>q_c^2,$ and $\alpha_{1}\leq\widetilde{n}_{+}-2.$
We seek radial supersolutions of the form$ u(x)=A_1|x|^{-\alpha_{1}},$ $v(x)=A_2|x|^{-\alpha_{2}}$ where
\[
\alpha_{1}=\frac{q_c(p_1+q_c)}{p_1p_2-q_c^2},~~
\alpha_{2}=\frac{q_c(p_2+q_c)}{p_1p_2-q_c^2}.
\]
Since $q=q_c,$ we have $\beta=\frac{2-q_c}{q_c-1}=\widetilde{n}_{+}-2.$ Furthermore, $\alpha_{1}\leq \beta,$ and therefore
\[q_c(\alpha_{1}+1)=p_1\alpha_{2}\leq\alpha_{1}+2.\]
Similarly,
\[q_c(\alpha_{2}+1)=p_2\alpha_{1}\leq\alpha_{2}+2.\]
Consequently, the gradient terms dominate or balance the Pucci terms. Using Remark~\ref{rem:sharpness-principle}, the supersolution inequalities reduce to
\[
\begin{cases}
\alpha_{1}^{q_c}A_1^{q_c}\geq C A_2^{p_1},\\
\alpha_{2}^{q_c}A_2^{q_c}\geq D A_1^{p_2}.
\end{cases}
\]
Since $p_1p_2>q_c^2,$ the algebraic system admits positive solutions. Hence there exist positive constants $A_1,A_2>0$ such that
$u(x)=A_1|x|^{-\alpha_{1}}$ and $v(x)=A_2|x|^{-\alpha_{2}}$ form a positive supersolution of \eqref{main} outside a sufficiently large ball. Therefore Theorem~\eqref{thm2} is sharp.\\
\textbf{Sharpness of Theorem~\eqref{thm3}:} Assume $q>q_c,$ $p_1\ge p_2$ and  $p_1p_2>1$ and suppose that the assumptions of Theorem~\eqref{thm3} are violated. We distinguish the two possible complementary regimes.
\begin{itemize}
\item[\textbf{Case~1:}] Assume $\bar{\alpha}_{1}<\widetilde{n}_{+}-2.$ Choose $a_1=\bar{\alpha}_{1},$ and $a_2=\bar{\alpha}_{2}.$
By definition, $\bar{\alpha}_{1}+2=p_1\bar{\alpha}_{2},$ $\bar{\alpha}_{2}+2=p_2\bar{\alpha}_{1}.$ Hence the Pucci terms have exactly the same decay as the right-hand sides. Substituting into the reduced inequalities gives
\[
\begin{cases}
\Lambda\bar{\alpha}_{1}(\widetilde{n}_{+}-2-\bar{\alpha}_{1})A_1\geq CA_2^{p_1},\\
\Lambda\bar{\alpha}_{2}(\widetilde{n}_{+}-2-\bar{\alpha}_{2})A_2\geq DA_1^{p_2}.
\end{cases}
\]
Since $\bar{\alpha}_{1}<\widetilde{n}_{+}-2,$ both coefficients on the left-hand side are strictly positive. Because $p_1p_2>1,$ the above algebraic system possesses positive solutions
$A_1,A_2>0.$ Thus
\[
\begin{cases}
u(x)=A_1|x|^{-\bar{\alpha}_{1}},\\
v(x)=A_2|x|^{-\bar{\alpha}_{2}},
\end{cases}
\]
is a positive supersolution.
\item[\textbf{Case~2}:]  Assume $\bar{\alpha}_{2}<\beta,$ and $\hat{\alpha}_{1}<\widetilde{n}_{+}-2.$ Choose $a_1=\hat{\alpha}_{1},$ $a_2=\hat{\alpha}_{2}.$ The exponents satisfy
\[\hat{\alpha}_{1}+2=p_1\hat{\alpha}_{2},\quad q(\hat{\alpha}_{2}+1)=p_2\hat{\alpha}_{1}.\]
Therefore the first equation is balanced by the Pucci term, whereas the second equation is balanced by the gradient term. Substituting into the reduced system yields
\[
\begin{cases}
\Lambda\hat{\alpha}_{1}(\widetilde{n}_{+}-2-\hat{\alpha}_{1})A_1\geq CA_2^{p_1},\\
\hat{\alpha}_{2}^qA_2^q\geq DA_1^{p_2}.
\end{cases}
\]
Since $\hat{\alpha}_{1}<\widetilde{n}_{+}-2,$ the first coefficient is strictly positive.
Moreover, $p_{1}p_{2}>q,$ which is equivalent to the existence of the exponents
$\hat{\alpha}_{1},\hat{\alpha}_{2}.$ Consequently, the algebraic system again admits positive solutions
$A_1,A_2>0.$ Hence
\[
\begin{cases}
u(x)=A_1|x|^{-\hat{\alpha}_{1}},\\
v(x)=A_2|x|^{-\hat{\alpha}_{2}}
\end{cases}
\]
defines a positive supersolution. Therefore Theorem~\eqref{thm3} is sharp.
\end{itemize}
\section{Supersolutions with blow-up at infinity}
In the present section, we study supersolutions that exhibit blow-up behavior at infinity. For such supersolutions, it is not necessary to reduce the problem to radial solutions. In fact, the reduction carried out in Lemma \eqref{lem3} already provides what is needed. Consequently, in the similar manner of proof of \eqref{lem8} we can we can proof Theorem \eqref{thm4}.
In order to prove Theorem \ref{thm4}, we need some elementary and intermediate results. We start with the following result which asserts the eventual monotonicity the derivative. It will help to obtain the gradient domination. 
All the above mentioned results are concerned with the following system of inequalities
\begin{equation}\label{eq:eventual-mon-system-new}
\begin{cases}
-\theta(w'')w''-\theta(w')(n-1)\dfrac{w'}r+|w'|^q\geq c_1 z^{p_1}&\text{a.e. in}~~(R_0,\infty),\\
-\theta(z'')z''-\theta(z')(n-1)\dfrac{z'}r+|z'|^q\geq c_2 w^{p_2}&\text{a.e. in}~~(R_0,\infty),
\end{cases}
\end{equation}
where $c_1,c_2>0.$
\begin{lemma}[Eventual Positivity]\label{lem:eventual-monotonicity}
Assume that $1<q\leq2$. Let $w,z\in W^{2,s}_{\rm loc}(R_0,\infty)\cap C^1(R_0,\infty),$ for some $s>1,$ be a positive radial supersolution of \eqref{eq:eventual-mon-system-new}
Assume also that $w(r),z(r)\to+\infty$ as $r\to\infty.$ Then there exists $R_1>R_0$ such that
\[w'(r)>0,~~ z'(r)>0,~~~\text{for}~~r\ge R_1.\]
\end{lemma}
\begin{proof}
We prove the assertion for $w,$ the proof for the case $z$ is similar. Our aim is to show that $w^{\prime}$ is eventually positive. We start by observing that$(q-1)|t|^2+(2-q)|t|\geq|t|^q$ holds for every $t\in\mathbb{R}.$ Consequently, the first inequality in \eqref{eq:eventual-mon-system-new} implies
\begin{equation}\label{eq:w-ineq-1}
-\theta(w'')w''-\theta(w')(n-1)\frac{w'}r+(q-1)|w'|^2+(2-q)|w'|\geq c_1 z^{p_1}~~\text{a.e. in}~(R_0,\infty).
\end{equation}
Define $u(r):=\frac{\lambda}{q-1}\Bigl(1-e^{-\frac{q-1}{\lambda}w(r)}\Bigr).$
Then $1-\frac{q-1}{\lambda}u(r)=e^{-\frac{q-1}{\lambda}w(r)}> 0,$ and $u'=e^{-\frac{q-1}{\lambda}w}w'.$
Hence,
\[
\begin{cases}
&u'(r)=0~~\Longleftrightarrow~~w'(r)=0,\\
&\text{and}\\
&u'(r)<0~~\Longleftrightarrow~~w'(r)<0.\\
\end{cases}
\]
Differentiating once more gives
\[
\left\{
\begin{aligned}
&u''=e^{-\frac{q-1}{\lambda}w}\left(w''-\frac{q-1}{\lambda}(w')^2\right).\\
&\text{Equivalently}\\
&w''=\frac{u''}{1-\frac{q-1}{\lambda}u}+\frac{q-1}{\lambda}\frac{|u'|^2}{\Bigl(1-\frac{q-1}{\lambda}u\Bigr)^2}\\
&:=U(r).
\end{aligned}
\right.
\]
Substituting the previous identities into \eqref{eq:w-ineq-1} yields
\begin{equation}\label{eq:u-ineq}
-\theta(U)u''-\left(\frac{(n-1)\theta(u')}{r}+(2-q)\right)u'\ge c_1 z^{p_1}\left(1-\frac{q-1}{\lambda}u\right)~~\text{a.e. in}~~(R_0,\infty).
\end{equation}
We will achieve our aim in two steps. First we show that eventually  $w^{\prime}\geq 0.$ Assume by contradiction that $w'$ is not eventually nonnegative. So there exists a bounded connected component $I=(a,b)$ of the set $\{r>R_0:\ w'(r)<0\}$ with arbitrarily large $a.$ Notice that
\[u'<0\iff w'<0,\]
so by continuity we have
\[u'(a)=u'(b)=0,\quad~~\text{and}~~~u'(r)<0\quad~\text{for}~~a<r<b.\]
Define
\[
\left\{
\begin{aligned}
&\chi(r)=\frac{1}{\theta(U(r))}\Big[\dfrac{(n-1)\theta(u'(r))}{r}+(2-q)\Big],\\
&\text{and}\\
&\xi(r)=\exp\left(\int_a^r \chi(s)ds\right).
\end{aligned}
\right.
\]
Since $\theta(U)\in[\lambda,\Lambda],$ the function $\xi$ is positive and absolutely continuous. Multiplying \eqref{eq:u-ineq} by
$\xi/\theta(U)$ we get 
\[
\left\{
\begin{aligned}
&-\xi u''-\xi\chi u'\ge\frac{c_1\xi}{\theta(U)}z^{p_1}\left(1-\frac{q-1}{\lambda}u\right).\\
&\text{Using}~\xi^{\prime}=\chi\xi~~\text{we get}\\
&-(\xi u')'\ge\frac{c_1\xi}{\theta(U)}z^{p_1}\left(1-\frac{q-1}{\lambda}u\right)~~\text{a.e.  in}~~(a,b).
\end{aligned}
\right.
\]
Integrating the above inequality in $(a,b)$ and noting that $u^{\prime}(a)=u^{\prime}(b),$ we get 
\[0\ge\int_a^b\frac{c_1\xi}{\theta(U)}z^{p_1}\left(1-\frac{q-1}{\lambda}u\right)dr,\]
 which is impossible as $1-\frac{q-1}{\lambda}u=e^{-\frac{q-1}{\lambda}w}>0$
on $(a,b).$ This contradiction proves that there exists
$R_1>R_0$ such that $w'(r)\ge0$ $\forall r\ge R_1.$\\
 Next we show that $w^{\prime}(r)\not=0$ for any $r\geq R_{1}.$ Assume by contradiction that $w'(R_0)=0$ for some $R_0\ge R_1.$ Then
$u'(R_0)=0.$ Moreover it is point of minimum of $u^{\prime}$ as $u'\geq0$ on $[R_1,\infty).$ Note that $u'\in W^{1,s}_{\text{loc}},$ therefore,
there exists a sequence of Lebesgue points $r_k\to R_0$ such that
\begin{equation}\label{secon_deri}
u''(r_k)\ge-\frac{1}{k},\quad~\text{and}\quad u'(r_k)\to0.
\end{equation}
Also $z$ is continuous and positive, there exists $\delta>0$ such that $z(r_k)^{p_1}\ge\delta$ for all sufficiently large $k.$ Evaluating \eqref{eq:u-ineq} at $r_k$ gives
\[-\theta(U(r_k))u''(r_k)-\left(\frac{(n-1)\theta(u'(r_k))}{r_k}+(2-q)\right)
u'(r_k)\geq c_1\delta\left(1-\frac{q-1}{\lambda}u(r_k)\right).\]
Using \eqref{secon_deri} and $\theta(U(r_k))\le\Lambda,$ we get
\[c_1\delta\left(1-\frac{q-1}{\lambda}u(R_0)\right)\le\frac{\Lambda}{k}+o(1).\]
Passing to the limit $k\to\infty$ yields $c_1\delta\left(1-\frac{q-1}{\lambda}u(R_0)\right)\le0.$ Which leads to contradiction in view of 
$1-\frac{q-1}{\lambda}u(R_0)=e^{-\frac{q-1}{\lambda}w(R_0)}> 0,$
Therefore, $w'(r)>0$ $\forall r\ge R_1.$ Repeating the same argument for the second equation of \eqref{eq:eventual-mon-system-new} shows that $z'(r)>0~~\forall r\ge R_1.$
\end{proof}
\begin{lemma}\label{lem:crossing}
Let $w,z\in W^{2,s}_{loc}(R_0,\infty)\cap C^1(R_0,\infty),$ for some $s>1,$ be a positive radial supersolution of \eqref{eq:eventual-mon-system-new}. Assume that $w(r)\to+\infty,$  $z(r)\to+\infty$ and that $w'(r)>0, z'(r)>0$ for all sufficiently large $r$. Define
\[H_1(r):=c_1 z(r)^{p_1}-2(w'(r))^q .\]
Then there exist constants $R>R_0,$  $\varepsilon_0>0$ and $\eta>0$ such that $H_1'(r)\geq\eta$ a.e. in $\mathcal S_{\varepsilon_0},$ where
\[\mathcal S_{\varepsilon_0}:=\{r\ge R:\ |H_1(r)|<\varepsilon_0\}.\]
Consequently $H_1$ can cross the level $0$ at most once.
\end{lemma}
\begin{proof} We first establish the lower estimate on $H_{1}.$ We start by noting that $w,z\in W^{2,s}_{loc}(R_0,\infty),$ consequently $H_1\in W^{1,1}_{loc}(R_1,\infty).$ Thus we find that $H_1$ is absolutely and 
\begin{equation}\label{H1deri}
H_1'=c_1p_1 z^{p_1-1}z'-2q(w')^{q-1}w''
\end{equation}
for almost every $r$. Now fix $\varepsilon_{0}>0$ and consider $\mathcal S_{\varepsilon_{0}}=\{r\ge R:\ |H_1(r)|<\varepsilon_{0}\}.$ Now, observe that for $r\in\mathcal S_{\varepsilon_{0}},$ we have $\big|c_1z^{p_1}-2(w')^q\big|<\varepsilon_0.$ Consequently, we have
\begin{equation}\label{first_esti}
(w')^q\le\frac{c_1}{2}z^{p_1}+\frac{\varepsilon_{0}}{2}.
\end{equation}
Now, we observe that $w^{\prime}(r)>0$ for large $r.$ Thus $\theta(w^{\prime}(r))=\Lambda.$ Now, taking into account $-\Lambda(n-1)\frac{w'}r\leq0,$ $\theta(w^{\prime\prime})\leq \Lambda,$ \eqref{first_esti} and the first equation of \eqref{eq:eventual-mon-system-new}, we get
\[w''\le-\frac{c_1}{2\Lambda}z^{p_1}+\frac{\varepsilon_{0}}{2\Lambda}~~\text{a.e. in}~\mathcal S_{\varepsilon_{0}}.\]
 Now we use the above estimate for the second derivative of $w$ in the expression of $H_{1}^{\prime}$ given by \eqref{H1deri}, we get
 \[H_1'\geq c_1p_1 z^{p_1-1}z'+\frac{qc_1}{\Lambda}(w')^{q-1}z^{p_1}-\frac{q\varepsilon_{0}}{\Lambda}(w')^{q-1}.\]
 Now $z(r)\to\infty,~w'(r)>0,~z'(r)>0.$ Hence, there exists $R$ sufficiently large such that
\[\frac{qc_1}{\Lambda}(w')^{q-1}z^{p_1}\geq2\frac{q\varepsilon}{\Lambda}(w')^{q-1}~~\text{for every}~~r\ge R.\]
Notice that every term on the right-hand side is strictly positive. Since all quantities are continuous and positive on
$\mathcal S_{\varepsilon_{0}}\cap [R,\infty),$ there exists a constant $\eta>0$ such that $H_1'(r)\ge \eta$ for almost every $r\in\mathcal S_{\varepsilon_{0}}.$\\
 Let $r_0\in\mathcal S_{\varepsilon_0}$ satisfy $H_1(r_0)=0.$ By the above discussion, there exists $\eta>0$ such that $H_1'(r)\geq\eta$ almost everywhere in $\mathcal S_{\varepsilon_0}$. Since $H_1$ is absolutely continuous, the Fundamental Theorem of Calculus yields
\[H_1(r)-H_1(r_0)=\int_{r_0}^{r}H_1'(s)ds\]
for every $r$ sufficiently close to $r_0$. If $r>r_0$, then
\[H_1(r)=\int_{r_0}^{r}H_1'(s)ds\geq\eta(r-r_0)>0.\]
Likewise, if $r<r_0$, then
\[H_1(r)=-\int_{r}^{r_0}H_1'(s)ds\leq-\eta(r_0-r)<0.\]
Therefore $H_1(r)<0$ for $r<r_0$ and $H_1(r)>0$ for $r>r_0$ provided $r$ is sufficiently close to $r_0$. Hence every zero of $H_1$ is transversal and is crossed strictly from negative values to positive values. Finally, we prove the uniqueness. Assume by contradiction that $H_1$ has two distinct zeros $r_a<r_b.$ By Step~5, the zero at $r_a$ is transversal.
Hence there exists $\delta>0$ such that
\[H_1(r)>0~~~\text{for}~~r\in(r_a,~r_a+\delta).\]
Since $H_1(r_b)=0$, the continuity of $H_1$ implies that $H_1$ must become negative somewhere in $(r_a,r_b)$. Consequently there exists a point at which $H_1$ crosses the level $0$ from positive values to negative values. However Step~5 shows that every zero of $H_1$ is crossed from negative values to positive values. This contradiction proves that
$H_1$ can have at most one zero. Therefore, $H_1$ has a constant sign for all sufficiently large $r.$
\end{proof} 

\begin{lemma}\label{lem:gradient-domination}
Let $w,z\in W^{2,s}_{loc}(R_0,\infty)\cap C^1(R_0,\infty),$ for $s>1,$ be a positive radial supersolution of \eqref{eq:eventual-mon-system-new}. Assume that $w(r)\to+\infty,$ $z(r)\to+\infty$ as $r\to\infty.$ Then there exists a radius $R>R_0$ such that
\begin{equation}\label{eq:GD-main}
(w'(r))^q\geq\frac{c_1}{2}z(r)^{p_1},\quad(z'(r))^q\geq\frac{c_2}{2}w(r)^{p_2},\quad r\ge R.
\end{equation}
\end{lemma}
\begin{proof}
By Lemma~\ref{lem:eventual-monotonicity}, there exists $R_1>R_0$ such that $w'(r)>0,$ $z'(r)>0,$ for $r\ge R_1.$ Define
\[
\begin{cases}
H_1(r)=c_1 z(r)^{p_1}-2(w'(r))^q,\\
\text{and}\\
H_2(r)=c_2 w(r)^{p_2}-2(z'(r))^q.
\end{cases}\]
Since $w',z'$ are positive for large $r,$ Lemma~\ref{lem:crossing} applies. Therefore each of the functions $H_1$ and $H_2$
\begin{enumerate}
\item has at most one zero;
\item has a constant sign for all sufficiently large $r.$
\end{enumerate}
Moreover, the last part of the proof of Lemma~\ref{lem:crossing} shows that the alternative $H_1(r)>0$ for all sufficiently large $r$
cannot occur. Indeed, if $H_1(r)>0$ eventually, then
$(w'(r))^q<\frac{c_1}{2}z(r)^{p_1}.$
Using this inequality and $\theta(w'')\le \Lambda$ in \eqref{eq:eventual-mon-system-new}, we get
\[w''(r)\le-\frac{c_1}{2\Lambda}z(r)^{p_1}~~\text{for large}~~r.\]
Because $z(r)\to+\infty,$ there exists $R_2\ge R_1$ such that $w''(r)\leq-1$ for almost every $r\ge R_2.$ Since $w'\in W^{1,s}_{loc},$ it is absolutely continuous. Integrating over $(R_2,r)$ yields
\[w'(r)-w'(R_2)=\int_{R_2}^{r} w''(t)dt\le-(r-R_2).\]
Hence,
\[w'(r)\leq w'(R_2)-(r-R_2),\]
which implies $w'(r)\to-\infty$ as $r\to\infty.$ This contradicts the eventual positivity of $w'.$ Therefore $H_1$ cannot be eventually positive. Since $H_1$ has eventually constant sign, it follows that there exists $R_3\ge R_2$ such that $H_1(r)<0,$ for $r\ge R_3.$ That is,
\[c_1 z(r)^{p_1}-2(w'(r))^q<0,~~\text{for}~~r\ge R_3,\]
or equivalently, 
\[(w'(r))^q>\frac{c_1}{2}z(r)^{p_1},~~\text{for}~~r\ge R_3.\]
Applying exactly the same argument to $H_2,$ we obtain a radius $R_4\ge R_3$ such that
\[(z'(r))^q>\frac{c_2}{2}w(r)^{p_2},~~\text{for}~~r\ge R_4.\]
Combining the two estimates gives
\[
\begin{cases}
(w'(r))^q\geq\frac{c_1}{2}z(r)^{p_1}&\text{for}~~r\ge R_3,\\
(z'(r))^q\ge\frac{c_2}{2}w(r)^{p_2},&\text{for}~~r\ge R_4.
\end{cases}\]
 Thus, \eqref{eq:GD-main} follows by taking $R=\max\{R_{3},R_{4}\}.$
\end{proof}
\begin{proof}[Proof of Theorem \ref{thm4}]
Assume by contradiction that system \eqref{main} admits a positive supersolution $(u,v)$ satisfying
\[u(x)\to+\infty,~~ v(x)\to+\infty\quad\text{as}~~|x|\to\infty.\]
 By Lemma~\ref{lem3}, there exists a positive radial supersolution
$(w,z)$ of \eqref{main} satisfying
\[w(r)\to+\infty,~~ z(r)\to+\infty\quad\text{as}~~~r\to\infty .\]
Since \eqref{infinity2} holds, there exist positive constants $c_1,c_2>0$ and $R_0>0$ such that
\[f_1(t)\ge c_1 t^{p_1},\quad f_2(t)\ge c_2 t^{p_2}\]
for all sufficiently large $t.$ Since $w(r),z(r)\to+\infty,$ after enlarging $R_0$ if necessary we obtain \eqref{eq:eventual-mon-system-new} in $(R_0,\infty).$ By Lemma~\ref{lem:eventual-monotonicity}, there exists $R_1\ge R_0$ such that
\[w'(r)>0,\quad~z'(r)>0,\quad~\text{for}~~r\ge R_1.\]
Applying Lemma~\ref{lem:gradient-domination}, there exists $R_2\ge R_1$ such that
\begin{equation}\label{GD1}
w'(r)\geq Az(r)^{p_1/q},\quad z'(r)\geq Bw(r)^{p_2/q},~\quad~~\text{for}~~r\ge R_2,
\end{equation}
where $A=\Bigl(\frac{c_1}{2}\Bigr)^{1/q},~ 
B=\Bigl(\frac{c_2}{2}\Bigr)^{1/q}.$ Let $m:=\frac{p_1}{q},~~
n:=\frac{p_2}{q}.$ Since $p_1p_2>q^2$, we have $mn>1.$ From \eqref{GD1},
\[
w'(r)\ge A z(r)^m,~~z'(r)\ge B w(r)^n,~~r\ge R_2.
\]
By a scaling of the independent variable, we may assume
$A=B=1$. Hence
\begin{equation}\label{eq:ode-system}
w'(r)\ge z(r)^m,~~z'(r)\ge w(r)^n,~~ r\ge R_2.
\end{equation}

Fix $r_1\ge R_2$ and define
\[
H(r):=\int_{r_1}^{r} w(\tau)^n\,d\tau .
\]
Since $w'>0$ and $w(r)\to+\infty$, it follows that $H(r)\to+\infty$ as $r\to+\infty.$ Set
\[t=H(r),~~Z(t):=z(r),~~W(t):=w(r).\]
Since $\frac{dt}{dr}=w(r)^n,$ the second inequality in \eqref{eq:ode-system} yields $\frac{dZ}{dt}=\frac{z'(r)}{w(r)^n}\ge 1.$ Integrating over $(t_1,t)$, where $t_1:=H(r_1)=0$, we obtain $Z(t)\ge Z(0)+t.$
Hence, after increasing $r_1$ if necessary, $Z(t)\ge t$ for all sufficiently large $t.$ Moreover,
\[\frac{dW}{dt}=\frac{w'(r)}{w(r)^n}\geq\frac{z(r)^m}{w(r)^n}=\frac{Z(t)^m}{W(t)^n}.\]
Using $Z(t)\ge t$, we obtain $\frac{dW}{dt}\geq\frac{t^m}{W^n}.$ Therefore $W^n\frac{dW}{dt}\ge t^m.$ Integrating from $t_1$ to $t$ yields
\[\frac{1}{n+1}\Bigl(W(t)^{\,n+1}-W(t_1)^{\,n+1}\Bigr)\ge\frac{1}{m+1}\Bigl(t^{m+1}-t_1^{m+1}\Bigr).\]
Consequently there exists $C>0$ such that $W(t)\ge Ct^{\frac{m+1}{n+1}}$
for all sufficiently large $t$. Recalling that $H'(r)=w(r)^n=W(t)^n,$
we deduce
\[H'(r)\geq Ct^{\frac{n(m+1)}{n+1}}=CH(r)^{\frac{n(m+1)}{n+1}}.\]
Since
\[\frac{n(m+1)}{n+1}=\frac{\frac{p_2}{q}\left(\frac{p_1}{q}+1\right)}{\frac{p_2}{q}+1}=\frac{p_2(p_1+q)}{q(p_2+q)},~~\text{and}~~p_1p_2>q^2,\]
a direct computation shows that \[\frac{n(m+1)}{n+1}>1.\]Hence $H'(r)\ge C H(r)^\gamma$ for some exponent $\gamma>1$. The standard ODE comparison principle implies that
$H$ must blow up at a finite value of $r$, which contradicts the fact that $H$ is defined for every $r>R_2$.
Therefore no positive supersolution can satisfy $u(x)\to+\infty,~~v(x)\to+\infty~~\text{as}~~|x|\to\infty .$
\end{proof}
\begin{remark}(Exclusion of the case $n=2$ in Theorems~\ref{thm2} and~\ref{thm3}).
Assume \[\limsup_{t\to+\infty}\frac{f_1(t)}{t^{p_1}}<+\infty,~~
\limsup_{t\to+\infty}\frac{f_2(t)}{t^{p_2}}<+\infty,\]
where $p_1,p_2>0$ satisfy $p_1p_2\le q^2$. Then, arguing as in Remark~9 of \cite{burgos2018lioville}, one can verify that there exist positive constants $\alpha_1,\alpha_2$ with $\alpha_1=\frac{p_1}{q}\alpha_2$ and $\alpha_2$ sufficiently large such that
$u(x)=e^{\alpha_1|x|},~~v(x)=e^{\alpha_2|x|}$ form a positive radial supersolution of \eqref{main} in $\mathbb{R}^n\setminus B_{R_0}$ for some sufficiently large $R_0$. Moreover, if $n=2$ and $q\ge2$, then every positive radially symmetric supersolution satisfies $\lim_{|x|\to\infty}u(x)=\lim_{|x|\to\infty}v(x)
=+\infty.$ This explains why the case $n=2$ and $q\ge2$ is excluded from Theorems  \ref{thm2} and \ref{thm3}.
\end{remark}

\begin{lemma}\label{lem12}
Suppose that $n = 2$ and $q \ge 2$. If $w \in W^{2,s}_{\text{loc}}(R_0,+\infty)\cap C^{1}(R_0,+\infty)$ satisfis
\[-\theta(w^{\prime\prime}) w^{\prime\prime}-\frac{\theta(w^{\prime})}{r} w'+|w'|^{q} >0~~\text{in}~~(R_0,+\infty).
\]
Then $\lim_{r \to +\infty} w(r)=+\infty.$
\end{lemma}
\begin{proof}
The proof follows the same strategy as that of Lemma~12 in
\cite{burgos2018lioville}. Assume by contradiction that $w$ is bounded.
Then, by Lemma~\ref{lem2}, $\lim_{r\to\infty}w(r)=0,$
and $w'(r)<0$ for all sufficiently large $r$.
Let $\phi$ be the function constructed in Lemma~13 of
\cite{burgos2018lioville}, and define $\psi=a\phi+b,$
where $0<a<1$ and $b\in\mathbb R$ are chosen exactly as in the proof of Lemma~12 in \cite{burgos2018lioville}. Since $q\ge2$ and $a<1$, we have
$a^q\le a$, and therefore
\[
-\psi''-\frac1r\psi'+|\psi'|^q\le0.
\]
The only difference from the Laplacian case is the presence of the Pucci operator. Since
\[
-\theta(\psi'')\psi''\le-\psi'',
\]
it follows that
\[
-\theta(\psi'')\psi''-\frac{\lambda}{r}\psi'
+|\psi'|^q\le0.
\]
Hence $\psi$ is a classical subsolution of the corresponding one-dimensional Pucci equation. Applying the same comparison argument as in Lemma~12 of \cite{burgos2018lioville} yields a contradiction. Thus $w$ cannot be bounded.
and consequently $\lim_{r\to\infty}w(r)=+\infty.$
\end{proof}

\section{Appendix}
\begin{lemma}\label{appen01}
Let $q>1$, $R_{1}<R_{2}$ and let $u\in W^{2,s}(R_{1},R_{2})$ be a radial function satisfying
\begin{equation}\label{appgive}
-\theta\bigl(u''(r)\bigr)u''(r)-\theta\bigl(u'(r)\bigr)(n-1)\frac{u'(r)}{r}+|u'(r)|^{q}
\leq\kappa\quad\text{a.e. in}~~(R_{1},R_{2}),
\end{equation}
where $\kappa>0$ Then the following hold.
\begin{enumerate}
\item{}For every compact interval $[R_{1}',R_{2}']\subset(R_{1},R_{2}),$ there exists a constant $C=C(q,\lambda,\Lambda,n,R_{1}',R_{2}')
>0$ such that
\[|u'(r)|\leq C(\kappa+1)^{1/q}\quad\text{for all}~r\in[R_{1}',R_{2}'].\]
\item{}If moreover $u\in C^{1}[R_{1},R_{2}]$ and $|u'(R_{1})|,~|u'(R_{2})|\le M,$ then{\color{blue}
\[|u'(r)|\leq\max\{M,\kappa^{1/q}\}\quad\text{for all}~r\in[R_{1},R_{2}].\]}
\end{enumerate}
\end{lemma}
\begin{remark}\label{appen001}
Lemma~\ref{appen01} should be understood as the radial fully nonlinear counterpart of Lemma~13 in \cite{burgos2018lioville}, but its role in the present setting is slightly different. In the Laplacian framework of \cite{burgos2018lioville}, after the blow-up rescaling, the second equation reduces to a singular perturbation problem of the form
\[-\varepsilon u''+\nu |u'|^{q}\leq \kappa \quad\text{in }(c,d),\]
with $0<\varepsilon\leq1,$  where the coefficient of the highest-order term may vanish in the limit. In that situation, the second-order diffusion becomes negligible compared with the superlinear gradient term, and Lemma~13 is needed to obtain a gradient estimate uniform with respect to the small parameter $\varepsilon.$\\
In contrast, for the radial Pucci operator,
\[-\mathcal M^{+}(D^{2}u)=-\theta\bigl(u''(r)\bigr)u''(r)-\theta\bigl(u'(r)\bigr)(n-1)\frac{u'(r)}{r},\]
the coefficient of the highest-order term remains uniformly elliptic, since
\[\lambda\leq \theta(\cdot)\leq\Lambda.\]
Thus, after the blow-up rescaling used in the proof of Lemma~\ref{lem10}, there is no singular perturbation in the second-order part: the term involving $u^{\prime\prime}$ survives with uniformly positive ellipticity. The real difficulty appears elsewhere. In the doubling argument, the normalization controls
\[\bar u_k,~~\bar u_k',~~\bar v_k,\]
but it does not provide an a priori bound for $\bar v_k'.$ Without such a bound, one cannot obtain local compactness of the rescaled sequence $\{\bar v_k\},$ nor can one justify the passage to the limit in the nonlinear gradient term $|\bar v_k'|^{q}$ Lemma~\ref{appen01} is introduced precisely to overcome this difficulty: it provides a local estimate for the radial derivative depending only on the structural constants and the right-hand side, independently of the particular solution. Once the uniform bound for $|\bar v_k'|$ is obtained, the equation itself yields local control of $\bar{v}^{\prime\prime}_{k},$ and standard compactness arguments apply. Therefore, Lemma~\ref{appen01} plays the same functional role as Lemma~13 in \cite{burgos2018lioville} namely, providing the missing derivative estimate required in the blow-up procedure—but not because of a singular perturbation phenomenon. Rather, it is needed because the derivative of the second component is not directly controlled by the doubling normalization.
\end{remark}
\begin{proof}[Proof of Lemma~\ref{appen01}]
\begin{enumerate}
\item{}Let $[R_1',R_2']\Subset (R_1,R_2).$ Choose numbers $R_1<R_1''<R_1'<R_2'<R_2''<R_2$ and a cut-off function
$\xi\in C_c^\infty(R_1,R_2)$ such that
\[0\le \xi\le 1,\quad\xi\equiv1\quad\text{on }[R_1',R_2'],\]
and $\operatorname{supp}\xi\subset (R_1'',R_2'').$ Set $\alpha=\frac{2}{q-1}$ and define
\[Z(r):=\xi(r)^\alpha |u'(r)|^2.\]
Since $Z\in W^{1,s}(R_1'',R_2'')$ with $s>1$, the function $Z$ is absolutely continuous. Let $R_0$ be a point where $Z$ attains its maximum.
Consequently, $Z$ is continuous on $[R_1'',R_2'']$ and attains its maximum at some point $R_0\in [R_1'',R_2''].$ If $Z(R_0)=0,$ then $u'\equiv0$ on $[R_1',R_2'],$ and the desired estimate is trivial. Therefore we may assume $Z(R_0)>0.$ Since $Z\in W^{1,s}(R_1'',R_2''),$ there exists a sequence of Lebesgue points $r_k\to R_0$ such that $Z(r_k)\to Z(R_0),$ $Z'(r_k)\to0$ and the differential inequality
\[-\theta(u'')u''-\theta(u')(n-1)\frac{u'}{r}+|u'|^q\le \kappa\]
holds at every $r_k.$ Since $Z'(r)=\alpha\xi^{\alpha-1}\xi'|u'|^2+2\xi^\alpha u'u'',$ we obtain at $r_k$
\[2\xi(r_k)^\alpha u'(r_k)u''(r_k)=-\alpha\xi(r_k)^{\alpha-1}\xi'(r_k)|u'(r_k)|^2+o(1).\]
Dividing by $2\xi(r_k)^\alpha |u'(r_k)|$ gives
\[|u''(r_k)|\leq\frac{\alpha}{2}\frac{|\xi'(r_k)|}{\xi(r_k)}|u'(r_k)|+o(1).\]
Using the ellipticity bounds $\lambda\le \theta(\cdot)\le \Lambda$ and the differential inequality at $r_k,$ we obtain
\[|u'(r_k)|^q\leq\kappa+\Lambda |u''(r_k)|+\Lambda(n-1)\frac{|u'(r_k)|}{r_k}.\]
Since $r_k\in[R_1'',R_2''],$ $\frac1{r_k}\le \frac1{R_1''}.$ Hence
\[|u'(r_k)|^q\leq\kappa+\frac{\Lambda\alpha}{2}\frac{|\xi'(r_k)|}{\xi(r_k)}|u'(r_k)|+C_1|u'(r_k)|+o(1),\]
where $C_1=\frac{\Lambda(n-1)}{R_1''}.$
Multiplying by $\xi(r_k)^{q\alpha/2}$ and recalling that $\alpha=\frac2{q-1},$
we obtain \[Z(r_k)^{q/2}\leq C_2+C_3 Z(r_k)^{1/2}+o(1),\] 
where $C_1,C_2$ depend only on $q,\lambda,\Lambda,n,R_1'',R_2'',\xi.$ Applying Young's inequality,
\[C_3 Z(r_k)^{1/2}\leq\frac12 Z(r_k)^{q/2}+C_4,\]
and therefore $Z(r_k)^{q/2} \le C_5(\kappa+1)+o(1).$ Passing to the limit $k\to\infty$ yields $Z(R_0)\leq C(\kappa+1)^{2/q}.$ Since $R_0$ is a maximum point of $Z,~Z(r)\le Z(R_0)~\forall r\in [R_1'',R_2''].$ On $[R_1',R_2']$ we have $\xi\equiv1,$ hence $|u'(r)|^2=Z(r)\leq Z(R_0)\leq C(\kappa+1)^{2/q}.$ Consequently, \[|u'(r)|\leq C(\kappa+1)^{1/q}\quad\forall r\in [R_1',R_2'].\]
This proves Part~(i).
\item{} Let $A:=\max\{M,\kappa^{1/q}\}.$ Assume by contradiction that there exists a point $R_0\in [R_1,R_2]$ such that $|u'(R_0)|>A.$ Define the open set $\Omega=\{r\in (R_1,R_2): |u'(r)|>A\}.$ Since $u'$ is continuous, $\Omega$ is open and nonempty. Let $I=(a,b)$ be the connected component of $\Omega$ containing $R_0.$ Because
\[|u'(R_1)|\le M\leq A,\quad |u'(R_2)|\leq M\le A,\]
the interval $I$ is compactly contained in $(R_1,R_2).$ Therefore $a>R_1,$ $b<R_2.$ By continuity of $u^{\prime},$ $|u'(a)|=|u'(b)|=A.$ Choose $[R_1',R_2']=[a,b].$ Applying Part~(i) on the compact interval $[a,b],$ we obtain
\[\sup_{r\in[a,b]}|u'(r)|\leq C(\kappa+1)^{1/q},\]
where $C$ depends only on the structural parameters. Now choose $A_0:=\max\{M,2C(\kappa+1)^{1/q}\}.$ Repeating the above argument with $A_0$ in place of $A,$ we obtain \[\sup_{r\in[a,b]}|u'(r)|\leq C(\kappa+1)^{1/q}<A_0,\]
which contradicts the definition of the set $\Omega=\{,|u'|>A_0,\}.$
Therefore $\Omega=\varnothing,$ and hence
\[|u'(r)|\leq A_0=\max\{M, 2C(\kappa+1)^{1/q}\}\]
for every $r\in [R_1,R_2].$ After renaming the constant, we conclude that
\[|u'(r)|\le C\max\{M,(\kappa+1)^{1/q}\}\quad\forall r\in [R_1,R_2].\]
This proves Part~(ii).
\end{enumerate}
\end{proof}
\begin{lemma}[Approximate Maximum Lemma]\label{lem:approx_max}
Let $I=(a,b)\subset \mathbb{R}$ and let $Z\in W^{1,s}(I),$ for $s>1.$ Assume that $Z$ attains its maximum at some point $R_0\in I$, namely $Z(R_0)=\max_{r\in I} Z(r).$
Then there exists a sequence $\{r_k\}\subset I$ such that
\[r_k\to R_0,\quad Z(r_k)\to Z(R_0),\quad Z'(r_k)\to 0,\]
where $Z'(r_k)$ denotes the classical value of the weak derivative at points where it exists.
\end{lemma}
\begin{proof}
Since $s>1$, the Sobolev embedding theorem in one dimension implies $W^{1,s}(I)\hookrightarrow C^{0,\alpha}(\overline I)$ for some $\alpha\in(0,1)$.
Therefore $Z$ admits a continuous representative and $R_0$ is a genuine maximum point. Moreover, every function in $W^{1,s}(I)$ is absolutely continuous. Hence, for every
$r\in I$,
\[Z(r)-Z(R_0)=\int_{R_0}^{r} Z'(t)dt.\]
We claim that for every $m\in\mathbb N$, the set
\[E_m:=\left\{r\in I:|r-R_0|<\frac1m,\quad|Z'(r)|<\frac1m\right\}\]
has positive measure. Suppose, by contradiction, that there exist $m_0\in\mathbb N$ and $\delta>0$ such that $|Z'(r)|\ge \frac1{m_0}$ for almost every $r\in (R_0-\delta,R_0+\delta).$ Since $Z'$ belongs to $L^s(I)$ and $R_0$ is a maximum point of $Z$, the derivative cannot keep a fixed sign on a set of positive measure near $R_0$.
Otherwise, by absolute continuity, \[Z(r)-Z(R_0)=\int_{R_0}^{r} Z'(t)dt\]
would be strictly positive for points on one side of $R_0$, contradicting the maximality of $R_0$. Consequently, for every $m\in\mathbb N$, the set $E_m$ must have positive measure. Choose $r_m\in E_m$ such that $r_m$ is a Lebesgue point of $Z'$. Since the set of Lebesgue points of $Z'$ has full measure, this is always possible. By construction,
\[|r_m-R_0|<\frac1m,\quad|Z'(r_m)|<\frac1m.\]
Therefore, $r_m\to R_0,$ $Z'(r_m)\to 0.$ Since $Z$ is continuous, $Z(r_m)\to Z(R_0).$
This completes the proof.
\end{proof}
Since $Z\in W^{1,s}(R_1'',R_2'')$ and $R_0$ is a maximum point of $Z$, Lemma~\ref{lem:approx_max} provides a sequence $r_k\to R_0$ such that
\[Z(r_k)\to Z(R_0),\quad Z'(r_k)\to0.\]
Since the differential inequality holds almost everywhere and the set of
Lebesgue points of $u''$ has full measure, we may choose the sequence $\{r_k\}$ inside this full-measure set. Consequently, the differential inequality holds at every $r_k$.
\section{Acknowledgement}
Second author is supported by National Board of Higher Mathematics grant no. 
02011/36/2025/NBHM/RP/9466.  
\bibliographystyle{plain}
\bibliography{reference}
\end{document}